\documentclass[11pt,letterpaper]{amsart}

\usepackage{constants}
\usepackage{verbatim}
\usepackage{hyperref}
\usepackage{amsmath}
\usepackage{enumerate, amsthm}
\usepackage{mathrsfs}
\usepackage{rotating}
\usepackage{xcolor}
\usepackage{mathtools}
\usepackage{tikz}
\usetikzlibrary{patterns}
\usetikzlibrary{arrows}
\usepackage{centernot}
\usepackage{url}

\newtheorem{theorem}{Theorem}[section]
\newtheorem{lemma}[theorem]{Lemma}
\newtheorem{lem}[theorem]{Lemma}
\newtheorem{corollary}[theorem]{Corollary}
\newtheorem{prop}[theorem]{Proposition}
\newtheorem{proposition}[theorem]{Proposition}

\theoremstyle{definition}
\newtheorem{defn}[theorem]{Definition}

\theoremstyle{remark}
\newtheorem{remark}[theorem]{Remark}

\numberwithin{equation}{section}

\newcommand{\Z}{\mathbb Z}
\newcommand{\E}{\mathbb E}
\renewcommand{\P}{\mathbb P}
\newcommand{\R}{\mathbb R}
\newcommand{\e}{\mathrm{e}}
\newcommand{\mb}{\mathcal}
\newcommand{\var}{\mathrm {var}}
\newcommand{\piv}{\mathrm{Piv}}
\newcommand{\G}{\mathcal G}
\newcommand{\N}{\mathbb N}
\newcommand{\Inf}{\mathrm{Inf}}
\renewcommand{\L}{\Lambda}

\newcommand{\ceil}[1]{\lceil #1 \rceil}
\newcommand{\floor}[1]{\lfloor #1 \rfloor}

\makeatletter
\newcommand{\lr}[4]{#3\xleftrightarrow[#1]{#2} #4}
\newcommand{\nlr}[4]{#3\mathrel{\mathop{\centernot\longleftrightarrow}_{#1}^{#2}} #4}
\newconstantfamily{c}{symbol=c}
\makeatother

\usepackage[utf8]{inputenc}
\usepackage[english]{babel}

\begin{document}

\title[Equality of critical parameters for GFF level-sets]{Equality of critical parameters for percolation of Gaussian free field level-sets}

\author[H.~Duminil-Copin]{Hugo Duminil-Copin}
\address{Institut des Hautes \'Etudes Scientifiques, 35, route de Chartres
 91440 -- Bures-sur-Yvette, France}
 \email{duminil@ihes.fr}
\address{Universit\'e de Gen\`eve,  Section de Math\'ematiques, 2-4, rue du Li\`evre, 1211 Gen\`eve 4, Switzerland}
\email{hugo.duminil@unige.ch}
\thanks{}

\author[S.~Goswami]{Subhajit Goswami}
\address{School of Mathematics, Tata Institute of Fundamental Research, 1, Homi Bhabha Road, Colaba, Mumbai 400005, India.}
\curraddr{}
\email{goswami@math.tifr.res.in}
\thanks{}

\author[P.-F.~Rodriguez]{Pierre-Fran\c{c}ois Rodriguez}
\address{Department of Mathematics, Imperial College London, London SW7 2AZ, United Kingdom}
\curraddr{}
\email{p.rodriguez@imperial.ac.uk}
\thanks{}

\author[F.~Severo]{Franco Severo} 
\address{Departement Mathematik, ETH Z\"urich, CH-8092 Z\"urich, Switzerland}
\curraddr{}
\email{franco.severo@math.ethz.ch}
\thanks{}



\date{\today}

\dedicatory{}

\begin{abstract}
We consider upper level-sets of the Gaussian free field on $\Z^d$, for $d\geq 3$, above a given real-valued height parameter $h$. As $h$ varies, this defines a canonical percolation model with strong, algebraically decaying correlations. We prove that three natural critical parameters associated with this model, respectively describing a well-ordered subcritical phase, the emergence of an infinite cluster, and the onset of a local uniqueness regime in the supercritical phase, actually 
coincide. At the core of our proof lies a new interpolation scheme aimed at integrating out the long-range dependence of the Gaussian free field. Due to the strength of correlations, its successful implementation requires that we work in an effectively critical regime. Our analysis relies extensively on certain novel renormalization techniques that bring into play all relevant scales simultaneously. The approach in this article paves the way to a complete understanding of the off-critical phases for strongly correlated disordered systems.
\end{abstract}

\maketitle

\section{Introduction}
\label{Sec:intro}

\subsection{Motivation}
Percolation has been at the heart of statistical physics for more than sixty years. Its most studied representative is the so-called Bernoulli (independent) percolation model. While the understanding of its critical phase is still incomplete, its behaviour away from criticality, in the sub- and supercritical phases, has been characterized very precisely, see \cite{AizBar87, Men86, GriMar90}.
Motivated by field theory and random geometry considerations, a  whole new class of percolation models, emerging from disordered systems with long-range interactions, has been the object of intense study over the last two decades. A common feature of these models is the strength of the correlations between local observables, which exhibit power law decay like $
|x-y|^{-a}$ as $|x-y|\to \infty$ for a certain (small) exponent $a>0$. This slow, non-summable decay --- often a distinguishing feature of critical phases --- is present throughout the entire parameter range, thus making the study of such models very challenging.

A few cases in point are the following: i) random interlacements on $\Z^d$, $d \geq 3$, see \cite{MR2680403,MR3050507,MR2892408,MR2932978}, which describe the local limit of a random walk trace on $(\Z/N\Z)^d$ as $N \to \infty$ and which relate to various covering and fragmentation problems for random walks; cf.~for instance \@ \cite{MR2520124,MR2561432,MR2838338,MR3563197}; ii) loop-soup percolation \cite{LJ12, lejan2013, ChSa16, Lup16}; iii) the voter percolation model \cite{LS86, LeMa06, RV17}; iv) \nolinebreak level-set percolation of random fields, see \cite{MoSt83, Sa1, An16} and references therein (see also \cite{BG17, RVb, BM18, MV} and \cite{NaSo09, CaSa19, SaWi19}) for Gaussian ensembles relating to various classes of functions, e.g. randomized spherical harmonics (Laplace eigenfunctions) at high frequencies; v) the massless Gaussian free field $\varphi$ on $\Z^d$ for $d \geq 3$. This last model, which will be the focus of the present article, was originally investigated by Lebowitz and Saleur in \cite{LS86} as a canonical percolation model with slow, algebraic decay of correlations. It has received considerable attention since then; see for instance \cite{MR914444,MR2080601,RodriguezSznitman13,MR3417515,DrePreRod,chiarini2018entropic,drewitz2018geometry, sznitman2018macroscopic, AC19.1, AC19.2}, and references below.

Most of these models have a different behaviour than Bernoulli percolation at criticality. Even their off-critical phases represent a challenge for mathematical physicists and probabilists, since their constructions involve correlations between vertices with slow algebraic decay. 
A persistent and fundamental question in this context is to assess whether several natural critical parameters (see Section~\ref{subsec:main}), defining regimes in which renormalization techniques lead to a deep understanding of the model, actually coincide; see \cite[Remark 2.8,1]{RodriguezSznitman13} and \cite[Remark 2.9]{MR3390739}, respectively regarding the sub- and supercritical phases of Gaussian free field level-sets, see also \cite[Remark 4.4,3)]{MR2680403}, \cite[(0.7)]{MR2561432} and \cite{Te11a} for similar questions concerning the vacant set of random interlacements. In the present work, we answer this question affirmatively for the historical example of Gaussian free field level sets; see Theorem \ref{thm:hh1} below. 
To the best of our knowledge, this is the first instance of a \textit{unified} approach towards the understanding of both sub- and supercritical regimes of percolation models.

The core of our proof is a new and delicate interpolation scheme aimed  at removing the long-range (algebraic) dependences intrinsic to the model. This scheme will work in a regime, expected to reduce to criticality, in which connection and disconnection probabilities decay slowly. Under the (a posteriori wrong) assumption that the critical parameters mentioned above do not coincide, this regime fictitiously extends ``away from'' criticality. As a consequence, our interpolation scheme implies the existence of a percolation model with finite-range dependence for which the corresponding critical parameters do not coincide either. This leads to a contradiction thanks to the result of \cite{GriMar90} combined with recent progress in the study of such models \cite{DumRaoTas17c,DumRaoTas17b}. We refer to  \cite{AizGri91}, \cite{pcnontriv18}, for interpolation schemes of a similar flavor, yet within simpler frameworks valid in perturbative regimes, and to \cite{S21} for sharpness results regarding a class of smooth short-range Gaussian fields obtained by interpolation methods. 

In the present context, a  ``bridging lemma'' for the Gaussian free field (see Lemma \ref{lem:bridging} below) will play a central role in allowing for various path reconstructions. Its derivation is obtained by expanding on renormalization ideas from \cite{MR2891880, drewitz2018geometry}, involving \textit{all} renormalized scales at the same time, see for instance Fig.~\ref{F:bridge}. We regard this step as a key progress in the understanding of percolation models that do not enjoy a so-called (uniform) finite-energy property. 

The blend of these interpolation and renormalization ideas represents the most innovative part of this article and we expect the underlying techniques to have great potential for future applications in the study of long-range models of the above class.

\medskip
 In combination with previous results in the literature, our findings have many implications regarding our understanding of the level-set geometry of $\varphi$, both in the subcritical and supercritical regimes. We defer a thorough discussion of these matters  for a few lines and first describe our results.

\subsection{Main result}\label{subsec:main}
We consider the massless Gaussian free field (GFF) on $\Z^d$, for $d \geq 3$, which is a centered, real-valued Gaussian field $\varphi = \{\varphi_x: x \in \Z^d\}$. Its canonical law $\P$ is uniquely determined by specifying that $\varphi$ has covariance function $\E[\varphi_x \varphi_y] = g(x, y)$ for all $x, y \in \Z^d$, where 
\begin{equation}
\label{eq:Green}
g(x, y) \stackrel{\textnormal{def.}}{=} \sum_{k = 0}^\infty P_x\left[X_k = y\right]\,, \ x, y \in \Z^d,
\end{equation}
denotes the Green function of the simple random walk on $\Z^d$. Here, $P_x$ stands for the canonical law of the discrete-time random walk $\{X_k: k \geq 0\}$ on $\Z^d$ with starting point $X_0 = x \in \Z^d$.  
For $h \in \R$, we introduce the level-set above height $h$ as $\{\varphi \geq h\} \stackrel{\textnormal{def.}}{=} \{x \in \Z^d: \varphi_x \geq h\}$ and for any $A,\, B,\, C \subset \Z^d$, let
\begin{equation}
\label{eq:intro_connectionevent}
\{\lr{C}{\varphi \geq h}{A}{B} \}\stackrel{\text{def.}}{=} \big\{A\text{ and }B\text{ are connected in }\{\varphi \geq h\}\cap C\big\}\,.
\end{equation}
The subscript $C$ is omitted when $C=\Z^d$. We also write $\{\lr{}{ \varphi \geq h}{A}{\infty} \}$ for the event that there is an infinite connected component (connected components will also be called {\em clusters}) of $\{\varphi \geq h\}$ intersecting $A$. Note that all previous events are decreasing in $h$. We then define the \emph{critical parameter} $h_*$ of $\{\varphi \geq h\}$ as 
\begin{equation}
\label{eq:h_*}
 h_*(d) \stackrel{\textnormal{def.}}{=} \inf \big\{h \in \R : \;  \P[\lr{}{\varphi \geq h}{0}{\infty}] = 0 \big\}.
\end{equation}

It is known that $0< h_*(d)<\infty$ for all $ d \geq 3$; see \cite{MR914444,RodriguezSznitman13,DrePreRod}. Thus, in particular, the level sets $\{\varphi \geq h\}$ undergo a (nontrivial) percolation phase transition as $h$ varies. Moreover, for all $h<h_*$, $\{\varphi \geq h\}$ has $\P$-a.s.~a unique infinite cluster, whereas for all $h>h_*$, $\{\varphi \geq h\}$ consists a.s.~of finite clusters only. 

Following \cite{RodriguezSznitman13}, we consider an auxiliary critical value $h_{**} \geq h_*$ defined as
\begin{equation}
\label{eq:h_**}
h_{**}(d) \stackrel{\textnormal{def.}}{=} \inf \big\{h \in \R : \inf_{R}\P[\lr{}{\varphi\geq h}{B_R}{\partial B_{2R}} ] = 0 \big\}\,,
\end{equation}
where $B_R\stackrel{\textnormal{def.}}{=}([-R, R]\cap \Z)^d$ and $\partial B_R\stackrel{\textnormal{def.}}{=} \{ y \in  B_R: y \sim z \text{ for some } z\in \Z^d\setminus B_R\}$ stand for the $\ell^{\infty}$-ball of radius $R$ centered at $0$ and its inner boundary, respectively. Here $ y \sim z$ means $y$ and $z$ are nearest-neighbors in $\Z^d$.
The quantity $h_{**}$ is well-suited for certain renormalization arguments in the subcritical phase, by which it was shown in \cite{RodriguezSznitman13} that $h_{**}(d)$ is finite for all $d \geq 3$ and that for all $h > h_{**}$, the level-set $\{\varphi \geq h\}$ is in a \textit{strongly non-percolative regime} 
in the sense that probabilities of connections decay very fast. More precisely, for any $h>h_{**}$ there exist constants $c, \rho \in (0,\infty)$ depending on $d$ and $h$, such that
\begin{equation}\label{eq:h_**2}
\P[\lr{}{\varphi\geq h}{0}{\partial B_R}]\leq e^{-cR^\rho}.
\end{equation}
In fact, one can even take $\rho=1$ for $d \geq 4$, with logarithmic corrections when $d=3$; see \cite{PopovRath15,MR3420516}. The arguments of \cite{RodriguezSznitman13} originally required the probability on the right-hand side of \eqref{eq:h_**} to decay polynomially in $R$ (along subsequences). It was later shown in \cite{PopovRath15,MR3420516} that it suffices for the infimum in \eqref{eq:h_**} to lie below the value $\frac{7}{2d\cdot 21^d}$ in order to guarantee (stretched) exponential decay of the probability that $B_R$ is connected to $\partial B_{2R}$ in $\{\varphi\ge h\}$ for all larger values of $h$, implying in particular the equivalence of these definitions with the one in \eqref{eq:h_**}, which is natural in the present context. It was further shown in \cite{MR3339867} that $h_*(d)\sim h_{**}(d)\sim \sqrt{2 \log d}$ as $d\to \infty$, where $\sim$ means that the ratio of the two quantities on either side converges to $1$, but little was otherwise known prior to this article about the relationship between $h_*$ and $h_{**}$.

Another critical parameter $\bar{h}\leq h_*$ was introduced in \cite{drewitz2018geometry}, inspired by similar quantities defined in \cite{MR3390739,MR3417515}, cf.~also \cite{AntPis96}, which allows to implement certain (static) renormalization arguments in the supercritical phase. 
As a consequence, the geometry of the level-sets $\{ \varphi \geq h\}$ is well-understood at levels $h< \bar h$, as will be explained further below. To define $\bar h$, we first introduce the events, for $\alpha\in\R$,
\begin{align}
&\text{Exist}(R,\alpha)\stackrel{\textnormal{def.}}{=}\left\{\begin{array}{c}  \text{there exists a connected component in }\\ \text{$\{\varphi \geq \alpha \}\cap B_R$ with diameter at least  $R/5$}\end{array}\right\}, \label{eq:EXIST}
\end{align}
and
\begin{align}
&\text{Unique}(R,\alpha)\stackrel{\textnormal{def.}}{=} \left\{\begin{array}{c}\text{any two clusters in $\{\varphi\geq\alpha\}\cap B_{R}$}\\\text{having diameter at least $R/10$ are}\\ \text{connected to each other in $\{\varphi\geq \alpha\}\cap B_{2R}$ } \end{array}\right\}, \label{eq:UNIQUE1}
\end{align}
(throughout the article, the diameter of a set is with respect to the sup-norm). We say that \textit{$\varphi$ strongly percolates up to level $h\in\mathbb{R}$} if there are constants $c,\rho\in (0,\infty)$, possibly depending on $d$ and $h$, such that for all $\alpha \leq h$ and $R \geq 1$,
\begin{align}
\label{eq:barh1}
&\P[\text{Exist}(R,\alpha)]\ge 1-{\e}^{-cR^\rho}\text{ and  }\P\left[\text{Unique}(R,\alpha)\right]\geq 1- {\e}^{-cR^\rho}.\end{align}
We then define
\begin{equation}
\label{eq:barh3}
\bar{h}(d) \stackrel{\textnormal{def.}}{=}  \sup\big\{h\in \R:~\text{$\varphi$ strongly percolates up to level $h$} \big\}.
\end{equation}
It was proved in \cite{MR3390739} that $\bar h~ (\leq h_*)$ is non-trivial, i.e. $\bar h  > -\infty$, and it was recently shown that $\bar h(d)>0$; see \cite{drewitz2018geometry}, which implies in particular that the sign clusters of $\varphi$ percolate. It is easy to see from \eqref{eq:barh1} that $\{\varphi \geq h\}$ has a unique infinite connected component for any $h < \bar h$, and one can show that any finite connected component $\{\varphi \geq h\}$ is necessarily tiny: for instance, the radius of a finite cluster has stretched exponential tails for $h< \bar h$. Let us briefly mention that different notions of $\bar{h}$ have been introduced in the literature, e.g.~\cite{MR3417515,drewitz2018geometry,sznitmanC1_2019}. We chose to consider here the strongest of these notions (resembling the one from \cite{sznitmanC1_2019}) so that our main result directly holds for the other ones as well. 

\bigskip
With $h_*$, $h_{**}$ and $\bar h$ given by \eqref{eq:h_*}, \eqref{eq:h_**} and \eqref{eq:barh3}, our main result is
\smallskip
\begin{theorem}
\label{thm:hh1}
For all $d \geq { 3}$, $\bar{h}(d) =  {h}_*(d) = h_{**}(d)$.
\end{theorem}
\smallskip
The following is an important consequence of Theorem \ref{thm:hh1}.
\smallskip
\begin{corollary}[Decay of the truncated two-point function except at criticality]
\label{cor:main} 
 For all $d \geq 3$ and $\varepsilon >0$, there exist $c=c(d,\varepsilon)\in (0,\infty)$ and $\rho=\rho(d) \in (0,\infty)$ such that for all $h\notin (h_*-\varepsilon,h_*+\varepsilon)$ and $x, y\in\Z^d$,
\begin{equation}
\label{eq:trunc2pt_decay}
\tau_h(x,y)\stackrel{\textnormal{def.}}{=} \P[\lr{}{\varphi\geq h}{x}{y}, \, \nlr{}{\varphi\geq h}{x}{\infty}]\leq \e^{-c|x-y|^\rho}.
\end{equation}
\end{corollary}
For $h > h_{**}~(=h_*)$, \eqref{eq:trunc2pt_decay} follows immediately from \eqref{eq:h_**2}. In fact, as mentioned above, one knows in this case that $\rho(d)=1$ whenever $d \geq 4$, with logarithmic corrections in dimension $3$; see \cite{PopovRath15,MR3420516}. For $h< \bar h~(=h_*)$, the bound \eqref{eq:trunc2pt_decay} follows from \eqref{eq:barh1} and a straightforward union bound. Indeed, this can be seen as follows: first note that the event $\bigcap_{n \geq 0} \text{Exist}(2^nR,h) \cap \text{Unique}(2^nR,h) $ implies $ A_R= \{\lr{}{\varphi\geq h}{B_R}{\infty} \}$, hence the latter has probability exceeding $1-{\e}^{-cR^\rho}$ for $h< \bar h$. 
Now, since $\tau_h(x,y)= \tau_h(0,y-x)$, choosing $R= |x-y|$ for $x \in \Z^d$, \eqref{eq:trunc2pt_decay} follows since the joint occurrence of $A_R$ and the event defining $\tau_h(0,y-x)$ imply that $\text{Unique}(2R,h)^c$ occurs. 

Moreover, the uniformity over $h<\bar{h}~(=h_*)$ for $\rho$ in \eqref{eq:barh1} (and therefore in \eqref{eq:trunc2pt_decay}) is a consequence of our proof, see Remark~\ref{R:exponent}.  
The optimal value of $\rho$ (if existing at all) for $h< h_*$, in both \eqref{eq:barh1} and \eqref{eq:trunc2pt_decay}, remains an open problem. 

To the best of our knowledge, the only instances in which a full (i.e.~adressing both sub- and supercritical regimes) analogue of Theorem~\ref{thm:hh1} and Corollary~\ref{cor:main} is known to hold, in all dimensions greater or equal to three, are the random cluster representation of the Ising model \cite{AizBarFer87, DumTas15, Bod05} and the aforementioned case of Bernoulli percolation \cite{AizBar87, Men86, GriMar90}. In particular, the analogue of $\bar h=h_*$ for the random cluster model with generic parameter $q \geq 1$ remains open.

\textit{Note added in the proof:} the questions surrounding the behavior of the two-point function have witnessed considerable progress recently. The precise speed of decay of \eqref{eq:trunc2pt_decay} for all $h \neq h_*$ has been identified in \cite{GRS20}. As it turns out, one has $\rho(d)=1$ for all $h \neq h_*$, but \eqref{eq:trunc2pt_decay} decays sub-exponentially as $R/ \log R$ with $R= |x-y| \to \infty$ in dimension three. We refer the reader to \cite{GRS20} for further details, including more precise leading asymptotics of $\tau_h$ for $d=3$. Furthermore, still for $d=3$ the full scaling behavior of $\tau_h(x,y)$ (quantitative in $|h-h_*|$ and $|x-y|$) was recently shown in \cite{DPR21} for a related bond percolation model corresponding to the excursion sets of the GFF on the corresponding cable graph.

\medskip
We now discuss further consequences of Theorem \ref{thm:hh1}. Various geometric properties of the (unique) infinite cluster $\mathscr{C}_\infty^h$ of $\{\varphi \geq h \}$ have been investigated in the regime $h < \bar h$, all exhibiting the ``well-behavedness'' of this phase. For instance, for $h< \bar h$, the chemical (i.e. intrinsic) distance $\rho$ on $\mathscr{C}_\infty^h$ is comparable to the Euclidean one, and balls in the metric $\rho$ rescale to a deterministic shape \cite{MR3390739}. Moreover, the random walk on $\mathscr{C}_\infty^h$ is known to satisfy a quenched invariance principle \cite{MR3568036} and  mesoscopic balls in $\mathscr{C}_\infty^h$ have been verified to exhibit regular volume growth and to satisfy a weak Poincar\'e inequality; see \cite{MR3650417}. This condition, originally due to \cite{Bar04}, has several important consequences, e.g.~it implies quenched Gaussian bounds on the heat kernel of the random walk on $\mathscr{C}_\infty^h$, as well as elliptic and parabolic Harnack inequalities, among other things. It has also been proved that the percolation function giving for each $h$ the probability that 0 is connected to infinity in $\{\varphi\ge h\}$ is $C^1$ on $(-\infty,\bar{h})$, and even analytic; see \cite{sznitmanC1_2019,panagiotis2021analyticity}. On account of Theorem~\ref{thm:hh1}, all the above results now hold in the entire supercritical regime $h<h_*$.  

The large-deviation problem of disconnection in the supercritical regime 
 has also received considerable attention recently; see \cite{MR3417515,nitzschner2017solidification,nitzschner2018,chiarini2018entropic}. Together with Theorems 2.1 and \nolinebreak5.5 of \cite{MR3417515} and Theorems 2.1 and 3.1 of \cite{nitzschner2018} (relying on techniques developed in \cite{nitzschner2017solidification}), Theorem \nolinebreak\ref{thm:hh1} yields the following: for $A \subset [-1,1]^d$ an arbitrary (not necessarily convex) regular compact set $A$ (regular in the sense that $A$ and its interior have the same Brownian capacity), one has
\begin{multline}
\label{eq:disco1}
\lim_N \frac{1}{N^{d-2}} \log \P[\nlr{}{\varphi\ge h}{(NA) \cap \Z^d}{\partial B_{2N}}] \\=-\frac1{2d}(h_*-h)^2 \text{cap}\left(A\right), \text{ for all $h< h_*$},
\end{multline}
where $\text{cap}(\cdot)$ stands for the Brownian capacity; see also \cite{chiarini2018entropic} for finer results on the measure $\P$ conditioned on the disconnection event above.

Theorem \ref{thm:hh1} also translates to a finitary setting: consider the zero-average Gaussian free field $\Psi$ on the torus $(\Z/N\Z)^d$ as $N\to \infty$; see \cite{abacherli2018local} for relevant definitions. As a consequence of Theorems 3.2 and 3.3 therein and Theorem \ref{thm:hh1} above, one deduces the following with high probability as $N \to \infty$: $\{ \Psi \geq h\}$ only contains connected components of size $o(\log^{\lambda} N)$ for any $h> h_{**}$ and $\lambda >d$, while $\{ \Psi \geq h\}$ has a giant, i.e.~of diameter comparable to $N$, connected component for all $h<h_*$. Plausibly, one could further strengthen these results and determine the size of the second largest component of $\{ \Psi \geq h\}$ for all $h<h_*$. 

Finally, we briefly discuss the \textit{massive} case, in which $g(\cdot,\cdot)$ in \eqref{eq:Green} is replaced by the Green function of the random walk killed with probability $\theta > 0$ at every step (whence correlations for $\varphi$ exhibit exponential decay). Let $h_*(\theta)$, $h_{**}(\theta)$ and $\bar{h}(\theta)$ denote the corresponding critical parameters. The techniques we develop here readily apply to prove that  $\bar{h}(\theta)=h_*(\theta)=h_{**}(\theta)$. Actually, the equality $h_* (\theta)= h_{**}(\theta)$, for all $\theta > 0$, can be obtained in a simpler fashion: one can apply Lemma \nolinebreak3.2 from \cite{DumRaoTas17b} directly to the law of $\{ \varphi \geq h\}$ (which is monotonic in the sense of \cite{DumRaoTas17b}) and combine it with Proposition 3.2 in \cite{MR3719059} to deduce a suitable differential inequality for the one-arm crossing probability. By current methods, the proof of $h_*(\theta)=\bar h(\theta)$ does however require a truncation (as for the case $\theta=0$, see below). One immediate difficulty that arises when attempting to extend the method for subcritical sharpness directly to the case $\theta=0$ is to obtain useful bounds from below for derivatives in $h$ in terms of so-called influences (as in Proposition 3.2 of \cite{MR3719059}). In a loose sense, this is what we achieve in Lemma~\ref{lem:piv_decoupling} below, in a certain (a posteriori fictitious) regime of parameters. The difficulty in deriving such bounds is, yet again, a manifestation of the long-range effects in the massless case.

\subsection{Key strategy: interpolation and renormalization}

We now give an overview of our proof of Theorem \ref{thm:hh1}. In doing so, we gather several stand-alone results that encapsulate the interpolation and renormalization ideas alluded to above, see in particular Propositions~\ref{prop:comparison},~\ref{prop:supercritical} and Lemma~\ref{lem:bridging}. 

 We start by introducing an additional critical parameter $\tilde{h}$ for $\varphi$ which quantifies how small disconnection probabilities are. Formally, let $u(R)\stackrel{\textnormal{def.}}{=}\exp[(\log R)^{1/3}] (\ll R)$ and define
\begin{equation}
\label{eq:tildeh}
\tilde{h}(d) \stackrel{\textnormal{def.}}{=} \sup\{h \in \R : \inf_{R} R^{d} \,\P[\nlr{}{\varphi\geq h}{B_{u(R)}}{\partial B_{R}}] =0\}.
\end{equation}
By \eqref{eq:h_**2} one knows that $\lim_R R^d\P[\lr{}{\varphi\geq h}{B_{u(R)}}{\partial B_{R}}]  =0$ whenever $h>h_{**}$. In view of \eqref{eq:tildeh}, this readily implies that $\tilde{h} \leq h_{**}$. Several reasons motivate the choice of the scale $u(R)$, one of them being the precise form of a certain ``reconstruction cost'' appearing in Lemma \ref{lem:bridging} below (see \eqref{eq:sprinklingintro}). We refer to Remarks~\ref{R:u_1} and \ref{R:u_2} for details on the choice of $u(\cdot)$. 

Our proof is organized in three parts, corresponding to Propositions \ref{prop:sharptruncated}, \ref{prop:comparison} and \ref{prop:supercritical} below: the first two will imply the equality $\tilde{h} =h_{**}$, while the last one will relate $\tilde h$ to $\bar h$.

We first decompose the GFF into an infinite sum of independent stationary Gaussian fields $(\xi^\ell)_{\ell\geq0}$ (see Section \nolinebreak\ref{sec:decompose_GFF} for precise definitions)
with each $\xi^\ell$ having finite range of dependence (in fact, the range of dependence will be exactly $\ell$), and define a \textit{truncated field}
\begin{equation}
\label{eq:intro_phi_L}
\varphi^L \stackrel{\textnormal{def.}}{=} \sum_{0 \leq \ell \leq L } \xi^{\ell}\,.
\end{equation}
The percolation processes $\{\varphi^L \geq h\}$ are natural finite-range approximations for 
$\{\varphi \geq h\}$. Instead of working directly with those, it turns out to be technically more convenient to use slightly {\em noised} versions of these approximations, for which a certain finite-energy property plainly holds (along with finite range, this property is crucially needed to deduce the first equality in Proposition \nolinebreak\ref{prop:sharptruncated} below in a straightforward way). To this end, we introduce for any $\delta \in (0, 1)$ and any
percolation configuration $\omega \in \{0,1\}^{\Z^d} $, a new configuration $T_\delta \omega$ where the state of every vertex is 
resampled independently with probability $\delta$, according to (say) a uniform distribution on $\{0,1\}$ (any non-degenerate distribution on $\{0,1\}$ would do). One can now define the critical parameters $h_*(\delta, L)$, $h_{**}(\delta, L)$ and $\tilde{h}(\delta, L)$ as in \eqref{eq:h_*}, \eqref{eq:h_**} and \eqref{eq:tildeh}, but for the family of processes $\{T_\delta \{\varphi^L\ge h\}: h \in \R\}$ instead of $\{ \{\varphi \ge h\}: h \in \R\}$. 
The next proposition states a sharpness result for these finite-range models.
\begin{prop}
\label{prop:sharptruncated}
For all $d\geq3$, $L\geq0$ and $\delta \in (0, 1)$, we have that $\tilde{h}(\delta, L) = h_*(\delta, L) = 
h_{**}(\delta, L)$.
\end{prop}
This proposition is a fairly standard adaptation of known results (a result of Grimmett-Marstrand \cite{GriMar90} on one side, and proofs of sharpness using the OSSS-inequality developed in \cite{DumRaoTas17b, DumRaoTas17a, DumRaoTas17c} on the other side); see Section~\ref{sec:sharpness} for  details. Nonetheless, Proposition \ref{prop:sharptruncated} offers a stepping stone for our argument, which, roughly speaking, will consist of carrying over the sharpness for these finite-range models to a sharpness result for level-sets of the full GFF by comparing the two 
models for parameter values $h\in(\tilde{h},h_{**})$ (notice that this interval can a priori be empty). The core of our strategy is therefore encapsulated in the following proposition.
\begin{prop}
	\label{prop:comparison}
	For every $d\ge3$ and $\varepsilon>0$, there exist $c, C>0$, $\delta \in (0,1)$ and an integer $L \geq 1$, all depending on $d$ and $\varepsilon$ only,
	such that for all $h \in (\tilde{h} + 3\varepsilon, 
	h_{**} - 3\varepsilon)$ and $R \geq 2r>0$, 
\begin{align}
	 &\P[\lr{}{T_{\delta}\{\varphi^L\ge h\}}{B_{r}}{\partial B_{R}}] \geq \P[\lr{}{\varphi\geq h+\varepsilon}{B_{r}}{\partial B_{R}}]  - C\exp(-\e^{c(\log r)^{1/3}}),\, \label{eq:comparison1} \\
	 &\P[\lr{}{T_{\delta}\{\varphi^L\ge h\}}{B_{r}}{\partial B_{R}}] \leq \P[\lr{}{\varphi\geq h-\varepsilon}{B_{r}}{\partial B_{R}}] + C\exp(-\e^{c(\log r)^{1/3}}). \label{eq:comparison2}
\end{align}
\end{prop}

Proposition \ref{prop:comparison} is truly the heart of the paper. Note that eventually, we show that $\tilde h=h_{**}$, so that the interval $(\tilde h+3\varepsilon,h_{**}-3\varepsilon)$ corresponds to a {\em fictitious regime}, in the sense that the interval in question is in fact empty as a consequence of Theorem \ref{thm:hh1} (a similar fictitious regime was introduced in \cite{DumRaoTas17c} to study Boolean percolation). 

The proof of Proposition \ref{prop:comparison} is based on an interpolation argument, inspired to some extent by \cite{pcnontriv18} and more remotely by \cite{AizGri91}, enabling us to remove the long-range dependences of the full model at the cost of slightly {\em varying} the parameter $h$. More precisely, we will define a family of Gaussian fields $\chi^t$ indexed by $t\ge0$ satisfying the following properties: for each integer $n \geq 0$, the field $\chi^n$ will be equal to $\varphi^{L_n}$ for a certain integer $L_n$ (henceforth referred as the $n$-th scale, see \eqref{eq:bridge1} below) and $\chi^t$ will interpolate linearly between $\chi^{\floor t}$ and $\chi^{\ceil t}$. Then, we will show that the functions
$$f_{\pm}(t)\stackrel{\textnormal{def.}}{=}\theta(t,h\pm 2e^{-t},r,R)\mp C\exp(-\e^{c(\log r)^{1/3}})e^{-t},$$
where $\theta(t,h,r,R)\stackrel{\textnormal{def.}}{=}\P[\lr{}{\chi^t\ge h}{B_r}{\partial B_R}]$, are increasing and 
decreasing respectively. This will follow from a careful comparison of the partial derivatives $\partial_t \theta$ and $\partial_h\theta$. One important step in this comparison will be the (re-)construction of suitable ``pivotal points'' from corresponding coarse-grained ones, cf.~Fig.~\ref{F:reconstruction}, which will involve an instance of a ``bridging lemma'', akin to Lemma \ref{lem:bridging} below, in order to (re-)construct various pieces of paths in $\{ \varphi \geq h\}$ for $h<h_{**}$. The arguments involved in the derivative comparison will repeatedly rely on the assumption that various connection and disconnection events are not too unlikely, as guaranteed by the assumption that $h \in (\tilde{h} + 3\varepsilon, h_{**} - 3\varepsilon)$, cf.~\eqref{eq:h_**} and \eqref{eq:tildeh}. This motivates the introduction of such a (fictitious) regime. A more thorough discussion of the interpolation argument underlying the proof of Proposition \ref{prop:comparison} goes beyond the scope of this introduction and is postponed to Section \ref{sec5.1}. 

As a straightforward consequence of Propositions~\ref{prop:sharptruncated} and \ref{prop:comparison}, one deduces that $\tilde h= h_{**}$ for every $d \geq 3$ as 
follows. On account of the discussion immediately following \eqref{eq:tildeh}, it suffices to argue that $h_{**} \leq \tilde h$. Suppose on the contrary that the interval $(\tilde h, h_{**})$ is non-empty 
and consider $L$ and $\delta$ provided by Proposition~\ref{prop:comparison} with 
$\varepsilon \stackrel{\textnormal{def.}}{=} (h_{**} -  \tilde h) / 8$. It then follows by Proposition 
\ref{prop:comparison} that the intervals $(h_{**}(\delta, L), \infty)$ and $(\tilde 
h+3\varepsilon, h_{**}-3\varepsilon)$ have empty intersection. Indeed, otherwise one 
could pick an $h \in (h_{**}(\delta, L), \infty) \cap (\tilde h+3\varepsilon, 
h_{**}-3\varepsilon)$ and \eqref{eq:comparison1} would yield that $\inf_R 
\P[\lr{}{\varphi\geq h+\varepsilon}{B_{R/2}}{\partial B_{R}}]=0$, thus violating the 
fact that $h+\varepsilon < h_{**}$, cf.~\eqref{eq:h_**}. A similar reasoning using 
\eqref{eq:comparison2} yields that $(-\infty, \tilde h(\delta, L)) \cap (\tilde 
h+3\varepsilon, h_{**}-3\varepsilon) =\emptyset$. But both $(h_{**}(\delta, L), 
\infty)$ and $(-\infty, \tilde h(\delta, L))$ having empty intersection with $(\tilde 
h+3\varepsilon, h_{**}-3\varepsilon)$ contradicts the equality $\tilde h(\delta, L) = 
h_{**}(\delta, L)$, which is implied by Proposition~\ref{prop:sharptruncated}.

All in all, the discussion of the previous paragraph shows that Theorem~\ref{thm:hh1} follows immediately from the Propositions~\ref{prop:sharptruncated} and \ref{prop:comparison}, combined with the following one.
\begin{prop}
\label{prop:supercritical}
For all $d\geq 3$, $\tilde{h}(d) \leq \bar{h}(d)$.
\end{prop}

The proof of this proposition will be rather different from that of the previous proposition. Our starting point is a result of Benjamini and Tassion \cite{benjaminitassion17}, stating that in Bernoulli percolation, for every $\varepsilon>0$, the probability that a graph spanning the whole box $B_R$ does not become connected after opening every edge independently with probability $\varepsilon>0$ is extremely small provided that $R$ is sufficiently large. In the present case, for $h<\tilde h$, one sees from \eqref{eq:tildeh} that the probability that every box of size $u(R)$ in $B_R$ is connected to $\partial B_R$ can be taken arbitrarily close to $1$ provided that $R$ is chosen large enough. From this, we perform a coarse-grained version of the Benjamini-Tassion argument to prove that the probability of $\text{Unique}(R,\beta)$ converges to 1 (along subsequences) for all $\beta<h$; see Proposition \ref{prop:uniq}. Then, we bootstrap  this estimate via a renormalization argument to show that the probabilities of $\text{Unique}(R,\alpha)$ and $\text{Exist}(R,\alpha)$ tend to~1 stretched-exponentially fast for $\alpha<\beta$. 

Implementing this scheme will raise a number of difficulties. First, the model has long-range dependence, a fact which forces us to use renormalization techniques, pioneered in \cite{MR2891880} (see also references therein) in the context of random interlacements, rather than elementary coarse-graining usually harvested in Bernoulli percolation. Second, the model does not enjoy uniform bounds on conditional probabilities that a vertex is in $\{\varphi\ge h\}$ or not. When conditioning on a portion of $\{\varphi\ge h\}\setminus\{x\}$, the stiffness of the field may force $\varphi_x\ge h$ or $\varphi_x<h$ in a very degenerate fashion. These {\em large-field effects} are difficult to avoid, as one can see for instance by observing that the probability that $\varphi_0\geq h$ conditioned on the event that $\varphi_x< h$ for every $x\in B_R\setminus\{0\}$ decays polynomially in $R$; see \cite{bolthausen1995, deuschel1999entropic}. This means that, when implementing a coarse-grained version of the argument from \cite{benjaminitassion17}, we will rely on yet another ``bridge construction'' to argue that decreasing $h$ by $\varepsilon$ indeed creates connections between nearby clusters.

We now describe more precisely a version of the ``bridging lemmas'' used in the proofs of both Propositions \ref{prop:comparison} and \ref{prop:supercritical}, which are needed in order to cope with the large-field effects of $\varphi$ alluded to above (a glance at Figure~\ref{F:bridge} in Section \ref{subsec:bridges1} 
might also help). For simplicity, we introduce an example of a useful statement asserting that it is still possible, outside of events of stretched-exponentially small probability to connect two large (connected) subsets $S_1$ and $S_2$ of $B_R$ at a reasonable cost, even when conditioning on $\varphi_x$ for  $x\notin B_R$ and on $1_{\varphi_x\ge h}$ for every $x\in S_1\cup S_2$. We refer to Lemma~\ref{lem:sprinkling} and Remark~\ref{remark:sprinkling} (see also \eqref{eq:yz}) below for results similar to Lemma~\ref{lem:bridging} but tailored to the proofs of Propositions \ref{prop:supercritical} and \ref{prop:comparison}, respectively. 

\begin{lemma}[Bridging lemma]\label{lem:bridging}
	For every $d\ge3 $ and $\varepsilon >0$, there are positive constants $c=c(d,\varepsilon), C=C(d,\varepsilon)$ and $\rho=\rho(d)$ such that the following holds for all $R\geq1$. There exist events 
	$\mathcal{G}(S_1,S_2)$ indexed by $S_1, S_2 \subset B_R$ satisfying
	\begin{equation}\label{eq:nicelikelyintro}
	\P \Big[\bigcap_{S_1,S_2} \mathcal{G}(S_1,S_2)\Big]\geq 1-e^{-cR^\rho}
	\end{equation}	
	with the additional property that, for all sets $S_1, S_2\subset B_R$ connected with diameter larger than $R/10$, all  $h<h_{**}-2\varepsilon$ and all events $D\in \sigma(1_{\varphi_x\geq h};~x\in S_1\cup S_2)$ and $E\in \sigma(\varphi_x;x\notin B_R)$,
	\begin{equation}
	\label{eq:sprinklingintro}
	\P\big[\lr{B_R}{\varphi\geq h-\varepsilon}{S_1}{S_2} ~\bigl\vert\,  D\cap E \cap \mathcal{G}(S_1,S_2) \,\big]\geq e^{-C(\log  R)^2}
	\end{equation}  
	 whenever $\P[D\cap E \cap \mathcal{G}(S_1,S_2)] > 0$.
\end{lemma}

Note that the assumption that $h<h_{**}$ is necessary since otherwise, already for the unconditioned measure the probability in \eqref{eq:sprinklingintro} to connect two sets $S_1$ and $S_2$ at a distance of order $R$ of each other is decaying stretched exponentially fast as soon as $h>h_{**}$, cf.~the discussion following~\eqref{eq:h_**}. 

We now explain the nature of the events $\mathcal{G}(S_1,S_2)$. The argument yielding \eqref{eq:sprinklingintro} will require (re-)constructing pieces of paths in $\{ \varphi \geq h\}$ for $h<h_{**}$ to connect $S_1$ and $S_2$ at an affordable cost. 
The paths in question will be built inside so-called \textit{good bridges}, introduced in Definitions \ref{def:bridge} and \ref{def:goodbridge}; see also Fig. \ref{F:bridge} below. Roughly speaking, a good bridge is formed by a concatenation of boxes at multiple scales $L_n$ ($n \geq 0$) defined by 
\begin{equation}
\label{eq:bridge1}
L_n \stackrel{\textnormal{def.}}{=} \ell_0^n L_0, \text{ for some } L_0 \geq 100, \ell_0 \geq 1000,
\end{equation}  
 in which $\varphi$ has certain desirable (good) properties. Together with the assumption that $h<h_{**}$, these conditions on $\varphi$ will allow to deduce the bound \eqref{eq:sprinklingintro}. Their precise form may however vary depending on the specific situation in which a bridge construction is applied. 

Apart from just connecting two sets of interest, bridges satisfy two important geometric constraints: i) any box at scale $L_k$ which is part of a bridge does not get closer than distance $\approx L_k$ to the two sets connected by the bridge, and ii) a bridge does not involve too many boxes at any scale $L_k$. The former will allow us to retain some independence when exploring the clusters that need to be connected while the latter is key in order to keep the reconstruction cost under control.

The events $\mathcal{G}(S_1,S_2)$ appearing in Lemma \ref{lem:bridging} then correspond to the existence of a good bridge linking the sets $S_1$ and $S_2$. Their likelihood, as implied by \eqref{eq:nicelikelyintro}, will follow from Theorem \nolinebreak\ref{T:bridge1}, derived in the next section. It asserts that good bridges (for a generic underlying notion of goodness, see e.g.~\eqref{eq:bridge10}--\eqref{eq:bridge11}) can be found with very high probability between any two sufficiently large sets. This result will then be applied in Sections \ref{sec:supercritical} and \ref{sec:comparison} with different choices of good events involving a decomposition of $\varphi$ into a sum of independent fields with range $L_n$ for $n \geq 0$, alluded to in \eqref{eq:intro_phi_L} and introduced in Section~\ref{sec:decompose_GFF}, to  yield \eqref{eq:nicelikelyintro} and \eqref{eq:sprinklingintro} as a surrogate ``finite-energy property'' for $\varphi$. The definition of bridges as well as the statement of Theorem~\ref{T:bridge1} are fairly technical and postponed to Section \ref{sec:bridges}. The proof of Theorem~\ref{T:bridge1} is based on renormalization ideas for $\varphi$ developed in \cite{RodriguezSznitman13, MR3390739, drewitz2018geometry}. Interestingly, and in contrast to these works, our main tool in the present context, introduced in the next section, is a geometric object (the good bridge) which involves \textit{all} scales $L_k$, $0\leq k \leq n$, for a given macroscopic scale $L_n$, cf.~Fig.~\ref{F:bridge} in Section \ref{subsec:bridges1} and Fig.~\ref{F:reconstruction} in Section \ref{sec5.2}.

\subsection*{Organization of the paper.}
Section~\ref{sec:bridges} contains the renormalization scheme and the notion of good bridges that will be used in several places later on. The statements and proofs have been made independent of the model. Section~\ref{sec:decompose_GFF} introduces the decomposition of the GFF into finite-range Gaussian processes and presents the proof of Lemma~\ref{lem:bridging}. Section~\ref{sec:supercritical} and Section~\ref{sec:comparison} are respectively devoted to the proofs of Propositions~\ref{prop:supercritical} and~\ref{prop:comparison}. The last section contains the proof of Proposition~\ref{prop:sharptruncated} and is independent of the rest of the paper.

\subsection*{Notation.} For $x\in\Z^d$, let $B_R(x)\stackrel{\textnormal{def.}}{=}x+B_R$ and $\partial B_R(x) \stackrel{\textnormal{def.}}{=}x+\partial B_R$, with $B_R$ and $\partial B_R$ as defined below \eqref{eq:h_**}. Except otherwise stated, distances are measured using the $\ell^\infty$-norm, which is denoted by $\vert\cdot\vert$. We use $d(U,V)$ to denote the $\ell^\infty$-distance between sets $U,V \subset \Z^d$.


We write $c,c',C,C'$ for generic numbers in $(0,\infty)$ which may change from line to line. They may depend implicitly on the dimension $d$. Their dependence on other parameters will always be explicit. Numbered constants $c_0, c_1,C_0,C_1,\dots$ refer to constants that are used repeatedly in the text; they are numbered according to their first appearance.

\section{Multi-scale bridges}\label{sec:bridges}
 
In this section, we introduce the notion of good (multi-scale) bridge which will be later used. The main result is Theorem \ref{T:bridge1}, which asserts that good bridges  connect any two ``admissible'' sets with very high probability when certain conditions are met. The proof of Theorem \ref{T:bridge1} appears in Section \ref{sec:bridge2.1}. It involves a suitable renormalization scheme, and revolves around Lemma~\ref{L:bridge1}, which is proved in a separate section (Section \ref{sec:bridge2.2}).

\subsection{Definition of a bridge and statement of Theorem~\ref{T:bridge1}}
\label{subsec:bridges1}

Recall the definition of scales $L_n$, $n\geq0$, from \eqref{eq:bridge1}. For $n\ge0$, let $\mathbb{L}_n \stackrel{\textnormal{def.}}{=} (2L_n+1)\Z^d$ and call a box of the form $B_{L_n}(x)$ for $x\in \mathbb L_n$ the \textit{$n$-box (attached to $x$)}. Note that for each $n \geq 0$, every point $y \in \Z^d$ is contained in exactly one $n$-box. We call a nearest-neighbor path in $\mathbb{L}_n$ any sequence of vertices in $\mathbb{L}_n$ such that any two consecutive elements are at $\ell^{1}$-distance $2L_n+1$ on~$\Z^d$. 
 
 We introduce two parameters
$\kappa \geq 20$ and $K \geq 100
$ which will respectively govern the ``separation scale'' and the ``complexity'' of a bridge, see \eqref{B3} and \eqref{B4} below. These parameters correspond to the geometric features i) and ii) highlighted in the introduction, see below \eqref{eq:bridge1}.

In what follows, for $n \geq 0$, we consider the triplet of domains $(\underline{\Lambda}_n,\Lambda_n,\Sigma_n)$ where $\Sigma_n \stackrel{\textnormal{def.}}{=} B_{ 9\kappa L_n } \setminus B_{\lfloor 8.5\kappa L_n\rfloor }$ and  
\begin{equation}
\label{eq:annuli}
(\Lambda_n,\underline{\Lambda}_n ) \stackrel{\textnormal{def.}}{=} (B_{10\kappa L_n}, B_{8 \kappa L_n })  \text{ or } ( B_{10\kappa L_n}\setminus B_{\kappa L_n} ,  B_{8 \kappa L_n }\setminus B_{\kappa L_n}). 
\end{equation}


\begin{defn}[Bridge]\label{def:bridge} For any $S_1,S_2 \subset \Lambda_n$, a bridge between $S_1$ and $S_2$ inside $\Sigma_n$ is a finite collection $\mathcal{B}$ of subsets of $\Sigma_n$ with the following properties:
\begin{align}
&\begin{array}{l}\text{every $B\in \mathcal{B}$ is an $m$-box, $0\leq m\leq n$, included in $\Sigma_n$}\\
\text{and $\bigcup_{B\in \mathcal{B}} B$ is a connected set;} \end{array}  \label{B1}  \tag{\bf B1}\\[0.3em]
&\begin{array}{l}\text{there exist $0$-boxes $B_1, B_2 \in \mathcal{B}$ such that $B_i \cap S_i \neq \emptyset$, $i=1,2$}\\
\text{and for all $B \in \mathcal{B} \setminus \{ B_1,B_2\}$, $B\cap (S_1\cup S_2)=\emptyset$;} \end{array}  \label{B2}  \tag{\bf  B2} \\[0.3em]
&\begin{array}{l}\text{for every $m$-box $B\in \mathcal{B}$ with $1\leq m\leq n$,}\\
\text{one has $d(B,S_1\cup S_2)\geq \kappa L_m$;} \end{array}  \label{B3}  \tag{\bf  B3} \\[0.3em]
&\begin{array}{l}\text{for every $m\ge0$, the number of $m$-boxes in $\mathcal{B}$ is at most $2K$. } \end{array}  \label{B4}  \tag{\bf  B4} 
\end{align}
\end{defn}
We now introduce ``good events'' which will be later chosen according to specific needs. For the remainder of this section, we simply consider, on some probability space $(\Omega, \mathcal{F}, \P)$, families of events $F=\{F_{0,x}: x \in \mathbb{L}_0 \}$ and $H_n = \{H_{n,x}: x \in \mathbb{L}_n \}$, for $n \geq 1$. 

\begin{defn}[Good bridge] \label{def:goodbridge} A bridge as defined above is \textit{good} if
\begin{align}
&\begin{array}{l}\text{for every $0$-box $B_{L_0}(x) \in \mathcal{B}$, the event $F_{0,x}$ occurs;}\end{array}  \label{G1}  \tag{\bf  G1} \\[0.3em]
&\begin{array}{l}\text{for every $m$-box $B_{L_m}(x) \in \mathcal{B}$, $0\leq m \leq n$, the events}\\
\text{$H_{j,\,y(x,j)}$ occur for every $j \geq m \vee 1$, where $y(x,j)$ is}\\\text{the unique element of $\mathbb L_j$ with $x \in B_{L_j}(y(x,j))$.} \end{array}  \label{G2}  \tag{\bf  G2} 
\end{align}
\end{defn}

The property \eqref{G2} ensures that every $m$-box in a good bridge 
sits inside a ``tower'' of good events attached to the $j$-boxes 
containing its center, for all $j \ge m$. Good bridges will be used to connect 
a certain class of sets. A set $S \subset \Lambda_n$ (not necessarily 
connected) is \textit{admissible} if each connected component of $S$ 
intersects $\partial B_{10\kappa L_n}$ and at least one connected 
component of $S$ intersects $B_{8\kappa L_n}$. We are interested in 
the event, for $n\ge0$,
\begin{multline}
\label{eq:bridge.GOOD}
\mathcal{G}_{n}\stackrel{\textnormal{def.}}{=}\{\text{there exists a good bridge }\\\text{inside $\Sigma_n$ between every admissible $ S_1, S_2 \subset \Lambda_n$}\},
\end{multline}
see also Figure \ref{F:bridge}. We define $\mathcal{G}_{n,x}$, $x\in \mathbb{L}_n$, as the event corresponding to \eqref{eq:bridge.GOOD} when one replaces the triplet $(\underline{\Lambda}_n,\Lambda_n,\Sigma_n)$ in \eqref{eq:annuli} and Definitions \ref{def:bridge} and \ref{def:goodbridge} by $(x+\underline{\Lambda}_n,x+\Lambda_n,x+\Sigma_n)$.
\begin{figure}[h!]
  \centering 
  \includegraphics[scale=0.7]{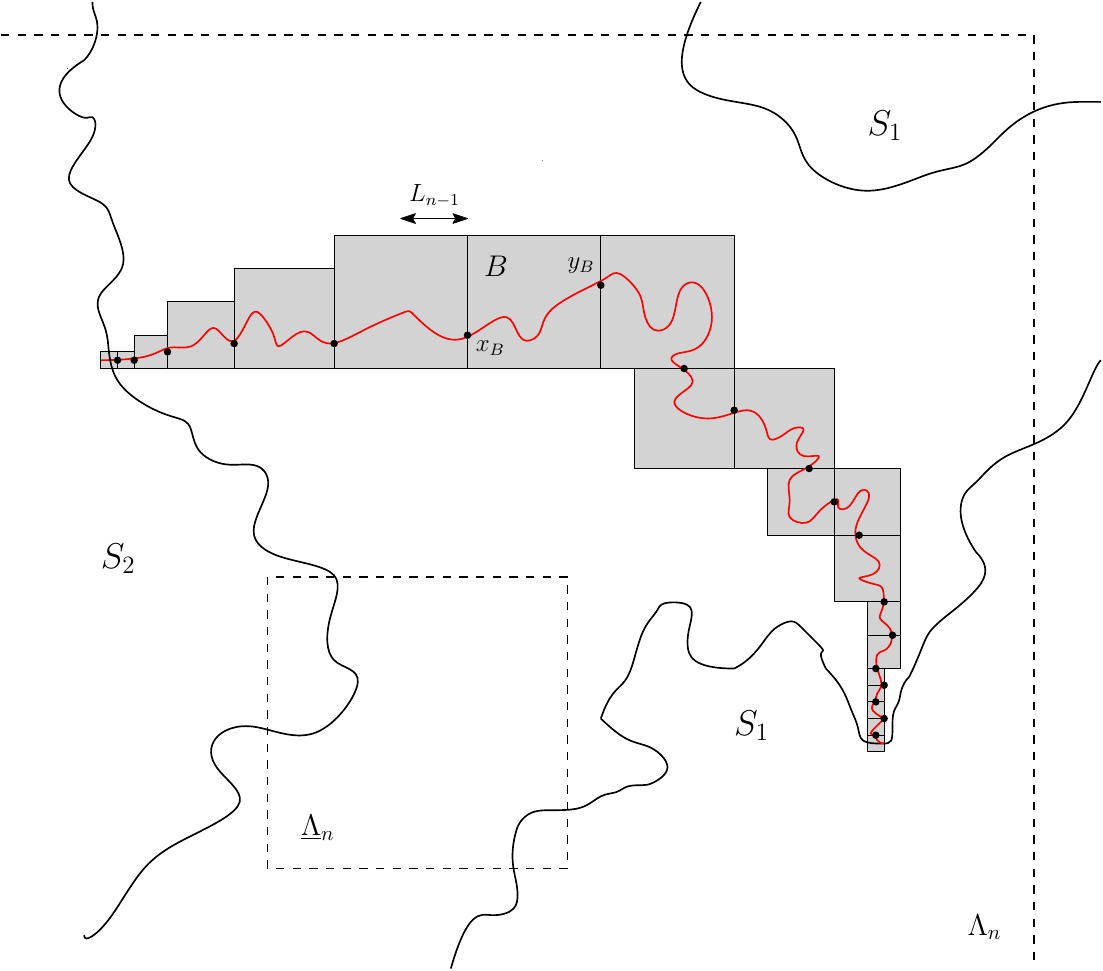}
  \caption{An illustration of the event $\mathcal{G}_n$: depicted is a pair of admissible sets $(S_1,S_2)$ and the good bridge (in light gray) connecting them. Later in Section~\ref{sec:decompose_GFF}, the underlying good events will guarantee that the sets $S_1$ and $S_2$ can be linked by a certain path (in red) inside a good bridge. Albeit not required by the definition, our construction of a good bridge on certain good events actually yields a ``croissant-type'' shape. More precisely, one can define two sequences of boxes, starting with the $0$-boxes $B_1$ and $B_2$ intersecting $S_1$ and $S_2$ (see \eqref{B2}), respectively, and corresponding to the two \textit{arches} in the proof of Lemma \ref{L:bridge1}, which comprise all but the largest boxes involved in the bridge construction and whose side lengths are non-decreasing. The largest boxes in the bridge have side length $L_{n-1}$ and comprise the \textit{deck}.}
  \label{F:bridge}
\end{figure}

Recall that the event $\mathcal{G}_{n}$ depends implicitly on the four parameters, namely $L_0,\ell_0,\kappa , K$, as well as on the choice of families $F$ and $H_n$, $n \geq 1$. In referring to $F$ (and similarly, $H_n$) as independent in the sequel, we mean that  $\{ 1_{F_{0,x}} : x \in \mathbb{L}_0\}$ is a collection of independent random variables. We now state the main result of this section.

\begin{theorem} \label{T:bridge1} For each $\kappa \geq 20$, $\ell_0 \geq C(\kappa)$, there exist $K=K(\kappa,\ell_0)$ and $\Cl[c]{bridges2}=\Cr{bridges2}(\kappa,\ell_0) \in (0,1)$ such that for all $L_0 \geq \Cl{bridges1}(\kappa,\ell_0)$ the following hold: if the families of events $F=\{F_{0,x}: x \in \mathbb{L}_0 \}$ and $H_n = \{H_{n,x}: x \in \mathbb{L}_n \}$ (for $n \geq 1$) satisfy
\begin{align}
&\text{ the families $F ,H_n, n \geq 1$ are independent,}  \label{C1}  \tag{\bf C1} \\[0.3em]
&\begin{array}{l}\text{for any set $U \subset \subset \Z^d$ such that $|y-z| \geq \frac{\kappa}{2}$ for all $y,z \in U$} \\ 
\text{with $y \neq z$, the events $F_{0,x}$, $x \in (2L_0+1)U$, are independent and} \\ \text{the events $H_{n,x}$, $x \in (2L_n+1)U$, are independent.} \end{array} \label{C2} \tag{\bf C2} \\[0.3em]
&\text{ for $x\in \mathbb L_0$, $\P\left[F_{0,x}^{\rm c}\right] \le c_1$ and for $n\ge1$, $x\in \mathbb L_n$,  $ \P[H_{n,x}^{\rm c}] \leq \Cr{bridges2}2^{-2^{n}}$},  \label{C3}  \tag{\bf C3}
\end{align}
then for every $n\ge0$ and $x \in \mathbb{L}_n$,
\begin{equation}
\label{eq:bridge8}
\P\left[\mathcal{G}_{n,x}\right]\geq 1- 2^{-2^n}.
\end{equation}
\end{theorem}

\begin{remark}\label{R:bridges} 1) Careful inspection of the proof below reveals that one could in fact replace \eqref{G1} and \eqref{G2} by the property that for any box $B=B_{L_n}(x)\in \mathcal{B}$, any $n\leq m \leq n_{\text{max}}$, where  
\begin{equation}
\label{eq:nmax}
n_{\text{max}}=n_{\rm max}(\Lambda)\stackrel{\textnormal{def.}}{=} \max\{k : B_{L_k}(x) \subset \Lambda \text{ for some $ x\in \mathbb{L}_k$} \},
\end{equation}
 and any $y \in \mathbb{L}_m$ such that $x \in B_{L_m}(y)$, the event $G_{m,y}(\Lambda_{n_{\text{max}}})$ occurs, together with the events $H_{m, y(x,m)}$ for $m > n_{\text{max}}$, and the conclusions of Theorem \ref{T:bridge1} continue to hold. 
\medbreak
\noindent 2) One could also replace the condition of admissibility of the sets $S_1,S_2 \subset \Lambda_n$ (cf.~above \eqref{eq:bridge.GOOD}) in Definition~\ref{def:goodbridge} by the requirement that $S_1$ and $S_2$ are any two connected subsets of $\underline{\Lambda}_n$ of diameter at least $\kappa L_n$ and connect them by a good bridge in ${\Lambda}_n$, i.e.~replacing $\Sigma_n$ by $\Lambda_n$ in Definition~\ref{def:bridge}; see Remark~\ref{R:0-bridge},1) below for the necessary small adjustments to the proof.
\end{remark}

\subsection{Proof of Theorem \ref{T:bridge1}}
\label{sec:bridge2.1}
The proof involves a multi-scale argument and a corresponding notion of ``goodness at level $n$'' for every $n \geq 0$, that we now introduce. For $\Lambda \subset \Z^d$ finite, let $\mathbb{L}_n(\Lambda) \stackrel{\textnormal{def.}}{=} \{ x \in \mathbb{L}_n : B_{L_n}(x) \cap \Lambda \neq \emptyset\}$ as well as for $x\in \mathbb L_0$,
\begin{equation}
\label{eq:bridge10}
G_{0,x}(\Lambda)\stackrel{\textnormal{def.}}{=} \bigcap_{y \in\mathbb{L}_0(\Lambda \cap B_{10\kappa L_0}(x)) } F_{0,y} \end{equation}
and for every $x\in \mathbb{L}_n$ with $n\ge1$, 
\begin{equation}
\label{eq:bridge11}
\begin{split}
G_{n,x}(\Lambda)\stackrel{\textnormal{def.}}{=} &\bigcap_{\substack{y,y' \in \mathbb{L}_{n-1}\left(\Lambda \cap B_{10\kappa L_n}(x)\right):\\ |y-y'|_{\infty} \geq 22\kappa L_{n-1}} }  (G_{n-1,y}(\Lambda) \cup G_{n-1,y'}(\Lambda)) \\
&  \cap \bigcap_{z \in \mathbb{L}_{n}\left(\Lambda \cap B_{10\kappa L_n}(x)\right)} H_{n,z}.
\end{split}
\end{equation}
A vertex $x \in \mathbb{L}_n$ will be called \textit{$n$-good} if the event $G_{n,x}(\Lambda_n)$ occurs with $\Lambda_n$ given by either choice of $(\Lambda_n,\underline\Lambda_n)$ in \eqref{eq:annuli}, and \textit{$n$-bad} otherwise. 
In words, $x$ is $n$-bad if either $H_{n,z}$ does not occur for some $z \in \mathbb{L}_{n}\left(\Lambda_n \cap B_{10\kappa L_n}(x)\right)$, or $\mathbb{L}_{n-1}\left(\Lambda_n \cap B_{10\kappa L_n}(x)\right)$ contains two distant vertices that are {\em both} $(n-1)$-bad.

The reason for introducing the notion of $n$-goodness is the following key deterministic lemma, which yields that $n$-goodness implies the occurrence of $\mathcal{G}_{n,0}$.
\begin{lemma}\label{L:bridge1} For all $\kappa \geq 20$, provided that $\ell_0 \geq C(\kappa)$ and $K \geq C' (\kappa, \ell_0)$, we have that for every $n\ge0$, $x \in \mathbb{L}_n$ and $L_0\ge100$, if the events $H_{m,y(x,m)}$, $m >n$, and $G_{n,x}(\Lambda_n)$ all occur, then so does $\mathcal{G}_{n,x}$.
\end{lemma}
We will prove Lemma \ref{L:bridge1} in the next section and now focus on the proof of Theorem~\ref{T:bridge1}.
 
 \begin{proof}[Proof of Theorem \ref{T:bridge1}]
 For simplicity, we assume that $x=0$. For a given $\kappa \geq 20$, consider $\ell_0,K,L_0$ such that the previous lemma holds true. Now, let 
\begin{equation}
\label{eq:bridge2.003}
q_n \stackrel{\textnormal{def.}}{=} \sup_{x\in\mathbb L_n} \P[ x\text{ is $n$-bad}\, ], \quad n \geq 0
\end{equation} 
and observe that by the previous lemma, it suffices to show, provided that $c_1$ such that \eqref{C3} holds is chosen small enough, that 
\begin{equation}\sum_{m > n} \P[H_{m,y(x,m)}^c]\leq  \frac12 2^{-2^n}\text{ and }q_n\le \frac12 2^{-2^n}
\end{equation} hold. The former is immediate by \eqref{C3}. 

With the choice that $\ell_0 \geq 22\kappa$, one deduces from \eqref{C1}--\eqref{C2} and the definition of goodness that the events $G_{n-1,y}(\Lambda_n)$ and $G_{n-1,y'}(\Lambda_n)$ are independent for any $|y-y'|_{\infty} \geq 22\kappa L_{n-1}$. Hence, the definition of goodness and the union bound yield that
\begin{equation}
\label{eq:recursion}
q_{n} \leq | \mathbb{L}_{n-1}(\Lambda_n)|^2q_{n-1}^2 + |\mathbb L_n(\Lambda_n)|\sup_x\P[H_{n,x}^{c}] 
\leq \tfrac14\Gamma^2 q_{n-1}^2+\tfrac12\Gamma c_12^{-2^n},
\end{equation}
for all $n \geq 1$, where in the second inequality we introduced $\Gamma= \Gamma(\kappa,\ell_0)\ge 2| \mathbb{L}_{n-1}\left(\Lambda_n\right)|$ (for all $n$), and we used \eqref{C3}. Since \eqref{C3} implies that $q_0\le \tfrac12\Gamma c_1$, a simple induction using \eqref{eq:recursion} implies that for every $n\ge1$,
$q_n\le c_1\Gamma 2^{-2^n}\le \frac12 2^{-2^n}$ as soon as $\tfrac14c_1\Gamma^3\le \frac12$ and $c_1\Gamma \leq \frac12$.\end{proof}

For later reference, we also collect the following consequence of the above setup.
For $\Sigma \subset \Lambda_n = B_{10\kappa L_n}$, define $\mathcal{B}^0$ to be a $0$-bridge inside $\Sigma$ between two sets $S_1,S_2\subset\Lambda_n $ if  $\mathcal{B}^0$ consists of $0$-boxes only and $\bigcup_{B \in \mathcal{B}^0} B$ is a connected subset of $\Sigma$ intersecting both $S_1$ and $S_2$, and call $\mathcal{B}^0$ good if for every $B=B_{L_0}(x) \in \mathcal{B}^0$, \eqref{G1} and \eqref{G2} occur. Let
\begin{multline}
\label{eq:bridge.GOOD_0}
\mathcal{G}_{n}^0 \stackrel{\textnormal{def.}}{=} \{\text{there is a good $0$-bridge inside $\widetilde{\Sigma}_n$}\\
\text{between any admissible sets $ S_1, S_2 \subset \Lambda_n$}\}
\end{multline}
where ``admissible'' can be either i) as defined above \eqref{eq:bridge.GOOD}, in which case one chooses $\widetilde{\Sigma}_n\stackrel{\textnormal{def.}}{=} {\Sigma}_n$, or ii) as defined in the previous paragraph with $\widetilde{\Sigma}_n\stackrel{\textnormal{def.}}{=} \Lambda_n$.
\begin{corollary}
\label{cor:0bridge}
Under the assumptions of Theorem \ref{T:bridge1}, 
$\P\left[\mathcal{G}_n^0\right]\geq 1-2^{-2^n}$, for all $n\geq 0$.\end{corollary}
We simply sketch the argument. Using Theorem \ref{T:bridge1}, one may obtain $\mathcal{B}^0$ from the good bridge $\mathcal{B}$ as follows: one replaces each box $B=B_{L_k}(x) \in \mathcal{B}$ for $k \geq 1$ by the set 
$$B^0\stackrel{\textnormal{def.}}{=} B \setminus  \Big( \bigcup_{0\leq k' < k} \ \ \bigcup_{y \in \mathbb{L}_{k'}(B): G_{k',y}(\Lambda_n)^{c} \text{ occurs}}B_{L_{k'}}(y) \Big)$$ 
and verifies that the $0$-boxes forming $B^0$ as $B$ ranges over all boxes in $\mathcal{B}$ contain a set $\mathcal{B}^0$ with the desired properties. We briefly sketch how to carry this out at the end of the proof of Lemma~\ref{L:bridge1}, cf.~Remark~\ref{R:0-bridge},2) below.

Alternatively, one can also prove Corollary \ref{cor:0bridge} directly, i.e.~without resorting to the existence of $\mathcal{B}$, by following the lines of the proof of Lemma~8.6 in \cite{drewitz2018geometry}, with suitable modifications (in particular, involving a different notion of $n$-goodness, due to the presence of $H_n$, $n\geq1$ in \eqref{eq:bridge11} replacing a sprinkling of the parameters to define the cascading events in (7.3) of \cite{drewitz2018geometry}).

\subsection{Proof of Lemma~\ref{L:bridge1}}
\label{sec:bridge2.2}

We now present the proof of Lemma \ref{L:bridge1}, which could be skipped at first reading. We insist on the fact that this proof is purely deterministic.
We distinguish two cases depending on whether $S_1$ and $S_2$ are close to each other, i.e.~at a distance at most $c(\kappa)L_{n-1}$, or not. The former case can be dealt with inductively (over $n$) and one can in fact create a good bridge involving $k$-boxes at levels $k \leq n-2 $ only, essentially by recreating the picture of $\mathcal{G}_n$ at level $n-1$ well inside $\Sigma_n$. The case where $S_1$ and $S_2$ are further apart requires more work. In this case, the good bridge is constructed by concatenating three pieces: a ``horizontal'' \textit{deck} and two \textit{arches} (the terms will be introduced in the course of the proof). Roughly speaking, the deck consists of good boxes at level $n-1$ only, which goes most of the distance between $S_1$ and $S_2$, leaving only two ``open'' ends. The ends are filled by two \textit{arches}  joining $S_1$ and $S_2$, respectively, to a nearby good $(n-1)$-box from the deck. The arches are constructed hierarchically and consist of boxes at lower levels, which, among other things, need to satisfy the conditions \eqref{B3} and \eqref{B4}. This requires a good deal of care. 

Throughout the proof, we set $\ell=22\kappa$ and assume for simplicity that $x=0$. Also, we introduce the notation $B_L(A)\stackrel{\textnormal{def.}}{=}\cup_{x\in A}B_L(x)$ for any subset $A$ of $\mathbb Z^d$. Recall that $d(\cdot,\cdot)$ refers to the $\ell^{\infty}$-distance between sets and let $\text{diam}(\cdot)$ denote to the $\ell^{\infty}$-diameter of a set. We first observe that, since the events $H_{m,0}$, $m > n$ occur by assumption in Lemma \ref{L:bridge1} it is  sufficient to build a good bridge as in Definition \ref{def:goodbridge} but with \eqref{G2} only required to hold for all $m$ satisfying $n \vee 1 \leq m \leq n_{\text{max}}$, rather than all $n \vee 1 \leq m$ (cf. \eqref{eq:nmax} for the definition of $n_{\text{max}}$ and note that $n_{\text{max}}(\Lambda)=n$ when $\Lambda =\Lambda_n$). 

\subsection*{Case 1:} {\em the $n=0$ case.}

\begin{proof} For admissible $S_1,S_2 \subset \Lambda_0$, consider a nearest-neighbor path $\gamma\subset\{ x \in \mathbb{L}_0: B_{L_0}(x) \subset \Sigma_0 \}$ of minimal length with starting point $y$ such that $B_{L_0}(y) \cap S_1 \neq \emptyset$ and endpoint $z$ such that $B_{L_0}(z) \cap S_2 \neq \emptyset$. The collection $\mathcal{B}= \{ B_{L_0}(x) : x \in \gamma \}$ plainly satisfies \eqref{B1}-\eqref{B3}, and \eqref{B4} holds for all $K \geq c \kappa^d$,
thus $\mathcal{B}$ is a bridge between $S_1$ and $S_2$ in $\Lambda_0$. Moreover if $0$ is $0$-good, i.e. $G_{0,0}(\Lambda_0)$ occurs, then by \eqref{eq:bridge10} $F_{0,x}$ occurs for each box $B_{L_0}(x) \in \mathcal{B}$, whence \eqref{G1} follows, and \eqref{G2} holds trivially since $n_{\text{max}}=0$. This concludes this case.\end{proof}

We now proceed by induction. From now on, we suppose that $n \geq 1$ and that the conclusion of the lemma is true for $n-1$ and consider any two admissible sets $S_1,S_2 \subset {\Lambda}_n$ (and assume for simplicity that $x=0$). Define  the random set 
\begin{equation}
\label{eq:bridge16}
\text{Bad}\stackrel{\textnormal{def.}}{=} \text{Fill}\Big( \bigcup_{x \in \mathbb{L}_{n-1}(\Lambda_n):\, G_{n-1,x}^{\,c}(\Lambda_n) \text{ occurs }} B_{L_{n-1}}(x)\Big)\,\subset\mathbb Z^d,
\end{equation} 
where for any $U \subset \Z^d$, $\text{Fill}(U)$ refers to the smallest set $V \supseteq U$ such that for every point $z \in \partial V$, there exists an unbounded nearest-neighbor path in $\Z^d \setminus V$ starting in $z$. Let
\begin{equation}
\label{eq:bridge16.01}
\begin{split}
&V_1 \stackrel{\textnormal{def.}}{=} \bigcup_{x: \,|x| = \lfloor8.5\kappa L_n\rfloor + 10\ell L_{n-1}} B_{\ell L_{n-1}}(x), \qquad V_1' \stackrel{\textnormal{def.}}{=} B_{\ell L_{n-1}}(V_1),\\
&V_2 \stackrel{\textnormal{def.}}{=} \bigcup_{x: \,|x| = \lfloor8.5\kappa L_n\rfloor + 20\ell L_{n-1}} B_{\ell L_{n-1}}(x), \qquad V_2' \stackrel{\textnormal{def.}}{=} B_{\ell L_{n-1}}(V_2).
\end{split}
\end{equation}
Provided $\ell_0$ is large enough, we may ensure that $0.5 \kappa \ell_0 \geq 30\ell$, so that $V_1'$ and $V_2'$ are both subsets of $\Sigma_n$, and $d(V_1',V_2') \geq 5 \ell L_{n-1}$. Thus, if $0$ is $n$-good, we have by \eqref{eq:bridge11} that $
\text{diam}(\text{Bad}) \leq \left(\ell+2\right) L_{n-1}$. Hence,
 there exists $V\in\{V_1,V_2\}$ such that
for every $x \in \mathbb{L}_{n-1}(V)$, the event $G_{n-1,x}(\Lambda_n)$ occurs.

Now, by definition of admissibility, both $S_1$ and $S_2$ must intersect $V$.
We now distinguish two cases, depending on whether $S_k \cap V$, $k=1,2$, are very close to each other (i.e.~at distance roughly of order $L_{n-1}$) or not. The former case, which we deal with next, can then essentially be dispensed with by ``zooming in'' and applying the induction hypothesis directly, which provides a bridge with the desired properties. 

\subsection*{Case 2:} $n \geq 1, \, d(S_1\cap V, S_2\cap V) < 15\kappa L_{n-1}$.
\medbreak
\begin{proof}
By considering a path $\gamma$ of $(n-1)$-boxes intersecting $V$ joining $S_1\cap V$ and $S_2\cap V$ of minimal length, we find a box $B=B_{L_{n-1}}(x)\in \gamma$ with $x \in \mathbb{L}_{n-1}(V)$ such that $\widetilde{\underline{\Lambda}}_{n-1} \stackrel{\textnormal{def.}}{=}B_{8\kappa L_{n-1}}(x)$ intersects both $S_1$ and $S_2$. Let $\widetilde{\Lambda}_{n-1} \stackrel{\textnormal{def.}}{=} B_{10 \kappa L_{n-1}}(x)$. 

Since $G_{n-1,x}(\Lambda_n)$ occurs for every $x \in \mathbb{L}_{n-1}(V)$ and since $G_{n,x}(\Lambda) \subseteq G_{n,x}(\Lambda')$ whenever $\Lambda' \subseteq \Lambda$ (this can be checked easily by induction), if $0$ is $n$-good then the event $G_{n-1,x}(\widetilde{\Lambda}_{n-1})\supseteq G_{n-1,x}(\Lambda_n)$ occurs, hence the induction assumption implies that there exists a good bridge $\mathcal{B}$ between $\widetilde{S}_1$ and $\widetilde{S}_2$ in $\widetilde{\Sigma}_{n-1}$, where $\widetilde{S}_i = S_i \cap \widetilde{\Lambda}_{n-1}$, $i=1,2$, and $\widetilde{\Sigma}_{n-1}= B_{ 9\kappa L_{n-1} }(x) \setminus B_{\lfloor 8.5\kappa L_{n-1}\rfloor }(x)$ (to apply the induction hypothesis, one observes that the sets $\widetilde{S}_1$ and $\widetilde{S}_2$ are admissible for $(\widetilde{\Lambda}_{n-1}, \widetilde{\underline{\Lambda}}_{n-1},\widetilde{\Sigma}_{n-1} )$).

We proceed to verify that the bridge $\mathcal{B}$ hereby constructed is in fact a good bridge between $S_1$ and $S_2$ in $\Sigma_n$. First, as we now explain, $\mathcal{B}$ is a bridge between $S_1$ and $S_2$. Indeed, \eqref{B1}, \eqref{B2} and \eqref{B4} are easy to check. For \eqref{B3}, since $B \in \widetilde{\Sigma}_{n-1}$ for any $B \in \mathcal{B}$, it follows that
$$
d\big(B, \bigcup_{i=1,2} (S_i \setminus \widetilde{S}_i) \big) \geq d\big(\widetilde{\Sigma}_{n-1}, \widetilde{\Lambda}_{n-1}^{\rm c}, \big)\geq \kappa L_{n-1},
$$
hence ``adding back'' $\bigcup_{i=1,2} (S_i \setminus \widetilde{S}_i)$ to form $S_1 \cup S_2$ does not produce additional constraints on the size of the boxes $B \in \mathcal{B}$ in \eqref{B3}. Thus $\mathcal{B}$ is a bridge between $S_1$ and $S_2$ inside $\Sigma_n$ since $G_{n-1,x}(\Lambda_n)$ occurs for every $x\in \mathbb L_{n-1}(V)$. 

It remains to argue that $\mathcal{B}$ is good. By definition of $G_{n,0}(\Lambda_n)$, the event $H_{n,z}$ occurs for the unique $z \in  \mathbb{L}_{n}(\Lambda_n)$ such that $x \in B_{n,z}$. Together with the induction assumption,  this implies \eqref{G1} and \eqref{G2}. This yields that $\mathcal G_n$ occurs as soon as $0$ is $n$-good and concludes this case.\end{proof}

\subsection*{Case 3:} $n \geq 1, \, d(S_1\cap V, S_2\cap V) \ge 15\kappa L_{n-1}$.
\begin{proof}
In this case, we have that if $W\stackrel{\textnormal{def.}}{=}V \setminus \bigcup_{x \in \partial V}B_{3\kappa L_{n-1}}(x )$ and $\widehat{S}_i\stackrel{\textnormal{def.}}{=} B(S_i\cap V, \kappa L_{n-1})$ (note that $\widehat{S}_i \cap W \neq \emptyset$ for $i=1,2$), then
$
d(\widehat{S}_1 \cap W, \widehat{S}_2 \cap W) \geq10\kappa L_{n-1}$. Using the fact that $ G_{n-1,x}(\Lambda_n)$ occurs for every $x\in V$ and that $G_{n,x}(\Lambda) \subseteq G_{n,x}(\Lambda')$ whenever $\Lambda' \subseteq \Lambda$, on the event $G_{n,0}(\Lambda_n)$ we can find a nearest-neighbor path $\gamma=(\gamma_1,\dots, \gamma_N)$ of vertices in $\mathbb{L}_{n-1}(W)$ of minimal length such that, if $B_i\stackrel{\textnormal{def.}}{=}B_{L_{n-1}}(\gamma_i)$,
\begin{align}
&\begin{array}{l}\text{$\bigcup_{i} B_i$ is connected, $B_i\cap( \widehat{S}_1 \cup \widehat{S}_2) =\emptyset$ for all $1 \leq i \leq N$,} \\ \text{and $d(B_1, \widehat{S}_1), d(B_N, \widehat{S}_2)\leq  3L_{n-1}$,} \end{array} \label{P1}  \tag{\bf P1} \\[0.3em]
&\begin{array}{l}\text{the events $G_{n-1,\gamma_i}(B_{3\kappa L_{n-1}}(\gamma_i))$ occur for all $1\leq i\leq N$.} \end{array} \label{P2}  \tag{\bf P2}
\end{align}
It remains to construct two suitable connections joining $S_1$ to  $B_1$ and $S_2$ to $B_N$, respectively. This will be done via two (good) \textit{arches}, defined below, whose existence is shown in Lemma~\ref{L:arch}. Together with the path $\gamma$, these will then yield the existence of a bridge $\mathcal{B}$ with the desired properties. 

Let $B$ be a $k$-box.  We say that a collection $\mathcal{A}$ of $n$-boxes with $0\leq n \leq k$ is an \textit{arch} between $U$ and $B$ in $\Sigma$ if \eqref{B1} holds with $\Sigma$ in place of $\Sigma_n$, 
 $B \in \mathcal{A}$ is the only $k$-box in $\mathcal{A}$, \eqref{B2} and \eqref{B3} both hold with $S_1=S_2=U$ and $B_1=B_2$ (and with $\mathcal A$ instead of $\mathcal B$), and \eqref{B4} holds with $K$ in place of $2K$. An arch $\mathcal{A}$ will be called {\em good} if \eqref{G1} and \eqref{G2} hold (with $\mathcal{A}$ in place of $\mathcal{B}$). 
 
 Set $L_k^+\stackrel{\textnormal{def.}}{=}L_k +L_{(k-1)\vee 0}$. The following lemma yields the existence of good arches.

\begin{lemma}
\label{L:arch}
For every $k \geq 0$ and $z \in \mathbb{L}_k$, if $G_{k,z}(B_{3\kappa L_k}(z))$ occurs, then with $B=B_{L_k}(z)$, for any set $U$ with the property that
\begin{equation}
\label{eq:condU}
\begin{split}
&\text{$\kappa L_k\le d(U, B) \le (\kappa+3)L_k$ and every con-}\\
&\text{nected component of $U$ intersects $\partial B_{2\kappa L_k^+}(z)$,}
\end{split}
\end{equation}
 there exists a good arch between $U$ and $B$ in $B_{2\kappa L_k}(z)$.
\end{lemma}

Assuming Lemma \ref{L:arch} holds, we first complete the proof of Lemma~\ref{L:bridge1} in Case 3 (and with it that of Theorem \ref{T:bridge1}). One applies Lemma \ref{L:arch} twice for $k=n-1$, with
$B= B_1$ and $U= U_1 \stackrel{\textnormal{def.}}{=} S_1\cap V \cap \tilde B_1$ where $\tilde B_1\stackrel{\textnormal{def.}}{=}B_{2\kappa L_k^+}(\gamma_1)$, respectively $B= B_N$ and $U= U_2 \stackrel{\textnormal{def.}}{=} S_2\cap V \cap \tilde B_2$ where $\tilde B_2\stackrel{\textnormal{def.}}{=}B_{2\kappa L_k^+}(\gamma_N)$. This is justified since the events $G_{k,\gamma_1}(B_{3\kappa L_k}(\gamma_1))$, $G_{k,\gamma_N}(B_{3\kappa L_k}(\gamma_N))$ occur by \eqref{P2}   and both $U_1$, $U_2$ satisfy \eqref{eq:condU} due to \eqref{P1}, 
and the admissibility of the sets $S_1$ and $S_2$. Thus, Lemma \ref{L:arch} yields the existence of a good arch $\mathcal{A}_1$ between $U_1$ and $B_1$, as well as a good arch $\mathcal A_2$ between $U_2$ and $B_N$. Let
\begin{equation}\label{eq:final_B}
\mathcal{B} \stackrel{\textnormal{def.}}{=} \{B_j : 1\leq j \leq N\} \cup \mathcal{A}_1 \cup \mathcal{A}_2.
\end{equation}
We proceed to check that the collection $\mathcal{B}$ is the desired good bridge between $S_1$ and $S_2$. Properties \eqref{B1}
and \eqref{B2} follow immediately from the corresponding properties of the arches $\mathcal{A}_1$, $\mathcal{A}_2$ and the definition of $\mathcal B$; in particular, $\bigcup_{B\in \mathcal{B}}B$ is connected. 
Property \eqref{B4} holds in the same way, noting that the number of $(n-1)$-boxes in $\mathcal{B}$ equals $ N \leq |\mathbb{L}_{n-1}(W)| \leq c (\kappa \ell_0)^d \leq K$ provided $K$ is chosen large enough (as a function of $\kappa$ and $\ell_0$).

We now turn to \eqref{B3}. The $(n-1)$-boxes $\{B_j : 1\leq j \leq N\} $ in $\mathcal{B}$ are at a distance greater than $\kappa L_{n-1}$ from $(S_1 \cup S_2)\cap V$ thanks to \eqref{P1}  and from $(S_1\cup S_2) \setminus V$ due to the fact that, by definition of $W$, $d(B_i, (S_1\cup S_2) \setminus V) > \kappa L_{n-1}$ for all $1\leq i \leq N$. The boxes at lower levels inherit the corresponding property from the arch they belong to, as we now explain. Consider a box $B \in \mathcal{A}_i \setminus \{ B_j, 1\leq j \leq N\}$. Thus $B$ is a $k$-box for some $k \leq n-2$. But $U_i = (S_1 \cup S_2) \cap \tilde B_i$ since $\tilde B_i \subset V$ (as follows from the definitions), and $S_j \cap \tilde B_i=\emptyset$ for $j \neq i$ since $\tilde B_i \cap S_i\cap V \neq \emptyset$, $\tilde{B}_i$ has radius smaller than $3\kappa L_{n-1}$ and $d(S_1\cap V, S_2\cap V) \ge 15\kappa L_{n-1}$. Thus, 
$$d(B,U_i) = d(B,  (S_1 \cup S_2) \cap \tilde{B}_i) \stackrel{\eqref{eq:condU}}{\geq} \kappa L_k $$
and since $B \subset B_{2\kappa L_{n-1}}(\gamma_1) \cup B_{2\kappa L_{n-1}}(\gamma_N) $ by construction $d(B, (S_1 \cup S_2) \setminus \tilde B_i) \geq \kappa L_{n-2}$. It follows that $\kappa L_k\le d(B, S_1 \cup S_2)$ as desired.

Finally, \eqref{G1}-\eqref{G2} are a consequence of the corresponding properties for the arches $\mathcal{A}_1$ and $\mathcal{A}_2$, \eqref{P2}, and the fact that $G_{n,0}(\Lambda_n)$ occurs (the latter to deduce that all the relevant events in $H_n$ also do). This completes the proof of the third case, and therefore of Theorem \ref{T:bridge1} (subject to Lemma~\ref{L:arch}).
\end{proof}

We conclude this section with the proof of Lemma \ref{L:arch}.

\begin{proof}[Proof of Lemma \ref{L:arch}] We assume for simplicity that $z=0$. Set $B\stackrel{\textnormal{def.}}{=}B_{L_k}$, $\Sigma\stackrel{\textnormal{def.}}{=}B_{2\kappa L_{k}}$ and $\tilde B\stackrel{\textnormal{def.}}{=}B_{2\kappa L_k^+}$. We proceed by induction over $k$. 

For $k=0$, the collection $\mathcal{A}$ of $0$-boxes corresponding to any nearest-neighbor path of $0$-boxes in $\Sigma$ joining $U$ to $B$ is an arch between $U$ and $B$ (note that $\Sigma \cap U \neq \emptyset$ by \eqref{eq:condU}). Moreover, in view of \eqref{eq:bridge10}, since $G_{0,z}(B_{3\kappa L_0})$ occurs by assumption, all the events $F_{0,x}$, $x \in \mathbb{L}_0(\Sigma)$ simultaneously occur, so that \eqref{G1} is satisfied. Since \eqref{G2} holds trivially (as $n_{\text{max}}=0$), $\mathcal{A}$ is a good arch.

We now assume that $k\geq 1$, and that the conclusions of the lemma hold for any $(k-1)$-box. Define
$
\overline B = B_{\kappa(2L_k-3L_{k-1})},
$
so that $B \subset \overline B \subset {\Sigma}$, and
\begin{equation}
\label{eq:arch10}
\text{Bad}_{\overline B} \stackrel{\textnormal{def.}}{=} \text{Fill}\Big( \bigcup_{y \in \mathbb{L}_{k-1}(\overline B):\, G_{k-1,y}^{c}(B_{3\kappa L_{k-1}}(y)) \text{ occurs }} B_{L_{k-1}}(y)\Big).
\end{equation}
Bounding the diameter of Bad as we did in Case 2, and noting that $B_{3\kappa  L_{k-1}}(y) \subset B_{3\kappa  L_k}$ for any $y \in \mathbb{L}_{k-1}(\overline B)$, one deduces from \eqref{eq:bridge11} that $\text{diam}(\text{Bad}_{\overline B}) \leq \left(\ell+2\right) L_{k-1}$ on the event $G_{k,0}(B_{3\kappa  L_k})$. For $U$ satisfying \eqref{eq:condU}, consider the disjoint sets  
\begin{equation}
\label{eq:arch11}
\begin{split}
&V_1 \stackrel{\textnormal{def.}}{=}  B , \\
&V_2\stackrel{\textnormal{def.}}{=} ((B_{(\kappa +3)L_{k-1}}(U) \setminus B_{\kappa L_{k-1}}(U)) \cap( \overline B \setminus  \text{Bad}_{\overline B}). 
\end{split}
\end{equation}
The upper bound on $\text{diam}(\text{Bad}_{\overline B})$ implies that, whenever $G_{k,0}(B_{3\kappa L_k})$ occurs (which will henceforth be assumed implicitly), $\overline B \setminus \text{Bad}_{\overline B} $ contains a connected component that intersects both $V_1$ and $V_2$ (for the latter, note that $\text{diam}(V_i) \geq L_k$ for $i=1,2$ thanks to \eqref{eq:condU}).
Hence, by \eqref{eq:arch10}, there exists a path $\gamma$ in $\mathbb{L}_{k-1}(\overline B \setminus \text{Bad}_{\overline B} )$ such that $\bigcup_{y \in \gamma} B_{L_{k-1}}(y)$ intersects both $V_1$ and $V_2$ and the events $G_{k-1,y}\left(B_{3\kappa L_{k-1}}(y)\right)$ occur for $y \in \gamma$. By choosing $\gamma$ to have minimal length, none of the boxes $B_{L_{k-1}}( y)$ with $y \in \gamma$ intersect $B_{\kappa L_{k-1}}(U)$.  For later purposes, record the collection 
\begin{equation}
\label{eq:arch12}
\begin{split}
\mathcal{A}' \stackrel{\textnormal{def.}}{=} & \{ V_1\} \cup \{ B_{L_{k-1}}(y) : y \in \gamma \}.
\end{split}
\end{equation}
Now, fix a vertex $y_0 \in \gamma$ such that $B'\stackrel{\textnormal{def.}}{=}B_{L_{k-1}}(y_0) \cap V_2 \neq \emptyset$
and consider $\tilde B'\stackrel{\textnormal{def.}}{=}B_{2\kappa L_{k-1}^+}(y_0)$. Since $y_0 \in \overline B $, we obtain that  $\tilde B'\subset \Sigma$.
The set $U'  \stackrel{\textnormal{def.}}{=}U \cap \overline B'$ is easily seen to satisfy \eqref{eq:condU} with $k-1$ in place of $k$ and $B'$ replacing $B$. Because $G_{k-1,y_0}\left(B_{3\kappa L_{k-1}}(y_0)\right)$ occurs, the induction assumption implies the existence of a good arch $\mathcal{A}''$ connecting $U' (\subset U)$ and $B'$ inside $\overline B' (\subset \Sigma)$. 

We claim that $\mathcal{A}= \mathcal{A}' \cup  \mathcal{A}'' $  has the desired properties, i.e.~it is a good arch between $U$ and $B$ inside ${\Sigma}$. Accordingly, we now argue that the (modified) conditions \eqref{B1}-\eqref{B4} and  \eqref{G1}-\eqref{G2} for arches hold. Condition \eqref{B1} is immediate by construction. So is \eqref{B2} since $\mathcal{A}''$ is a good arch between $U'$ and $B'$, $U' \subset U$ and any box $B \in \mathcal{A}'$ does not intersect $B_{\kappa L_{k-1}}(U)$. Condition \eqref{B3} follows readily from the induction assumption (applied to the boxes in $\mathcal{A}''$) and the fact that, except for $B=V_1$ which is at the correct distance from $U$, $\mathcal{A}'$, cf.~\eqref{eq:arch12}, only consists of $(k-1)$-boxes, none of which intersects $B_{\kappa L_{k-1}}(U)$, by definition of $V_2$ in \eqref{eq:arch11} and construction of $\gamma$. For \eqref{B3}, the bound on $N_m$, $m \leq k-2$ follows by the induction assumption, and $N_{k-1} =|\gamma| \leq K$ provided $K$ is chosen large enough, where we used that the boxes in $\gamma$ are all contained in $\overline B$.

Finally, the modified conditions \eqref{G1} and \eqref{G2} for $n \leq k-2$ and $m \leq k-1$ are immediate (by the induction hypothesis), and the remaining cases i) $n=k-1$, $m=k-1$, and ii) $m=k$ (and $n$ arbitrary) for \eqref{G2} follow from the occurrence of the events $G_{k-1,y}\left(B_{3\kappa L_{k-1}}(y)\right)$, $y \in \gamma$, and $G_{k,z}\left(B_{3\kappa L_{k}}\right)$. 
Overall, $\mathcal{A}$ is a good arch between $U$ and $B$ inside $\tilde B$, which completes the proof.
\end{proof}
\begin{remark}
\label{R:0-bridge} 1) In order to prove Lemma~\ref{L:bridge1} with the notion of admissibility put forth in Remark~\ref{R:bridges},2), i.e.,~assuming $S_1$ and $S_2$ to have diameter at least $\kappa L_n$, one follows the lines of the above proof, but replaces the sets $V_1'$ and $V_2'$ introduced in \eqref{eq:bridge16.01} by $\ell L_{n_1}$-thickenings of two disjoint nearest-neighbor paths of minimal length, each connecting $S_1$ and $S_2$. Importantly, the assumption on the diameter of $S_1$ and $S_2$ ensures that one can pick these paths in such a way that $d(V_1',V_2') \geq 5 \ell L_{n-1}$, where $V_k'=B_{\ell L_{n-1}}(V_k)$. As before, if $0$ is $n$-good, i.e.~if $G_{n,0}(\Lambda_n)$ occurs, then by~\eqref{eq:bridge11}
 there exists $V\in\{V_1,V_2\}$ such that
for every $x \in \mathbb{L}_{n-1}(V)$, the event $G_{n-1,x}(\Lambda_n)$ occurs.
Continuing as in the proof of Lemma~\ref{L:bridge1}, the bridge $\mathcal{B}$ is then constructed ``along'' this set $V$.

\medskip
\noindent 2) We briefly sketch how to extract the $0$-bridge of Corollary~\ref{cor:0bridge} from the above construction. Proceeding inductively, one proves that for all $n$, i) $G_{n,0}(\Lambda_n) \subset \mathcal{G}_n^0$ (the objective), and, along with it, ii) if $U$ is admissible, $B=B_{L_n}(z)$, $G_{n,z}(B_{3\kappa L_n}(z))$ occurs and \eqref{eq:condU} holds with $k=n$, then there exists a $0$-bridge between $U$ and $B$. 

For $n=0$, i) and ii) are readily verified (much as in Case~1 above). To carry i), ii) from level $n-1$ to $n$, one argues as follows. We will focus on Case~3 which is generic (Case~2 is simpler to handle). To show i) at scale $n$, let $\mathcal{B}$ be the bridge in \eqref{eq:final_B} joining $S_1$ and $ S_2$. For any of the $(n-1)$-boxes $B_i$ comprised in $\mathcal{B}$, $1\leq i \leq N$, fix a reference path $\gamma_i$ containing the center $z_i$ of $B_i$ and having diameter $\kappa L_{n-1}$. Note that all events $G_{n-1,z_i}(B_{10\kappa L_{n-1}}(z_i))$ occur by construction (see also Remark~\ref{R:bridges},1) above). Now applying the induction hypothesis ii), one obtains a $0$-bridge $\mathcal{B}_0'$ between $S_1$ and $B_1$ (as well as $\mathcal{B}_0''$, a $0$-bridge between $S_2$ and $B_N$). Moreover, by \eqref{eq:condU}, $\mathcal{B}_0'$ constitutes a subset of $B(z_1,8 \kappa L_{n-1})$ of diameter exceeding $\kappa L_{n-1}$. Thus applying hypothesis i), which is in force since $G_{n-1,z_1}(B_{10\kappa L_{n-1}}(z_1))$ occurs, to the box $B_1$, it follows that $\mathcal{B}_0'$ extends to a zero bridge between $S_1$ and $\gamma_1$. Continuing in this way, one deduces that $\mathcal{B}_0'$ and $\mathcal{B}_0''$ are in fact connected by a $0$-bridge, yielding i). 

One deduces ii) at level $n$ similarly, using Lemma~\ref{L:arch} to obtain an arch $\mathcal{A}$ between $B$ and $U$, applying the induction hypothesis ii) to a suitable sub-arch of $\mathcal{A}$ (in fact, the arch $\mathcal{A}''$ below \eqref{eq:arch12}, with $k \equiv n$), using that $G_{n-1,y}\left(B_{3\kappa L_{n-1}}(y)\right)$ occurs for the remaining $(n-1)$-boxes comprising $\mathcal{A}$, i.e.~for all of $B_{L_{n-1}}(y)$, $y \in \gamma$ in \eqref{eq:arch12}, and arguing similarly as above using i) for each $y \in \gamma$ to make $\mathcal{B}_0'''$ reach $U$.

\end{remark}

\section{Decomposition of $\varphi$ and ``bridging lemma''}
\label{sec:decompose_GFF}

In this section, we gather several results that will be needed for both the proofs of Proposition~\ref{prop:supercritical} (Section \ref{sec:supercritical}) and Proposition \ref{prop:comparison} (Section \ref{sec:comparison}).  Among other things, we set up a certain decomposition of the free field $\varphi$ (Lemma \ref{lem:white_noise_representation}) which will be used throughout, and prove a modified form of the ``gluing'' Lemma~\ref{lem:bridging} (Lemma \ref{lem:sprinkling}). The likelihood of the notion of ``goodness'' involved in the statement, see \eqref{eq:sprinkling}, will be guaranteed by an application of Theorem \ref{T:bridge1}.

\subsection{Decomposition of $\varphi$}
\label{subsec:decomposition}

Consider the graph with vertex set $\widetilde{\Z}^d=\Z^d \cup \mathbb{M}^d$, where $\mathbb{M}^d$ denotes the set of midpoints $\frac{x+y}{2}$, for $x,y\in \Z^d$ neighbors, with an edge joining every midpoint $m \in \mathbb{M}^d$ to each of the two vertices in $\Z^d$ at distance $\frac12$ from $m$ (each original edge is thereby split into two). Note that $\widetilde{\Z}^d$ is bipartite. Let $\widetilde{Q}$ be the transition operator (acting on $\ell^2(\widetilde{\Z}^d)$) for the simple random walk on $\widetilde{\Z}^d$, with transition kernel
\begin{equation}
\tilde{q}(\tilde{x}, \tilde{y}) = \frac{1}{|\{\tilde{z} \in  \widetilde{\Z}^d: \text{$\tilde{z}\sim\tilde{x}$}\}|} 1\{ \tilde{x} \sim \tilde{y} \},
\end{equation}
 for $\tilde{x}, \tilde{y} \in \widetilde{\Z}^d$,
where $\tilde{x} \sim \tilde{y}$ means that $\tilde{x}$ and $\tilde{y}$ are neighbors in $\widetilde{\Z}^d$, and write $\tilde{q}_{\ell}( \tilde{x}, \tilde{y})= (\widetilde{Q}^{\ell}1_{ \tilde{y}})( \tilde{x})$, for $\ell \geq 0$. 

Let $\mb{Z}= \{ \mb Z_{\ell}(\tilde{z}): \tilde{z} \in \widetilde{\Z}^d, \, \ell \geq 0  \}$, denote a family of independent, centered, unit variance Gaussian random variables under the probability measure $\P$. For later reference, let $\tau_x$, $x \in \Z^d$, denote the shifts on this space induced by $(\tau_x \mb Z_{\ell})(\tilde{z})=  \mb Z_{\ell}(x+\tilde{z})$, for $\tilde{z} \in \widetilde{\Z}^d$, $\ell \geq 0$. We define the processes $\xi^{\ell}$, $\ell \geq 0$ (and $\varphi$) alluded to in the introduction in terms of $\mb{Z}$ as follows. For each $\ell \in \{ 0,1,2,\dots\}$ and $x \in \Z^d$, let 
\begin{equation}
\label{eq:white_noise_representation}
\xi_x^\ell \stackrel{\textnormal{def.}}{=} c(\ell) \sum_{\tilde{z} \in \widetilde{\Z}^d}\tilde{q}_\ell(x,\tilde{z})\mb{Z}_\ell(\tilde{z}),
\end{equation}
with $c(\ell) = \sqrt{d/2}$ if $\ell$ is odd and $c(\ell) = \sqrt{1/2}$ if $\ell$ is even. Note that $\tilde{q}_0(\tilde y,\tilde{z})=\delta (\tilde y,\tilde{z})$ so $\xi^0_{\cdot}=\mb{Z}_0(\cdot)/\sqrt{2}$ is an i.i.d.~field indexed by $\Z^d$. The fields $\xi_{\cdot}^\ell$, $\ell \geq 0$ are independent, translation invariant centered Gaussian fields that have finite range:\begin{align}
& \label{range.G} \text{$\E[\xi_{x}^\ell\xi_{y}^\ell]=0$ for any $x,y\in  \Z^d$ with $|x-y|> \ell$,}
\end{align} 
which follows readily from \eqref{eq:white_noise_representation}. Our interest in $\xi_{\cdot}^\ell$ stems from the following orthogonal decomposition of $\varphi$.

\begin{lemma}
	\label{lem:white_noise_representation}
For every $\ell\geq 1$ and $x\in \Z^d$, $\var(\xi^\ell_x)\leq C \ell^{-\frac{d}{2}}$. In particular, the series
\begin{equation}
\label{eq:phi_workingdef}
\varphi \stackrel{\textnormal{def.}}{=} \sum_{\ell \geq 0} \xi^{\ell}
\end{equation}
converges (pointwise in $x$) in $L^2(\P)$. Moreover, the convergence also holds $\P$-a.s.~and the field $\varphi$ is a Gaussian free field under $\P$.
\end{lemma}

\begin{proof} One verifies that $\E[\xi_{x}^\ell\xi_{y}^\ell]= \frac12\tilde{q}_{2\ell}(x,y)$, for all $\ell \geq 0$ and $x, y \in \Z^d$, using in case $\ell$ is odd that $\widetilde{Q}\widetilde{Q}^* = \frac1d \widetilde{Q}^2$, where $\widetilde{Q}^*$ denotes the adjoint of $\widetilde{Q}$, with kernel $\tilde{q}^*(\tilde{x},\tilde{y})=\tilde{q}(\tilde{y},\tilde{x})$. One naturally identifies $q_{\ell}(x,y) \stackrel{\textnormal{def.}}{=} \tilde{q}_{2\ell}(x,y)$, for $x,y \in \Z^d$ as the transition kernel of a lazy simple random walk on $\Z^d$, which stays put with probability $\frac12$ and otherwise jumps to a uniformly chosen neighbor  at every step. 
 One knows from the local central limit theorem that 
\begin{equation}		
\label{eq:bound.decomp}
q_\ell(x,x)\leq \frac{C}{\ell^{d/2}}, \quad \text{for all $\ell \geq0$ and $x\in  \Z^d$,}
\end{equation}
which implies the convergence in $L^2(\P)$ in \eqref{eq:phi_workingdef}. The $\P$-a.s.~convergence is then standard (e.g.~as a consequence of Kolmogorov's maximal inequality). Finally, the previous observation also implies that
\begin{equation}
\label{decomp.G}
g(x,y)=\tfrac{1}{2}\sum_{\ell\geq 0}q_\ell(x,y) \stackrel{\eqref{eq:phi_workingdef}}{=} \E[\varphi_x\varphi_y] , \text{ for all $x,y\in  \Z^d$}
\end{equation}	 
with $g(\cdot, \cdot)$ as defined in \eqref{eq:Green}, so $\varphi$ defined by \eqref{eq:phi_workingdef} is indeed a Gaussian free field.
\end{proof}

We will tacitly work with the realization of $\varphi$ given by \eqref{eq:phi_workingdef}, \eqref{eq:white_noise_representation} throughout the remainder of this article. 
We now gather a few elementary properties of this setup. Denote the sequence of partial sums of $\xi^\ell$'s as
\begin{equation}
\label{eq:phi^k}
\varphi^L \stackrel{\textnormal{def.}}{=} \sum_{0 \leq \ell \leq L} \xi^{\ell}\,
\end{equation}
and define for $\Lambda \subset \Z^d$,
\begin{equation}
\label{Z_Lambda}
\mathcal{Z}(\Lambda)\stackrel{\textnormal{def.}}{=}\big\{\mathcal{Z}_{\ell}(\tilde{z}): \, (\ell, \tilde{z}) \text{ s.t.~} \tilde{q}_\ell(x,\tilde{z}) \neq 0 \text{ for some } x \in \Lambda\big\}.
\end{equation}
By \eqref{eq:white_noise_representation} and \eqref{eq:phi_workingdef}, $(\varphi_x)_{ x \in \Lambda} $ is measurable with respect to $\mathcal{Z}(\Lambda)$. Moreover, on account of \eqref{range.G}, for any $L \geq 0$,
\begin{equation}
\label{eq:phi_L_indep}
\text{$(\varphi_x^L)_{x \in U}$ is independent of $\mathcal{Z}(V)$ whenever $d(U,V)>L$}.
\end{equation}

We state below two simple lemmas which will be used repeatedly afterwards. The first one says that, up to a certain scale, it is easy to compare $\varphi$ and $\varphi^L$; while the second gives a lower bound for point-to-point connections in a box at levels below $h_{**}$. 

\begin{lemma}
	\label{lem:def_infty}
	There exist $c,C>0$ such that for every $\varepsilon>0$ and $L,R\geq1$,
	\begin{equation}
	\label{eq:infty_control2}
	\P[|\varphi_x-\varphi^L_x|<\varepsilon~,~\forall x\in B_R]\geq 1-CR^d e^{-c\varepsilon^2 L^{\frac{d-2}{2}}}.
	\end{equation}
\end{lemma}
\begin{proof} Since $\varphi_x-\varphi^L_x$ is a centered Gaussian variable and, using \eqref{eq:phi_workingdef},\eqref{eq:phi^k} and independence of $\xi^{\ell}_x$, $\ell \geq 0$, in the first equality below and recalling the bound \eqref{eq:bound.decomp}, its variance is bounded by
\begin{equation*}
\var(\varphi_x-\varphi^L_x)\stackrel{}{=} \sum_{\ell > L}\E[ (\xi^{\ell}_x)^2] \leq C \sum_{\ell > L} \ell^{-\frac{d}2} \leq C' \sum_{k \geq \lfloor\log L\rfloor} \e^{-\frac{k(d-2)}{2}} \leq C''L^{-\frac{d-2}{2}},
\end{equation*}
the result follows from a simple union bound and a standard Gaussian tail estimate.
\end{proof}

\begin{remark}
We note that the exact value $(d-2)/2$ of the exponent appearing in \eqref{eq:infty_control2} is immaterial for our arguments (cf.~for instance the proof of Lemma~\ref{lem:piv_decoupling}, in particular \eqref{def:alphat'}-\eqref{def:alphat}). In fact, this exponent is in a sense sub-optimal: using the Markov property of $\varphi$ at scale $L$, see for instance \cite[Lemma~1.2]{RodriguezSznitman13}, one can decompose $\varphi= \tilde{\varphi}^L + \eta^L$, into a sum of Gaussian fields, where $\tilde{\varphi}^L_{\cdot}$ has range of dependence $L$ and $\var(\eta^L_x) \leq CL^{-(d-2)}$. The missing factor of $2$ in \eqref{eq:infty_control2} stems from the fact that, while the walk may reach distance $L$ in as many steps (thus giving rise to a hard bound of the same order for the range of dependence of $\varphi^L$), its typical displacement is of order $\sqrt{L}$, and $\varphi^L$ starts to decorrelate at that scale.
\end{remark}

\begin{lemma}\label{lem:twopointsbound}
	For every $h< h_{**}$, there exist $\Cl{Ctwopoints}=\Cr{Ctwopoints}(d)>0$ and $\Cl[c]{ctwopoints}=\Cr{ctwopoints}(d,h)>0$ such that for every $L\geq 1$ and $x,y\in B_L$,
	\begin{equation}\label{eq:twopointsbound}
	\P[\lr{B_{2L}}{\varphi^{L}\geq h}{x}{y}]\geq \Cr{ctwopoints}L^{-\Cr{Ctwopoints}}.
\end{equation}
\end{lemma}
\begin{proof}
	For arbitrary $h<h_{**}$, let $\varepsilon\stackrel{\textnormal{def.}}{=} (h_{**}-h)/2$. By definition of $h_{**}$, see \eqref{eq:h_**}, we have 
	$$\P[\lr{}{\varphi\geq h+\varepsilon}{B_L}{\partial B_{2L}}]\geq c(h)>0~~~\forall L\geq 1.$$ A union bound over $x\in \partial B_L$ and translation invariance thus imply that
	$$
	\P[\lr{}{\varphi\geq h+\varepsilon}{0}{\partial B_{L}}]\geq c'(h)L^{-(d-1)}~~~\forall L\ge1.
$$
	By using arguments akin to those appearing in the proof of \cite[Lemma~6.1]{cerf2015}, which involve only the FKG inequality and the invariance under reflections and permutation of coordinates, we deduce that
		\begin{equation}\label{eq:twopointsproof2}
	\P[\lr{B_{2L}}{\varphi\geq h+\varepsilon}{x}{y}]\geq c''(h) L^{-\Cr{Ctwopoints}} ~~~\forall x,y\in B_L.
	\end{equation}
Finally, Lemma~\ref{lem:def_infty} enables us to replace $\{\varphi\ge h+\varepsilon\}$ by $\{\varphi^L\ge h\}$ provided $L$ is chosen large enough. This concludes the proof.
\end{proof}
\begin{remark}\label{rem:twopointsbound}
	Following the same lines as the proof above, one can show that for any percolation model $\omega$ satisfying an FKG inequality and invariance under reflections and permutation of coordinates, if one has $\P[\lr{}{\omega}{B_L}{\partial B_{2L}}]\geq a>0$ for all $L\leq R$, then for all  $x,y\in B_L$ and $L\leq R$, 
	\begin{equation}
	\P[\lr{B_{2L}}{\omega}{x}{y}]\geq  \Cl[c]{ctwopointsbis}(a) L^{-\Cr{Ctwopoints}}.
	\end{equation} This will be useful in Section~\ref{sec:comparison}.
\end{remark}

\subsection{The ``bridging lemma''}
\label{subsec:bridging}

We now borrow the notation from Section~\ref{sec:bridges}. Recall the definition of the scales $L_n$, $n \geq 0$, from \eqref{eq:bridge1}. We first choose
$\kappa=20$ and $\ell_0, K$ with $\ell_0 \geq 10\kappa$ large enough such that the conclusions of Theorem~\ref{T:bridge1} hold whenever $L_0 \geq  \Cr{bridges1}(\kappa,\ell_0)$. The parameters $\kappa,\ell_0$ and $K$ will remain fixed throughout the remainder of this article. This will guarantee that all exponents $\rho$ appearing in the following statements depend on $d$ only.

For the rest of this section, we use the notation $\Lambda_n\stackrel{\textnormal{def.}}{=}B_{10\kappa L_n}$ as appearing in \eqref{eq:annuli}, along with the corresponding notion of admissible sets, see above \eqref{eq:bridge.GOOD}. The use of annuli will not be necessary until Section \ref{sec:comparison}. We now prove a result which is slightly different from Lemma~\ref{lem:bridging} (see Remark \ref{R:bridginglemma},~2) below for a comparison between the two) and tailored to our later purposes.

\begin{lemma}[Bridging]
	\label{lem:sprinkling}
	For every $\varepsilon>0$ and $L_0 \geq \Cl{CL_0}(d,\varepsilon)$, there exist positive constants $\rho=\rho(d)>0$, $\Cl[c]{c:bridging}=\Cr{c:bridging}(d,\varepsilon,L_0)>0$ and  $\Cl{C:bridging}=\Cr{C:bridging}(d,\varepsilon, L_0)>0$ such that the following holds. For all  $n\geq0$, there is a family of events $\mathcal{G}(S_1,S_2)$ indexed by $S_1,S_2 \subset \L_n$, measurable and increasing with respect to ${\mb Z}(\L_n)$, such that 
	\begin{equation}\label{eq:nicelikely}
	\P \Big[\bigcap_{S_1,S_2} \mathcal{G}(S_1,S_2)\Big]\geq 1-  e^{-\Cr{c:bridging}L_n^\rho},
	\end{equation}
	 and for every $h\leq h_{**}-2\varepsilon$, every admissible $S_1, S_2$ and all events $D\in \sigma(1_{\varphi_x\geq h};~x\in S_1\cup S_2)$ and $E\in \sigma({\mb Z}(\Lambda_n^c))$, 
	 \begin{equation}
	\label{eq:sprinkling}
	\P\big[\lr{\Lambda_n}{\varphi\geq h-\varepsilon}{S_1}{S_2} ~\bigl\vert D\cap E \cap \mathcal{G}(S_1,S_2)\big]\geq e^{-\Cr{C:bridging} (\log  L_n)^2},
	\end{equation}
	whenever $\P[ D\cap E \cap \mathcal{G}(S_1,S_2)] > 0$.
\end{lemma}

\begin{remark}\label{R:bridginglemma} 1) We will apply Lemma~\ref{lem:sprinkling} in Section~\ref{sec:supercritical} in order to connect, after sprinkling, two families of clusters $\mathcal{C}_1$ and $\mathcal{C}_2$ inside of a ball $\L_n$ whenever the event $\mathcal{G}(S_1,S_2)$ occurs; cf.~also Fig.~\ref{F:bridge}. In this context, $S_1$ and $S_2$ will represent the explored regions of $\L_n$ when discovering $\mathcal{C}_1$ and $\mathcal{C}_2$ (i.e. $S_i=B(\mathcal{C}_i,1)\cap\L_n$, $i=1,2$, where $B(\mathcal{C}_i,1)$ is the $1$-neighborhood of the set $\mathcal{C}_i$ for the $\ell^{\infty}$-norm); $D$ will represent the information discovered inside this region; and $E$ will represent all the information outside of $\L_n$. \\[-0.5em]

\noindent 2) One can derive Lemma~\ref{lem:bridging} by following the proof of Lemma \ref{lem:sprinkling}, with minor modifications. The differences between the two are the following: in Lemma \ref{lem:sprinkling}, we require i) a certain measurability and monotonicity property of the events $\mathcal{G}(S_1,S_2)$ with respect to the $\sigma$-algebra $\mb Z(\Lambda_n)$ and ii) the bound \eqref{eq:sprinkling} on the connection probability to hold for admissible sets, rather than sets with large diameter, cf. Remark \ref{R:bridges},~2). 
\end{remark}


	\begin{proof}
	We start by defining events $\mathcal{G}(S_1,S_2)$ for which \eqref{eq:nicelikely} holds. Let $\varepsilon >0$ and $M,L_0 \geq 1$ to be chosen later. In the framework of Section~\ref{sec:bridges}, consider the event $\mathcal{G}_n$ (see \eqref{eq:bridge.GOOD})  given by the following choice of families of events $H$ and $F$:
	\begin{align}
	&F_{0,x}\stackrel{\textnormal{def.}}{=}\{\varphi_y^{L_0}-\varphi_y^0\geq -M+\varepsilon, ~~\forall y\in B_{L_0}(x)\}, \label{eq:gluinggood1} \\
	&H_{m, x} \stackrel{\textnormal{def.}}{=} \big\{\varphi^{L_m}_y  - \varphi^{L_{m-1}}_y \geq -\frac{6\varepsilon}{(\pi m)^2},~~\forall y\in  B_{2L_m}(x)\big\}. \label{eq:gluinggood2}
	\end{align}
Note for later purposes that the events \eqref{eq:gluinggood1}-\eqref{eq:gluinggood2} are in fact increasing in the $\mathcal{Z}$-variables, cf.~\eqref{eq:white_noise_representation} and \eqref{Z_Lambda}. Now, for any pair of admissible subsets $S_1, S_2$ of $\L_n$, define
	\begin{equation}\label{eq:defG}
	\mathcal{G}(S_1,S_2)\stackrel{\textnormal{def.}}{=} \bigcup_{\mathcal{B}} \{\text{$\mathcal{B}$ is good}\},
	\end{equation}
	where the union is taken over all the bridges between $S_1$ and $S_2$ inside $\Sigma_n$, see Definition \ref{def:goodbridge} and around \eqref{eq:annuli} for the relevant notions. For non-admissible $S_1, S_2$, set $\mathcal{G}(S_1,S_2)=\Omega$ (the full space on which $\P$ is defined). The events $\mathcal{G}(S_1,S_2)$ have the desired monotonicity property. For later reference, we note that 
\begin{equation}
\label{G_nchoice}
 \bigcap_{\substack{S_1, S_2}} \mathcal{G}(S_1,S_2) = \mathcal{G}_n,
 \end{equation}
 with $\mathcal{G}_n$ as defined in \eqref{eq:bridge.GOOD} for the choice of families $F$ and $H$ in \eqref{eq:gluinggood1}--\eqref{eq:gluinggood2}. We then assume (tacitly from here on) that $L_0 \geq \Cr{bridges1} \vee C(\varepsilon)$, so that the bounds in \eqref{C3} are respectively satisfied when choosing $M=(\log L_0)^2$ in \eqref{eq:gluinggood1}. Indeed, the bound for $\P[F_{0,x}^c]$ then simply follows by a union bound and a Gaussian tail estimate, noting that $\mathbb E[(\varphi_y^{L_0}-\varphi_y^0)^2] \leq C$ uniformly in $y \in \Z^d$ and $L_0 \geq 1$; to bound the probability of $H_{m,x}$, one proceeds similarly and uses \eqref{eq:bound.decomp}, or one applies \eqref{eq:infty_control2} twice. Any choice of $M=M(L_0)$ in \eqref{eq:gluinggood1} yielding  \eqref{C3} for the collection $F$ would work. It follows from \eqref{eq:white_noise_representation} and \eqref{range.G} that the families $F$ and $H$ in \eqref{eq:gluinggood1} and \eqref{eq:gluinggood2} satisfy \eqref{C1} and \eqref{C2} (recall that $\kappa=20$). Hence, Theorem~\ref{T:bridge1} applies and yields \eqref{eq:nicelikely}.

	The choices \eqref{eq:gluinggood1}, \eqref{eq:gluinggood2} and \eqref{eq:defG}, along with \eqref{G1}, \eqref{G2} in Definition~\ref{def:goodbridge} imply the following property, which will be used repeatedly in the sequel. For any good bridge $\mathcal{B}$ and any $m$-box $B=B_{L_m}(x) \in \mathcal{B}$ ($m\geq0$), the following holds:
	\begin{align}\label{eq:niceness}
	\begin{split}
		&\varphi_z-\varphi_z^0\geq -M, ~~\forall z\in B_{L_m}(x),  ~~\text{if $m=0$,}\\
		&\varphi_z-\varphi_z^{L_m}\geq -\varepsilon, ~~\forall z\in B_{2L_m}(x), ~~\text{if $m\geq1$}.
	\end{split}
	\end{align}


	We now turn to the proof of \eqref{eq:sprinkling}.
	Consider a bridge $\mathcal{B}$ between a pair of (admissible) sets ${S_1}$ and ${S_2}$ in $\Lambda_n$. 
	It follows directly from Definition \ref{def:bridge} that one can find vertices $s_1\in S_1\cap B_1$, $s_2\in S_2\cap B_2$ (recall $B_1$ and $B_2$ from \eqref{B2}) and $x_B, y_B \in B$ for each $B\in\mathcal{B}$ so that for any family of paths $(\pi_B)_{B\in\mathcal{B}}$ between $x_B$ and $y_B$, the union of $s_1, s_2$ and $(\pi_B)_{B\in\mathcal{B}}$ forms a path connecting $S_1$ and $S_2$, cf. also Fig.~\ref{F:bridge}. We can further impose that, for $B\in\{B_1,B_2\}$, the vertices $x_B, y_B$ are chosen in such a way that there exists a path $\pi_B \subset B\setminus (S_1\cup S_2)$ between $x_B$ and $y_B$ (in particular, $x_B, y_B \notin S_1\cup S_2$).
	For each $B=B_{L_m}(x)$, $x \in \mathbb{L}_m$, consider the event 
	\begin{equation} \label{eq:defA_B}
	A_B\stackrel{\textnormal{def.}}{=}
	\begin{cases} 
	\{\lr{B_{L_m}(x)}{\varphi^{0}\geq h+M}{x_B}{y_B}\}, &  \text{if $m=0$},\\
	\{\lr{B_{2L_m}(x)}{\varphi^{L_m}\geq h+\varepsilon}{x_B}{y_B}\}, &  \text{if $m\geq1$}.
	\end{cases}
	\end{equation}
	Then it follows directly from \eqref{eq:niceness} that $x_B$ and $y_B$ are connected in $\{\varphi\ge h\}\cap B_{2L_m}(x)$ if $B\in \mathcal{B}$ and $A_B$ occurs. 
	
	By these observations, we deduce that for any pair of admissible sets $S_1,S_2$, any events $D,E$ as above \eqref{eq:sprinkling}, and any bridge $\mathcal{B}$ inside $\Sigma_n$ between ${S_1}$ and ${S_2}$,
	\begin{multline}
		\label{eq:gluing1}
			\P\Big[\{\lr{\Lambda_n}{\varphi\geq h-\varepsilon}{S_1}{S_2}\}\cap D\cap E\cap \mathcal{G}(S_1,S_2) \Big]\\
			\geq \P\Big[\{\lr{\Lambda_n}{\varphi\geq h-\varepsilon}{S_1}{S_2}\}\cap D\cap E\cap\{\text{$\mathcal{B}$ is good}\}\Big]\\
			\geq \P\Big[ D\cap E\cap\{\text{$\mathcal{B}$ is good}\}\cap\{\varphi_{s_1},\varphi_{s_2}\geq h-\varepsilon\} \cap \bigcap_{B\in\mathcal{B}} A_B\Big]
	\end{multline}
	(also, note for the last inequality that the path in $\{\varphi \geq h -\varepsilon\}$ connecting $S_1$ and $S_2$ in the second line is indeed contained in $\Lambda_n$ since $\mathcal B$ itself lies in $\Sigma_n$, cf.~\eqref{eq:annuli} and \eqref{eq:defA_B}). Let  $\mathcal{F}$ denote the $\sigma$-algebra generated by the random variables $1_{\varphi_x\geq h}$, $x\in S_1\cup S_2$, as well as $ \mb Z_{\ell}(\tilde{z})$, $\tilde{z} \in \widetilde{\Z}^d$, $\ell \geq 0$, except $\mathcal{Z}_{0}(s_1)$ and $\mathcal{Z}_{0}(s_2)$. On account of \eqref{eq:white_noise_representation} and \eqref{eq:phi^k}, the latter are proportional to $\varphi_{s_1}^0$ and $\varphi_{s_2}^0$ respectively. Conditioning on $\mathcal{F}$ and on the $\mathcal{F}$-measurable event that $\varphi_{s_i}-\varphi_{s_i}^0\geq -M$ for $i=1,2$, one finds that
	$$ \P[\varphi_{s_i}  \geq h-\varepsilon\,|\, \mathcal{F}] \geq \inf_{a\geq -M} \inf_{h\leq h_{**}} \P[\varphi^0_0+a\geq h-\varepsilon \,\vert \, \varphi^0_0+a<h] =\Cl[c]{c:bridgelastpoint}(d,\varepsilon,L_0) >0.$$
Noticing that $\mathcal{B}$ being good implies that $\varphi_{s_i}-\varphi_{s_i}^0\geq -M$ for $i=1,2$ and that the event $D\cap E\cap\{\text{$\mathcal{B}$ is good}\} \cap \bigcap_{B\in\mathcal{B}} A_B$ is $\mathcal{F}$-measurable, one readily deduces that 
	\begin{multline}
		\label{eq:gluing2}
			\P\Big[ D\cap E\cap\{\text{$\mathcal{B}$ is good}\}\cap\{\varphi_{s_1},\varphi_{s_2}\geq h-\varepsilon\} \cap \bigcap_{B\in\mathcal{B}} A_B \Big]\\\geq \Cr{c:bridgelastpoint}^2 \P\Big[ D\cap E\cap\{\text{$\mathcal{B}$ is good}\} \cap \bigcap_{B\in\mathcal{B}} A_B \Big].
	\end{multline} 
Now, by \eqref{eq:white_noise_representation}, \eqref{eq:phi^k} and \eqref{eq:niceness}, \eqref{eq:defA_B}, conditionally on ${\mb Z}(\Lambda_n^c\cup{S_1}\cup{S_2})$, the events $\{\mathcal{B} \text{ is good}\}$ and $(A_B)_{B\in\mathcal{B}}$ are all increasing in the remaining random variables from ${\mb Z}$. Also, $D\cap E$ is measurable with respect to ${\mb Z}(\Lambda_n^c\cup{S_1}\cup{S_2})$ and $A_B$ is independent of ${\mb Z}(\Lambda_n^c\cup{S_1}\cup{S_2})$ for all $B\in\mathcal{B} \setminus\{B_1,B_2\}$ because of \eqref{B2} and \eqref{B3}. Together with the FKG-inequality for the i.i.d.~random variables in ${\mb Z}$, these observations imply that for suitable $c', c''$ depending on $d$, $\varepsilon$, and $L_0$,
	\begin{multline}
		\label{eq:gluing3}
\P\Big[ D\cap E\cap\{\text{$\mathcal{B}$ is good}\} \cap  \bigcap_{B\in\mathcal{B}} A_B\Big]		\\
=\E\Big[1_{D\cap E} ~ \P\Big[\{\text{$\mathcal{B}$ is good}\}\cap   \bigcap_{B\in\mathcal{B}} A_B \Big\vert {\mb Z}(\Lambda_n^c\cup{S_1}\cup{S_2})\Big]\Big] \\
 \geq \E\Big[1_{D\cap E} \, \P[\text{$\mathcal{B}$ is good}\,\vert \,{\mb Z}(\Lambda_n^c\cup{S_1}\cup{S_2})]\prod_{B\in\mathcal{B}} \P[A_B\vert {\mb Z}(\Lambda_n^c\cup{S_1}\cup{S_2})]\Big]\\
 \geq c'\P[ D\cap E\cap\{\text{$\mathcal{B}$ is good}\} ] \prod_{B\in\mathcal{B}\setminus \{ B_1,B_2\}} \P[A_B] \\
			 \geq  c''e^{-C(\varepsilon)(\log  L_n)^2} ~\P[ D\cap E\cap\{\text{$\mathcal{B}$ is good}\}],
	\end{multline}
	where in the fourth line, we used that $\prod_{i=1,2} \P[A_{B_i}\vert {\mb Z}(\Lambda_n^c\cup{S_1}\cup{S_2})] \geq c'$, which follows from the existence of a path $\pi_{B_i}\subset B_i\setminus(S_1\cup S_2)$ between $x_{B_i}$ and $y_{B_i}$ together with the fact that $\varphi^0$ is an i.i.d.~field. In the last line we used that for any $m$-box $B$ with $m \geq 1$ one has
	$$\P[A_B]\geq \Cr{ctwopoints}(L_m)^{-\Cr{Ctwopoints}}=\Cr{ctwopoints}\e^{-\Cr{Ctwopoints}(\log L_m)},$$
	where $\Cr{ctwopoints}=\Cr{ctwopoints}(d,h_{**}-\varepsilon)>0$ and $\Cr{Ctwopoints}>0$ are given by Lemma~\ref{lem:twopointsbound} (remind that $h+\varepsilon \leq h_{**}-\varepsilon$); while for $m=0$, we simply bounded $\P[A_B] \geq c(L_0, d, \varepsilon)>0$ using the finite energy of $\varphi^0$. The last line of \eqref{eq:gluing3} then follows from the fact that
	$$\sum_{\substack{B\in\mathcal{B}\\ m\text{-box}}} \Cr{Ctwopoints} \log L_m \leq 2 \Cr{Ctwopoints} K \sum_{0\leq 3m\leq n} (\log L_m) \leq C(\log  L_n)^2,$$ which relies on \eqref{B4}. 
	Combining \eqref{eq:gluing1}, \eqref{eq:gluing2} and \eqref{eq:gluing3}, we conclude that 
	\begin{align*}
	\P[\{\lr{\Lambda_n}{\varphi\geq h-\varepsilon}{S_1}{S_2}\}\cap D\cap E\cap \mathcal{G}(S_1,S_2) ]&\geq e^{-C(\log  L_n)^2} ~\P[ D\cap E\cap\{\text{$\mathcal{B}$ is good}\} ],
	\end{align*}
	where $C$ depends on $d$, $\varepsilon$ and $L_0$.
	Summing this inequality over the at most $$ \prod_{0\leq k \leq n} (CL_n/L_k)^{2dK} \leq \exp[C'(\log  L_n)^2]$$ possible bridges $\mathcal{B}$ between ${S_1}$ and ${S_2}$ gives
	\begin{align*}
		&\P[\{\lr{\Lambda}{\varphi\geq h-\varepsilon}{S_1}{S_2}\}\cap D\cap E\cap \mathcal{G}(S_1,S_2)  ]
\geq e^{-C{''}(\log  L_n)^2} ~\P[ D\cap E\cap \mathcal{G}(S_1,S_2) ],
	\end{align*}
	as desired. 
\end{proof}
\begin{remark}
	\label{remark:sprinkling}
	Retracing the steps of the above proof and imposing the occurrence of $\bigcap_{B\in\mathcal{B}} A_B$ but not of $\{\varphi_{s_1},\varphi_{s_2}\geq h-\varepsilon\}$ implies a connection between the $1$-neighborhoods of $S_1$ and $S_2$ \textit{without} an $\varepsilon$-sprinkling. In other words, Lemma \ref{lem:sprinkling} continues to hold with \eqref{eq:sprinkling} replaced by
	\begin{multline}
		\label{eq:finite_energy_wo_sprinkle}
		\P\Big[\bigcup_{\substack{s_1 \in S_1\\ s_2 \in S_2}} \big\{\lr{\L_n\setminus(S_1\cup S_2)}{\varphi\geq h}{\mathcal{N}(s_1)}{\mathcal{N}(s_2)}\big\} \cap \mathcal H_{s_1}\cap 
		\mathcal H_{s_2} ~\Bigl\vert   F\cap\mathcal{G}(S_1,S_2)\Big]\\\geq e^{-\Cr{C:bridging}(\log  L_n)^2},\,
	\end{multline}
	for all $h< h_{**}-2\varepsilon$, $L_0\geq C(\varepsilon)$, and $F\in\sigma(\mathcal{Z}(\Lambda_n^c\cup S_1\cup S_2))$, where $\mathcal H_y\stackrel{\textnormal{def.}}{=}\{\varphi_y-\varphi_y^0\geq -M\}$ with $M=M(L_0)$ suitably large (as chosen below \eqref{eq:defG}) and $\mathcal{N}(x)\stackrel{\textnormal{def.}}{=}\{y \in \Z^d: |y-x|_1 \leq 1\}$. We will use the kind of events appearing in \eqref{eq:finite_energy_wo_sprinkle} in Section~\ref{sec:comparison}. 
\end{remark}

\section{Local uniqueness regime}\label{sec:supercritical}

This section deals with Proposition~\ref{prop:supercritical}, whose proof is split into three parts. In Section~\ref{subsec:uniqueness} we prove Proposition~\ref{prop:uniq}, which roughly asserts that for $h<\tilde h$ with $\tilde h$ given by \eqref{eq:tildeh},  
crossing clusters inside an annulus are typically connected in $\{\varphi \geq h-\varepsilon \}$. This is then used in Section~\ref{subsec:renormalization} to trigger a renormalization and thereby deduce that $\{\varphi\geq h\}$ has a ``ubiquitous'' cluster inside large boxes for all values of $h<\tilde{h}$ with very high probability, see Proposition~\ref{prop:ubiq}. Finally in Section~\ref{subsec:supercriticalproof}, we use the previous result in order to conclude the proof of Proposition~\ref{prop:supercritical} by proving the desired stretched-exponential decay of the probabilities defining $\bar h$ in \eqref{eq:barh1}. Some care is needed because one ultimately wants to avoid any sprinkling for the local uniqueness event. The last part of the argument would simplify if one worked with a weaker notion of $\bar h$ as in \cite{MR3417515} involving sprinkling for the uniqueness event \eqref{eq:UNIQUE1}, see Remark \ref{R:hbar_alternative} below.

\subsection{From connection to local uniqueness}\label{subsec:uniqueness}

We start by defining a certain ``unique crossing'' event $\mathcal{E}$: given any $\alpha>\beta$, let 
\begin{multline}
\label{eq:defcross}
\mathcal{E}(N,\alpha,\beta)\stackrel{\textnormal{def.}}{=} \{\lr{}{\varphi\geq \alpha}{B_N}{\partial B_{6N}}\}\,\cap\\\Big\{\begin{array}{c}\text{all clusters in $\{\varphi\geq\alpha\}\cap B_{4N}$ crossing $B_{4N}\setminus B_{2N}$} \\ \text{are connected to each other in $\{\varphi\geq \beta\}\cap B_{4N}$} \end{array}\Big\}.
\end{multline}
Notice that unlike ${\rm{Unique}}(2N,\alpha)$ defined in  \eqref{eq:UNIQUE1}, the corresponding event in $\mathcal{E}(N,\alpha,\beta)$ involves a sprinkling. 
\begin{prop}
	\label{prop:uniq}
	For every $\varepsilon>0$ one has
	\begin{equation}\label{eq:propuniq2}
	\limsup_{N\to \infty} \inf_{h\leq \tilde{h}-2\varepsilon} \P[\mathcal{E}(N,h,h-\varepsilon)]=1.
	\end{equation}
\end{prop}

\vspace{0.3cm}

The idea of the proof is roughly the following. We first require that all the balls of size $u(N)$  inside $B_{4N}$ are connected to distance $N$, which happens with probability  converging to~1 along a subsequence of values of $N$ since $h<\tilde{h}$. On this event, the picture we see at level $h$ inside the ball $B_{N}$ is that of an ``almost everywhere percolating'' subgraph: every vertex is at distance at most $ u(N) \ll N$ from some macroscopic cluster in $B_N$. In other words, the union of all macroscopic clusters form a $ u(N)$-dense subset of the ball $B_{4N}$. The goal is then to adapt the techniques from \cite{benjaminitassion17} in order to show that after an $\varepsilon$-sprinkling, all such clusters will be connected together. In order to implement this adaptation we need some kind of ``sprinkling property'' stating that conditionally on the configuration at level $h$, there is a decent probability of making extra connections at level $h-\varepsilon$. As explained in the introduction, the level sets of $\varphi$ do \textit{not} have such a property and this issue will be overcome by applying Lemma~\ref{lem:sprinkling}. 

\begin{proof}[Proof of Proposition~\ref{prop:uniq}] 
Fix $\varepsilon>0$ and take any $h\leq\tilde{h}-2\varepsilon$.
We import the notation and definitions from Section~\ref{sec:decompose_GFF}. We fix $L_0=L_0(\varepsilon)$ large enough 
such that the conclusions of Lemma~\ref{lem:sprinkling} hold (recall that $h\leq \tilde{h}-2\varepsilon\leq h_{**}-2\varepsilon$). We say that a ball $\L_n=B_{10\kappa L_n}(x)$ for some $x \in \Z^d$ is \textit{good} if the event $ \mathcal{G}_{n,x}$ defined below \eqref{eq:bridge.GOOD} occurs for the families of events $F$ and $\{H_m, m \geq 1\}$ given by \eqref{eq:gluinggood1}-\eqref{eq:gluinggood2}. Throughout the remainder of this section, all constants $c,C$ may depend implicitly on $\varepsilon$.

Recall the function $u(\cdot)$ defined above \eqref{eq:tildeh}, let $n_0\stackrel{\textnormal{def.}}{=} \min\{n: L_n\geq u(10N)\}$, $h \leq \tilde{h}-2\varepsilon$ and consider the events
\begin{align}
A&\stackrel{\textnormal{def.}}{=}\{\lr{}{\varphi\geq h}{B_{L_{n_0}}(x)}{\partial B_{6N}} ~\text{ for all } x \text{ s.t. } B_{L_{n_0}}(x) \subset B_{4N}\} \label{eq:supercritA}, \\
G&\stackrel{\textnormal{def.}}{=}\{\text{$B_{10\kappa L_{n_0}}(x)$ is good for all $x\in \mathbb{L}_{n_0}( B_{4N})$}\}  \label{eq:supercritG}
\end{align}
(see above \eqref{eq:bridge10} for notation). Clearly, if $A$ does not occur, then there must be $x\in B_{4N}$ such that $B_{u(10N)}(x)$ is not connected to $\partial B_{6N}$ in $\{\varphi \geq h \}$. Now since $B_{6N} \subset B_{10N}(x)$ for any $x \in B_{4N}$, it follows directly that $B_{u(10N)}(x)$ is not connected to $\partial B_{10N}(x)$ for some $x\in B_{4N}$ on the complement of $A$. 
Consequently, by translation invariance, 
\begin{equation}
\label{eq:boundA}
\P[A]\geq 1-CN^d\P[\nlr{}{\varphi\geq h}{B_{u(10N)}}{\partial B_{10N}}].
\end{equation}
At the same time, it follows  from Lemma~\ref{lem:sprinkling} that for any $N$ sufficiently large, one has
\begin{equation}
\label{eq:boundG}
\P[G]\geq 1-e^{- \Cr{c:bridging}'u(N)^\rho}.
\end{equation}
Define the collection
\begin{equation}
\label{eq:definiq1}
\mathcal{C} \stackrel{\textnormal{def.}}{=} \{\text{$C\subset B_{4N}$: $C$ a cluster in $\{ \varphi \geq h \} \cap B_{4N}$ intersecting $\partial B_{4N}$}\},
\end{equation}
and for any percolation configuration $\omega\in\{ 0,1\}^{\Z^d}$ with $\{\omega=1\}\supset \{\varphi \geq h\} $, let
\begin{equation}
\label{eq:equiv}
C  \sim_{\omega} C' \text{ if } \lr{}{\omega}{C}{C'}, \text{ for } C,C' \in \mathcal{C}. 
\end{equation}
The relation $\sim_{\omega}$ defines an equivalence relation on any $\tilde{\mathcal{C}} \subset \mathcal{C}$.
The elements of $\tilde{\mathcal{C}} / \sim_{\omega} $ thus form a partition of $\tilde{\mathcal{C}}$, whereby clusters of $ \tilde{\mathcal{C}}$ which are connected in the configuration $\omega$ get grouped. For every $0\leq i\leq \floor{2\sqrt{N}}$, let $V_i\stackrel{\textnormal{def.}}{=}B_{4N-i\sqrt{N}}$. We will study the sets
\begin{equation}
\label{eq:Uij}
\mathcal{U}_{i}(\omega)\stackrel{\textnormal{def.}}{=} \big\{C \in \mathcal{C} : C\cap V_{2i} \neq \emptyset  \, \big\} \big/ \sim_{\omega}
\end{equation}
for $0\leq i\leq \floor{ \sqrt{N}}$ and $\{\omega=1\}\supset \{\varphi \geq h\} $. We denote by ${U}_{i}(\omega)= |\mathcal{U}_{i}(\omega)|$ and will frequently rely on the fact that ${U}_{i}(\omega)$ is decreasing in both $\omega$ and \nolinebreak$i$, as apparent from \eqref{eq:Uij}. 
We will use $\mathscr{C}$ in the sequel to denote groups of clusters of $\mathcal{C}$, e.g.~elements of $\mathcal{U}_{i}(\omega)$, and more generally of $2^{\mathcal{C}}$. It will be convenient to write $\text{supp}(\mathscr{C})=\bigcup_{C \in \mathscr{C}} C \subset \Z^d$, for $\mathscr{C} \in 2^{\mathcal{C}}$. Now, for $0\leq i\leq \floor{\sqrt{N}}$, introduce the percolation configurations 
\begin{equation}
\label{omega_i}
\omega_0 \leq \omega_1 \leq \dots, \quad \text{ where } \omega_i =\omega_i (\varphi)\stackrel{\textnormal{def.}}{=} \begin{cases}
1_{\{\varphi \geq h\}}, & x \in V_{2i},\\
1_{\{\varphi \geq h-\varepsilon\}}, & x \notin V_{2i}, 
\end{cases}
\end{equation}
so $\omega_i \in\{ 0,1\}^{\Z^d}$ corresponds to a partial sprinkling outside of $V_{2i}$, and set 
\begin{equation}
\label{eq:Ui}
\mathcal{U}_{i}\stackrel{\textnormal{def.}}{=} \mathcal{U}_{i}(\omega_{i}), \quad U_i \stackrel{\textnormal{def.}}{=}|\mathcal{U}_i|, \quad 0\leq i\leq \floor{\sqrt{N}}. 
\end{equation}
Note that $U_i$ is decreasing in $i$. In view of \eqref{eq:defcross}, \eqref{eq:supercritA} and \eqref{eq:Uij}--\eqref{eq:Ui}, the event $\mathcal{E}(N,h,h-\varepsilon)$ occurs as soon as $A$ does and ${U}_{\lfloor\sqrt{N}\rfloor} =1$. Hence, \eqref{eq:boundA} and \eqref{eq:boundG} give that
\begin{multline}\label{eq:uniqueproof}
\P[\mathcal{E}(N,h,h-\varepsilon)^c] 
\\ \leq \P[A^c] + \P[G^c] + \P[A\cap G\cap\{ {U}_{\lfloor\sqrt{N}\rfloor}  >1\}] \\ 
 \leq CN^d  \P[\nlr{}{\varphi\geq h}{B_{u(10N)}}{\partial B_{10N}}] + e^{- \Cr{c:bridging}'u(N)^\rho} + \P[A\cap G\cap\{{U}_{\lfloor\sqrt{N}\rfloor} >1\}].
\end{multline} 
By the definition of $\tilde{h}$ and the monotonicity of the disconnection event with respect to $h$, the first two terms on the right-hand side of \eqref{eq:uniqueproof} converge to $0$ uniformly in $h\leq \tilde{h}-2\varepsilon$ along a subsequence $N_k\to\infty$. As a consequence, it suffices to prove that the last term tends to 0 uniformly in $h\leq \tilde{h}-2\varepsilon$ as $N\to\infty$ to conclude. 
The proof will be based on the following lemma.
\begin{lemma}
	\label{lem:uniqreduct} There exists a constant $c=c(\varepsilon)>0$ such that for any $h\leq \tilde{h}-2\varepsilon$, $N \geq 1$, any $a\in \mathbb{N}$ with $4\leq a \leq N^{1/4}$ and any  $0\leq i\leq \floor{\sqrt{N}}-a$,
	\begin{equation}
	\label{eq:uniqreduct}
	\P[ A\cap G\cap \{ U_{i+a} > 1\vee 2 U_i /a \} ] \leq \exp(-cN^{1/4}).
	\end{equation}
\end{lemma}

Admitting Lemma~\ref{lem:uniqreduct}, we first finish the proof of Proposition \ref{prop:uniq}. Observe that, if the event $\bigcap_{0\leq k < M } \{  U_{(k+1)a} \leq 1\vee \frac{2 U_{ka} }{a} \}$ occurs for some $M \geq 1$ and an $a$ as appearing in Lemma \ref{lem:uniqreduct} with $aM \leq \lfloor\sqrt{N}\rfloor$, then either 
\begin{equation}
\label{eq:Ubound1}
{U}_{\lfloor\sqrt{N}\rfloor}  \leq  U_{Ma} \leq \frac{2}{a} U_{(M-1)a} \leq \cdots \leq \Big(\frac{2}{a}\Big)^{M} U_0 = \Big(\frac{2}{a}\Big)^{M} |\mathcal{C}| \stackrel{\eqref{eq:definiq1}}{\leq} C  \Big(\frac{2}{a}\Big)^{M} N^{d-1},
\end{equation}
or $U_{(k+1)a} \leq 1$ for some $0\leq k < M $, in which case ${U}_{\lfloor\sqrt{N}\rfloor}  \leq U_{(k+1)a} \leq 1$ by monotonicity. Thus, letting $M = \lfloor(C' \log N / \log a)\rfloor$, with $C'=C'(d)$ chosen large enough so that the right-hand side of \eqref{eq:Ubound1} is bounded by $1$, we deduce that 
$A\cap G\cap\{{U}_{\lfloor\sqrt{N}\rfloor} >1\} $ implies the event $$ \bigcup_{0\leq k < M } A\cap G \cap  \{  U_{(k+1)a} \geq 1\vee 2 U_{ka} /a \}.$$ 

Applying a union bound over $k$, choosing say, $a=4$, \eqref{eq:uniqreduct} readily yields that 
\begin{equation}
\label{eq:reduct2}
 \P\big[A\cap G\cap\big\{{U}_{\lfloor\sqrt{N}\rfloor} >1\big\}\big] \leq  M \exp(-cN^{1/4}).
\end{equation}
Proposition~\ref{prop:uniq} then follows immediately from \eqref{eq:uniqueproof} and \eqref{eq:reduct2}.
\end{proof}

\begin{figure}[h!]
  \centering 
  \includegraphics[scale=0.44]{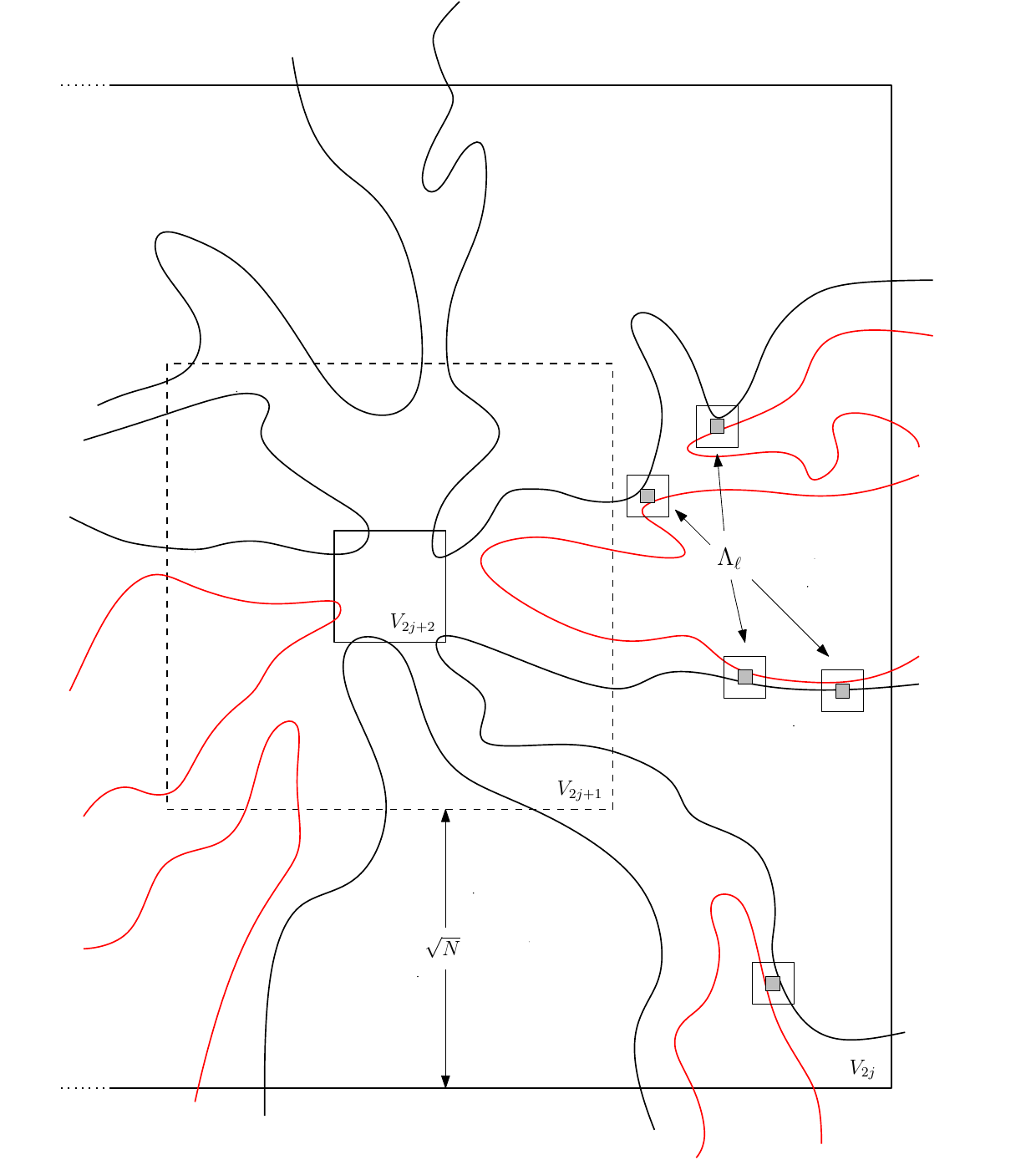}
  \caption{Some of the clusters in $\mathcal{C}$. On the event $A$, these clusters are $u(N)$-dense in the annulus $V_{2j}\setminus V_{2j+2}$. Each of the boxes $\tilde{\Lambda}_{\ell}$ (grey) is intersected by both the clusters in the support of $\widetilde{\mathcal{U}}_1$ (black) and $\widetilde{\mathcal{U}}_2$ (red), see \eqref{Upartition}; the picture corresponds to Case~2, i.e. $ |\mathcal{U}_{j+\frac12,j+1}(\omega_j)| \neq 0$ in the arguments following \eqref{Uconnect}--\eqref{Upartition}. When $G$ occurs, each box $\Lambda_{\ell}$ provides the opportunity to link two admissible pieces of $\widetilde{C}_1$ and $\widetilde{C}_2$ using a good bridge at a cost given by Lemma \ref{lem:sprinkling}.}
  \label{F:dense}
\end{figure}

It remains to give the proof of Lemma~\ref{lem:uniqreduct}. Roughly speaking, upon sprinkling in the corresponding annulus to increase from configuration $\omega_i$ to $\omega_{i+a}$, cf.~\eqref{omega_i}, the presence of the event $G$ in \eqref{eq:uniqreduct} will provide many opportunities for clusters (implied by the event $A$) crossing this annulus to merge at a preferential cost (supplied by Lemma~\ref{lem:sprinkling}), thus rendering it likely for $U_{i+a}$ to decrease significantly compared to $U_i$, as asserted in \eqref{eq:uniqreduct}.

\begin{proof}[Proof of Lemma~\ref{lem:uniqreduct}] We begin with a reduction step. For $\{\omega=1\}\supset \{\varphi \geq h \}$, $0\leq i < \floor{ \sqrt{N}}$ and $k\in \{0,\frac12\}$, let
\begin{equation}
\label{eq:U_ijbis}
\mathcal{U}_{i+k,i+1}(\omega) \stackrel{\textnormal{def.}}{=} \{ \mathscr{C} \in \mathcal{U}_i(\omega): \text{supp}(\mathscr{C}) \cap V_{2(i+1)} =\emptyset, \, \text{supp}(\mathscr{C}) \cap V_{2(i+k)} \neq \emptyset  \},   \end{equation}
and set $U_{i,i+1}(\omega) \stackrel{\textnormal{def.}}{=} |\mathcal{U}_{i,i+1}(\omega) |$. Note that $U_{i,i+1}(\omega)$ is decreasing in $\omega$. Moreover, since the elements of $ \mathcal{U}_i(\omega)\setminus \mathcal{U}_{i,i+1}(\omega) $ each contain a cluster $C \in \mathcal{C}$ intersecting $V_{2(i+1)}$, the map $\psi:  \mathcal{U}_i(\omega) \setminus \mathcal{U}_{i,i+1}(\omega)  \to \mathcal{U}_{i+1}(\omega)$ given by $ \psi(\mathscr{C})\stackrel{\textnormal{def.}}{=} \mathscr{C} \setminus \{ C \in \mathcal{C} : C\cap V_{2(i+1)} =\emptyset\} $ is a bijection. Hence, $U_{i} (\omega) = U_{i+1} (\omega) +{U}_{i,i+1}(\omega) $ and by iteration
$$
 {U}_{i+a }(\omega_i) + \sum_{j: \, i \leq j < i+a} {U}_{j,j+1}(\omega_i)  = {U}_{i }(\omega_i)=U_i,
$$
whence ${U}_{j,j+1}(\omega_i)  \leq U_i/a$ for some $j$ with $ i \leq j < i+a$. Together with a union bound, we see that \eqref{eq:uniqreduct} follows at once if we can show that 
\begin{equation}
\label{eq:uniqreduct_new1}
\P\big[ A\cap G\cap \big\{ U_{j+1}> 1\vee \big( U_{j}/a + {U}_{j, j+1}(\omega_j)   \big)\big\} \big] \leq \exp(-c'N^{1/4}),
\end{equation}
for all $0\leq i\leq \floor{\sqrt{N}}-a$ and $ i \leq j < i+a$. 
To see this, simply notice that $U_{i+a} = U_{i+a}(\omega_{i+a}) \leq  U_{j+1}(\omega_{j+1})=U_{j+1} $ by monotonicity since $i+a \geq j+1$, and similarly that $ U_i \geq U_j $ and $  {U}_{j,j+1}(\omega_i) \geq {U}_{j,j+1}(\omega_j)$, for all $i \leq j$.

We now prove \eqref{eq:uniqreduct_new1} for all $0\leq j< \floor{\sqrt{N}}$. Fix any such $j$ and let $E$ denote the event on the left-hand side of \eqref{eq:uniqreduct_new1}. Recalling \eqref{eq:U_ijbis}, we introduce
\begin{equation}
\label{Utilde}
\widetilde{\mathcal{U}}(\omega_j) \stackrel{\textnormal{def.}}{=}
\begin{cases}
 \mathcal{U}_{j}(\omega_j) \setminus \mathcal{U}_{j,j+1}(\omega_j), & \text{ if } \mathcal{U}_{j+\frac12,j+1}(\omega_j) =\emptyset, \\
( \mathcal{U}_{j}(\omega_j) \setminus \mathcal{U}_{j,j+1}(\omega_j)) \cup\{\widetilde{\mathscr{C}}\}, & \text{ otherwise, } 
\end{cases}
\end{equation}
where  $\widetilde{\mathscr{C}} \stackrel{\textnormal{def.}}{=} \{ C : C\in \mathscr{C} \text{ for some } \mathscr{C} \in \mathcal{U}_{j, j+1}(\omega_j)  \}$ is obtained by merging the elements of $\mathcal{U}_{j, j+1}(\omega_j) $. We drop the argument $\omega_j$ in the sequel and proceed to verify that $\widetilde{\mathcal{U}}$ has the following properties: on the event $E$,
\begin{align}
&\begin{array}{l}
\text{for $A_j=V_{2j}\setminus V_{2j+1}$ or $A_j=V_{2j+1}\setminus V_{2j+2}$, each of the} \\
\text{sets $\text{supp}(\mathscr{C})$, with $\mathscr{C} \in \widetilde{\mathcal{U}}$, crosses $A_j$ and their union}\\
\text{intersects all the balls of radius $L_{n_0}$ contained in $A_j$,}\\
\end{array} \label{Uconnect}\\
&\begin{array}{l}
\text{$\exists$ a non-trivial partition $\widetilde{\mathcal{U}}= \widetilde{\mathcal{U}}_1 \sqcup \widetilde{\mathcal{U}}_2$}\\\text{such that $|\{ \mathscr{C} : \mathscr{C} \in \widetilde{\mathcal{U}}_1  \}| \leq a$ and $\displaystyle \{\nlr{}{\omega_{j+1}}{\widetilde{C}_1}{\widetilde{C}_2}\}$}, 
\end{array} \label{Upartition}
\end{align}
where $ \textstyle \widetilde{C}_i = \bigcup_{\mathscr{C}\in \widetilde{\mathcal{U}}_i}\text{supp}(\mathscr{C}) $. We first check that \eqref{Uconnect} holds with the choice $A_j=V_{2j+1}\setminus V_{2j+2}$ when $|\mathcal{U}_{j+\frac12,j+1}| =0$ (henceforth referred to as Case~1) and $A_j=V_{2j}\setminus V_{2j+1}$ when $|\mathcal{U}_{j+\frac12,j+1}| \neq0$ (Case 2). Indeed, in either case 
each $\mathscr{C} \in \mathcal{U}_{j} \setminus \mathcal{U}_{j,j+1}$ contains a cluster $C$ crossing $V_{2j}\setminus V_{2(j+1)}$, see  \eqref{eq:Uij} and \eqref{eq:U_ijbis}. Moreover, the assumption  $|\mathcal{U}_{j+\frac12,j+1}| \neq 0$ of Case 2 implies that $\widetilde{\mathscr{C}} $ defined below \eqref{Utilde} contains a cluster $C$ crossing $V_{2j}\setminus V_{2j+1}$.

To conclude that \eqref{Uconnect} holds, it thus remains to check that all the $L_{n_0}$-balls in $A_j$ are intersected by the set $\bigcup_{\mathscr{C} \in \widetilde{\mathcal{U}}}\text{supp}(\mathscr{C})$. First note that on the event $E \subset A$ (recall \eqref{eq:supercritA}), by definition of $ \mathcal{U}_j$ each such ball is intersected by $\text{supp}(\mathscr{C})$, for some $\mathscr{C} \in \mathcal{U}_j$. Since each $\mathscr{C} \in \mathcal{U}_j$ belongs to a group of $\widetilde{\mathcal{U}}$ in Case 2, the claim immediately follows. In Case 1, the assumption $ |\mathcal{U}_{j+\frac12,j+1}(\omega_j)| =0$ implies that none of the sets $\text{supp}(\mathscr{C})$, $\mathscr{C} \in  \mathcal{U}_{j,j+1} $, intersects $V_{2j+1}\setminus V_{2j+2}=A_j$ and \eqref{Uconnect} follows as well.

We now argue that \eqref{Upartition} holds. It suffices to show that in either case,
\begin{equation}
\label{eq:Upidgeonhole}
\text{on $E$, } \quad |\widetilde{\mathcal{U}} /\sim_{\omega_{j+1}}| \geq \frac{|\widetilde{\mathcal{U}} |}{a},
\end{equation}
where, with hopefully transparent notation, we extend the relation \eqref{eq:equiv} by declaring $\mathscr{C}\sim_{\omega} \mathscr{C}'$ for arbitrary $\mathscr{C}, \mathscr{C}'\in 2^{\mathcal{C}}$ if $\lr{}{\omega}{C}{C'}$  for some $C \in  \mathscr{C}$ and $C' \in \mathscr{C}'$. Indeed if \eqref{eq:Upidgeonhole} holds then at least one element of $\widetilde{\mathcal{U}} /\sim_{\omega_{j+1}}$ is obtained by merging at most $a$ elements of $\widetilde{\mathcal{U}} $, and its clusters are not connected to their complement in $\widetilde{\mathcal{U}} $ by definition of $\sim_{\omega_{j+1}}$. In Case $1$ we simply use that $U_{j+1}\geq U_{j} /a$ on $E$ by \eqref{eq:uniqreduct_new1}, from which \eqref{eq:Upidgeonhole} follows because $ U_{j}=|\mathcal{U}_j| \geq |\widetilde{\mathcal{U}}|$, cf.~\eqref{Utilde}, and 
$$|\widetilde{\mathcal{U}} /\sim_{\omega_{j+1}}|=|\mathcal{U}_{j+1}(\omega_j) /\sim_{\omega_{j+1}} | =|\mathcal{U}_{j+1}(\omega_{j+1})|= U_{j+1}$$ (regarding the first of these equalities, see the discussion following \eqref{eq:U_ijbis} and the definition below \eqref{eq:Upidgeonhole}). In Case $2$ we deduce \eqref{eq:Upidgeonhole} from $U_{j+1}>  |\widetilde{\mathcal{U}}|/a + {U}_{j, j+1}(\omega_j) $, which holds on $E$ due to \eqref{eq:uniqreduct_new1} and \eqref{Utilde}, together with the inequality $|\widetilde{\mathcal{U}} /\sim_{\omega_{j+1}}| \geq U_{j+1} -  {U}_{j, j+1}(\omega_j)$. 

To see the latter, one thinks of $ \mathcal{U}_{j+1}(\omega_{j+1}) $, which has $U_{j+1}$ elements, as obtained from $\mathcal{U}_{j}(\omega_{j})$ by first forming $\mathcal{U}_{j}(\omega_{j}) /\sim_{\omega_{j+1}}$ and then removing the clusters $
C\in \mathcal{C}$ not intersecting $ V_{2(j+1)}$ from the resulting groups. As $\widetilde{\mathcal{U}} $ is formed from $\mathcal{U}_{j}(\omega_{j})$ by merging the elements of $\mathcal{U}_{j, j+1}(\omega_j)$, the quotient $\widetilde{\mathcal{U}} /\sim_{\omega_{j+1}}$ will cause at most $ {U}_{j, j+1}(\omega_j)$ of the elements in $\mathcal{U}_{j}(\omega_{j}) /\sim_{\omega_{j+1}}$ to merge, yielding the desired inequality.

\medskip
As a consequence of \eqref{Uconnect} and \eqref{Upartition}, we deduce that on $E$, there exist $k\stackrel{\textnormal{def.}}{=} \lceil \sqrt{N}/(100 \kappa L_{n_0}) \rceil$ disjoint balls $\Lambda_1, \dots, \Lambda_k$ of radius $10\kappa L_{n_0}$ centered in $\mathbb{L}_{n_0}$ and contained in $A_j$ such that each $\tilde\Lambda_\ell$ intersects both sets $\widetilde{C}_1$ and $\widetilde{C}_2$ defined in \eqref{Upartition}, where $\tilde{\Lambda}_\ell$ denotes the ball of radius $8\kappa L_{n_0}$ with the same center as $\Lambda_\ell$. One constructs the balls $\tilde\Lambda_\ell$, $1\leq \ell \leq k$, for instance as follows: consider the shells $\mathbb{S}_{\ell} \stackrel{\textnormal{def.}}{=} \partial B(V, \ell \cdot 50\kappa L_{n_0})$ with $V=V_{2j+1}$ or $V_{2j+2}$ depending on $A_j$, so that $\mathbb{S}_{\ell} \subset A_j$ for all $1\leq \ell \leq k$. Color a vertex $x\in \mathbb{S}_{\ell}$ black if $B_{L_{n_0}}(x)$ intersects $\widetilde{C}_1$ and red if it intersects $\widetilde{C}_2$. By \eqref{Uconnect}, each vertex in $\mathbb{S}_{\ell}$ is black or red (or both) and $\mathbb{S}_{\ell}$ contains at least one black and one red vertex (each possibly carrying the other color as well), as $\widetilde{C}_1$ and $\widetilde{C}_2$ both cross $\mathbb{S}_{\ell}$. In particular, there exists a pair of neighboring vertices in $\mathbb{S}_{\ell}$ carrying a different color. The ball $\tilde\Lambda_\ell$ centered at the closest vertex in $\mathbb{L}_{n_0}$ from this pair will then have the desired properties.

 Finally, we note that if $E\subset G$ happens, cf.~\eqref{eq:supercritG}, then each ball ${\Lambda}_\ell$ is good. Now, conditioning on the possible realizations $\{ \mathscr{C}\}$ of $\widetilde{\mathcal{U}}$ and applying a union bound on the partition $\{ \mathscr{C}\} = \{ \mathscr{C}\}_1 \sqcup \{ \mathscr{C}\}_2$ provided by \eqref{Upartition} (where $ \{ \mathscr{C}\}_i$ corresponds to the realization of $\widetilde{\mathcal{U}}_i$ on the event $\{ \widetilde{\mathcal{U}} =\{ \mathscr{C}\}\}$), we get, with ${C}_i = \bigcup_{ \mathscr{C} \in \{ \mathscr{C}\}_i}\text{supp}( \mathscr{C})$,
\begin{multline}
\label{eq:proofuniq1}
\P[E]\le \sum_{\{ \mathscr{C}\}} \P\big[\,\widetilde{\mathcal{U}}=\{ \mathscr{C}\}\big] \times \\ \sum_{\substack{\{ \mathscr{C}\} = \{ \mathscr{C}\}_1 \sqcup \{ \mathscr{C}\}_2 \\| \{ \mathscr{C}\}_1 |\leq a}} 
\P\Big[\bigcap_{\ell\leq k} \{\text{${\Lambda}_\ell$ is good}\} \cap \{\nlr{\Lambda_{\ell}}{\varphi\geq h-\varepsilon}{{C}_1}{{C}_2}\} ~\Bigl\vert~ \widetilde{\mathcal{U}}=\{ \mathscr{C}\} \Big].
\end{multline}
Notice that the subsets of $\Lambda_\ell$ defined by $S_i^\ell\stackrel{\textnormal{def.}}{=} B({C}_i,1)\cap {\Lambda}_\ell$, $i=1,2$, are admissible for all $1\leq \ell\leq k$. We deduce that
\begin{equation}
\label{eq:proofuniq2}
\begin{split}
&\P\Big[\bigcap_{\ell\leq k} \{\text{${\Lambda}_\ell$ is good}\} \cap \{\nlr{\Lambda_\ell}{\varphi\geq h-\varepsilon}{{C}_1}{{C}_2}\}\Bigl\vert \, \widetilde{\mathcal{U}}=\{ \mathscr{C}\} \Big]\\
& \leq  \P\Big[\bigcap_{\ell\leq k} \mathcal{G}(S_1^\ell, S_2^\ell) \cap \{\nlr{{\Lambda}_\ell}{\varphi\geq h-\varepsilon}{S^\ell_1}{S^\ell_2}\} \,\Bigl\vert\, \widetilde{\mathcal{U}}=\{ \mathscr{C}\} \Big] \\
& \leq  \P\Big[\bigcap_{\ell\leq k} \{\nlr{{\Lambda}_\ell}{\varphi\geq h-\varepsilon}{S^\ell_1}{S^\ell_2}\}\, \Bigl\vert\, \widetilde{\mathcal{U}}=\{ \mathscr{C}\}, \bigcap_{\ell\leq k} \mathcal{G}(S_1^\ell, S_2^\ell)\Big]\\
& =\prod_{\ell\leq k} \P\Big[\nlr{{\Lambda}_\ell}{\varphi\geq h-\varepsilon}{S^\ell_1}{S^\ell_2}\, \Bigl\vert\, \widetilde{\mathcal{U}}=\{ \mathscr{C}\}, \bigcap_{j\leq k} \mathcal{G}(S_1^j, S_2^j), \bigcap_{j<\ell} \{\nlr{{\Lambda}_j}{\varphi\geq h-\varepsilon}{S^j_1}{S^j_2}\}\Big]\\
& \leq (1-e^{-\Cr{c:bridging}(\log u(N))^2})^k,
\end{split}
\end{equation}
where in the last line we used Lemma~\ref{lem:sprinkling} in order to bound each one of the $k$ terms in the product by $1-e^{-\Cr{c:bridging}(\log u(N))^2}$ (one can easily check that the event in the conditioning can indeed be written as $D\cap E\cap \mathcal{G}(S_1^\ell, S_2^\ell)$, with $D\in \sigma(1_{\varphi_x\geq h};~x\in S^\ell_1\cup S^\ell_2)$ and $E\in \sigma({\mb Z}(\Lambda_\ell^c))$). Now combining \eqref{eq:proofuniq1}, \eqref{eq:proofuniq2} and the fact that, as $|\{ \mathscr{C}\}|\leq |\partial B_{4N}| \leq CN^{d-1}$, the number of partitions $\{ \mathscr{C}\} = \{ \mathscr{C}\}_1 \sqcup \{ \mathscr{C}\}_2$ with $| \{ \mathscr{C}\}_1|\leq a$ is at most $(CN^{d-1})^{a}$, we get
\begin{equation}
\label{eq:finalE}
\P[E]\le (CN^{d-1})^{a}(1-e^{-\Cr{c:bridging}(\log u(N))^2})^k \leq  \exp(-c'N^{1/4}),
\end{equation}
recalling that $k\geq c \sqrt{N}/ u(N)$ and $a \leq N^{1/4}$, as required by \eqref{eq:uniqreduct_new1}. This completes the proof of Lemma \ref{lem:uniqreduct}.
\end{proof}

\begin{remark}
\label{R:u_1}
The estimate \eqref{eq:finalE} imposes a constraint on the choice of $u(\cdot)$ in the definition of $\tilde h$ of the form $(\log u(N))^{2} \ll \log N$ (in particular, any choice $\log u(N) =(\log N)^{1/(2 +\delta)}$, $\delta > 0$ would be sufficient, but $u(N)$ needs to be subpolynomial). This constraint is indirectly caused by the lower bound derived in \eqref{eq:sprinkling}.
\end{remark}

\subsection{Renormalization}\label{subsec:renormalization}

In the sequel, consider a fixed $\varepsilon>0$ and any
$h \leq \tilde{h}-6\varepsilon$. We will eventually show that $\varphi$ strongly percolates at level $h$ to deduce Proposition \ref{prop:supercritical}. For $L_0 \geq 100$ and $x \in  \widetilde{\mathbb{L}}_0\stackrel{\textnormal{def.}}{=}L_0\Z^d$, define $x$ to be ($L_0$-)\textit{good} if (see \eqref{eq:defcross} for notation) the translate by $x$ of $\mathcal{E}(L_0,h+2\varepsilon,h+\varepsilon)$ occurs and $\sup_{B_{6L_0}(x)} |\varphi - \varphi^0 | \leq M$ with $M=M(L_0) \stackrel{\textnormal{def.}}{=} (\log L_0)^2$.
Whenever $x \in \widetilde{\mathbb{L}}_0$ is good, the set $B_{L_0}(x)$ will be called a ($L_0$)-\nolinebreak {\em good box}. Finally, let $\mathscr{S}_N$ be the $L_0$-neighborhood of the connected component of good vertices in $\widetilde{\mathbb{L}}_0(B_{2N})$ intersecting $B_{N/2}$ with largest diameter (if there is more than one such component, choose the smallest in some deterministic order). Notice that $\mathscr{S}_N$ is measurable with respect to the $\sigma$-algebra
\begin{equation}
\label{eq:sigma_F}
\begin{split}
\mathcal{F}\stackrel{\textnormal{def.}}{=} \sigma \big(
& \varphi_x-\varphi^0_x, \,1\{\varphi_x  \geq h +m\varepsilon\}, \, m=1,2, \, x\in  \Z^d \big).
\end{split}
\end{equation}
It will be important below that $\mathcal{F}$ does not completely determine $\varphi^0$, i.e.~$\varphi^0$ is not $\mathcal{F}$-measurable.

The following result asserts that with very high probability, $\mathscr{S}_N$ (under $\P$) is ubiquitous at a mesoscopic scale of order $N^{1/2}$ inside $B_N$, for all sufficiently large $N$. 

\begin{proposition}\label{prop:ubiq} For every $\varepsilon>0$, there exist constants $\rho=\rho(d)\in(0,\frac12)$, $L_0=L_0(d,\varepsilon)\geq100$ and $\Cl[c]{c:renorm}=\Cr{c:renorm}(d,\varepsilon)>0$  such that for all $h<\tilde{h}-6\varepsilon$ and $N \geq 1$,
\begin{align}
&\label{eq:uniqnew6}
\P\left[\begin{array}{c}\text{$\mathscr{S}_N$ intersects every connected set}\\ \text{$S \subset B_N$ with } \text{diam}(S) \geq N^{1/2} \end{array}\right] \geq 1- \exp(-\Cr{c:renorm}N^{\rho}).
\end{align}
\end{proposition}
\begin{proof}  Consider the event $F_{0,x}\stackrel{\textnormal{def.}}{=} F_{ x}^{(1)} \cap F_{ x}^{(2)} \cap F_{ x}^{(3)}$ defined for $x \in \widetilde{\mathbb{L}}_0$
, where
 \begin{align*}
 &F_{ x}^{(1)} =\big\{ \lr{}{\{ \varphi^{L_0} \geq h + 3 \varepsilon\}}{B_{L_0}(x)}{\partial B_{6L_0}(x)}\big\}, \\[0.3em]
 &F_{ x}^{(2)} =  \left\{\begin{array}{c}\text{all clusters in $\{\varphi^{L_0}\geq h+\frac53 \varepsilon\}\cap B_{4L_0}(x)$ crossing} \\ \text{ the annulus $B_{4L_0}(x)\setminus B_{2L_0}(x)$ are connected}\\ \text{to each other in $\{\varphi^{L_0}\geq h+\frac43\varepsilon\}\cap B_{4L_0}(x)$ } \end{array}\right\}, \\[0.3em]
 &F_{ x}^{(3)} =  \left\{ \textstyle \sup_{y\in B_{6L_0}(x)} |\varphi^{L_0}_y - \varphi^0_y | \leq  M(L_0)-1\right\}.
 \end{align*}
 These choices are motivated by \eqref{eq:F+H=good} below. The events $F_{0,x}$, $x\in \widetilde{\mathbb{L}}_0$, are typical, meaning that
\begin{equation}
 \label{eq:propuniqtohbar9}
\limsup_{L_0 \to \infty} \P[F_{0,0}] =1.
\end{equation}
Indeed, since $ h + 4\varepsilon < \tilde h-2\varepsilon$ and 
\begin{equation*}
\begin{split}
&F_{0}^{(1)}  \supset \mathcal{E}(L_0,h+4\varepsilon,h+3\varepsilon) \cap \{ \textstyle \sup_{B_{6L_0}}|\varphi-\varphi^{L_0}| \leq  \varepsilon \}, \\
&\textstyle F_{0}^{(2)} \supset \mathcal{E}(L_0,h+\frac{14}9\varepsilon,h+\frac{13}9\varepsilon) \cap \{ \textstyle \sup_{B_{6L_0}}|\varphi-\varphi^{L_0}| \leq  \frac{\varepsilon}{9} \},
\end{split}
\end{equation*}
(the first of these inclusions only requires the existence part of the event $\mathcal{E}$, the second only the uniqueness part, cf.~\eqref{eq:defcross})
 it follows using Proposition~\ref{prop:uniq} and Lemma \ref{lem:def_infty} that 
 $$\limsup_{L_0\to \infty} \P[ F_{0}^{(1)} \cap F_{0}^{(2)} ] =1.$$ The fact that $\P[ F_{0}^{(3)} ]$ tends to $1$ as $L_0 \to \infty$ follows immediately from a union bound and a standard Gaussian tail estimate, since $\E[(\varphi_0^{L_0} - \varphi_0^0)^2]$ is  bounded uniformly in $L_0$.
 
We now aim at applying Corollary \ref{cor:0bridge}. For a sequence of length scales $(L_n)_{n\geq 0}$ as in \eqref{eq:bridge1} (recall that $\ell_0$ has been fixed at the beginning of Section \ref{subsec:bridging}) and $\widetilde{\mathbb{L}}_n \stackrel{\text{def.}}{=} L_n\Z^d$, let
\begin{equation}
\label{eq:H_newnew}
H_{n,x}=\big\{\sup_{y\in B_{L_n}(x)}|\varphi_y^{L_n}-\varphi_y^{L_{n-1}}|\leq \frac{2\varepsilon}{(\pi n)^2}  \big\}, \ x \in \widetilde{\mathbb{L}}_n,\, n \geq 1.
\end{equation}
Together, \eqref{eq:H_newnew} and the definitions of $F_{ x}^{(1)}$, $F_{ x}^{(2)}$ and $F_{ x}^{(3)}$ yield that 
\begin{equation}
\label{eq:F+H=good}
\big(F_{0,x} \cap \bigcap_{n \geq 1}H_{n,x} \big)\subset \{x \text{ is good} \} , \quad x \in \widetilde{\mathbb{L}}_0, 
\end{equation}
with the notion of goodness introduced above \eqref{eq:sigma_F}.
Using Lemma~\ref{lem:def_infty} and \eqref{eq:propuniqtohbar9}, we then choose $L_0$ large enough such that the following hold: the probabilities of the events $H_{n,x}$ and of the seed events $F_{0,x}$ satisfy the bounds in \eqref{C3}. It follows that all the assumptions of Corollary~\ref{cor:0bridge} (with the notion of admissibility given by Remark \ref{R:bridges}, 2)) are in force. 
With all parameters fixed, consider any $N \geq L_1$ and choose $n\stackrel{\textnormal{def.}}{=} \max \{ k : L_{k+1} \leq N^{1/2}\}$, so that
\begin{equation}
\label{eq:uniqnew10}
\frac{N^{1/2}}{\ell_0^2} < L_n \leq \frac{N^{1/2}}{\ell_0}.
\end{equation}
We now define a random set $\tilde{\mathscr{S}}_N$ which will soon be shown to satisfy $\tilde{\mathscr{S}}_N \subset {\mathscr{S}}_N$ and to have similar connectivity properties as those required of $ {\mathscr{S}}_N$ in \eqref{eq:uniqnew6} with very high probability. For a vertex $x \in \widetilde{\mathbb{L}}_n$, we write $\text{Comp}(x)$ for the largest (in diameter) connected component in $\widetilde{\mathbb{L}}_0$ of vertices $y \in (\widetilde{\mathbb{L}}_0 \cap B(x,4\kappa L_n))$ such that $F_{0,y} \cap \bigcap_{n \geq 1}H_{n,y}$ occurs (if several such components exist, choose the one containing the smallest vertex for some given deterministic ordering of the vertices in $\widetilde{\mathbb{L}}_0$). We then set $\overline{\text{Comp}}(x)$ to be the connected component of $\text{Comp}(x)$ inside $\{ y \in \widetilde{\mathbb{L}}_0 \cap B(x, 10\kappa L_n) : \text{ $F_{0,y} \cap \bigcap_{n \geq 1}H_{n,y}$ occurs}  \}$ and define
\begin{equation}
\label{eq:uniqnew11}
\tilde{\mathscr{S}}_N = \bigcup_{x \in  \widetilde{\mathbb{L}}_n \cap B_N} \bigcup_{y \in \overline{\text{Comp}}(x)} B_{L_0}(y).
\end{equation}
Since $\tilde{\mathscr{S}}_N \subset B_{N+ 10\kappa L_n+L_0}$, it  follows on account of \eqref{eq:uniqnew10} that  $\tilde{\mathscr{S}}_N \subset B_{2N}$ whenever $N \geq C(h)$, which will be tacitly assumed.

We then consider, with $\mathcal{G}_{n,x}^0$ being the translate of $\mathcal{G}_{n}^0$ by $x$ (see \eqref{eq:bridge.GOOD_0} for the definition of $\mathcal{G}_n^0$),  
\begin{equation}
\label{eq:uniqnew12}
G_N\stackrel{\textnormal{def.}}{=} \bigcap_{x \in  \widetilde{\mathbb{L}}_n \cap B_N} \mathcal{G}_{n,x}^0 
\end{equation}
and proceed to verify that $G_N$ is contained in the event appearing on the left-hand side of \eqref{eq:uniqnew6} with $\tilde{\mathscr{S}}_N$ in place of ${\mathscr{S}}_N$. The lower bound asserted in \eqref{eq:uniqnew6} then follows by applying a union bound over $x$, using \eqref{eq:uniqnew10} and Corollary \ref{cor:0bridge}.

We now check the desired inclusion. First, by \eqref{eq:F+H=good} and \eqref{eq:uniqnew11}, $\tilde{\mathscr{S}}_N$ is the $L_0$-neighborhood of a set of good vertices. Moreover, the set $\tilde{\mathscr{S}}_N$ is connected on $G_N$, as we now explain. To this end, it suffices to argue that on $G_N$, for any two points $x,y \in\widetilde{\mathbb{L}}_n \cap B_N$ with $|x-y|=L_n$,
\begin{equation}
\label{eq:uniqnew13}
 \Big( \bigcup_{z \in \overline{\text{Comp}}(x)} B_{L_0}(z) \Big)\cap  \Big( \bigcup_{z \in \overline{\text{Comp}}(y)} B_{L_0}(z) \Big) \neq \emptyset.
\end{equation}
To see this, first note that $ \text{diam}(S_x ) \geq \kappa L_n$ where $S_x \stackrel{\textnormal{def.}}{=} \bigcup_{z \in \overline{\text{Comp}}(x)} B_{L_0}(z) $. Indeed, any two fixed opposite faces of the box $B_{\kappa L_n}(x)$ form two sets of diameter larger than $\kappa L_n$ in $B_{8\kappa L_n}(x)$, which, on account Remark~\ref{R:bridges},2), are connected on the event $G_N \subset \mathcal{G}_{n,x}^0$ by the $L_0$-neighborhood of a path consisting of vertices $v \in \widetilde{\mathbb{L}}_0 \cap B_{10\kappa L_n}(x)$ such that $F_{0,{v}} \cap \bigcap_{n \geq 1}H_{n,v}$ occurs. In particular, $\text{diam}(S_x \cap B_{4\kappa L_n}(x)) \geq \kappa L_n$. One deduces in the same way that $\text{diam}(S_y) \geq \kappa L_n$. Since both $S_x, S_y \subset B_{8\kappa L_n}(x)$, a similar reasoning using $\mathcal{G}_{n,x}^0$ implies \eqref{eq:uniqnew13}.

Now we show that every connected set $S \subset B_N$ with $\text{diam}(S) \geq N^{1/2}$ intersects $\tilde{\mathscr{S}}_N$. By \eqref{eq:uniqnew10}, $\text{diam}(S) \geq \ell_0 L_n \geq 10\kappa L_n$ (recall that $\ell_0\geq 10\kappa$), whence, by considering an $ x \in  \widetilde{\mathbb{L}}_n \cap B_N$ such that $B_{L_n}(x) \cap S \neq \emptyset$, it immediately follows that $B_{4 \kappa L_n}(x) \cap S$ has a connected component with diameter at least $\kappa L_n$. The event $\mathcal{G}_{n,x}^0$ then ensures as in the previous paragraph that $S \cap \bigcup_{y \in \overline{\text{Comp}}(x)} B_{L_0}(y) \neq \emptyset$. Thus $S$ intersects~$\tilde{\mathscr{S}}_N$. 

It is an easy consequence of the two previous paragraphs that $\mathscr{S}_N$ 
contains $\tilde{\mathscr{S}}_N$ and therefore intersects every connected set $S
\subset B_N$ with $\text{diam}(S) \geq N^{1/2}$. The claim \eqref{eq:uniqnew6} follows.
\end{proof}

\begin{remark}\label{R:hbar_alternative}
Following \cite{MR3417515}, one may define a weaker definition of $\bar h$ by considering instead of $\text{Unique}(R,\alpha)$ in \eqref{eq:UNIQUE1} the event $\text{Unique}(R,\alpha, \beta)$ which allows to connect the clusters of $\{\varphi\geq\alpha\}\cap B_{R}$ of interest with a sprinkling, i.e. in $\{\varphi\geq \beta\}\cap B_{2R}$, for $\alpha > \beta$.
 One then introduces a corresponding notion of strong percolation at levels $(\alpha, \beta)$ by requiring bounds analogous to \eqref{eq:barh1} and defines 
 $${\bar{h}}' \stackrel{\textnormal{def.}}{=} \sup\big\{h\in \R:~ \varphi \textrm{ strongly percolates at levels } (\alpha , \beta)~\textrm{for all } \beta < \alpha < h \big\}
.$$ 
Then Proposition \ref{prop:uniq} and the renormalization scheme from the proof of Proposition \ref{prop:ubiq} (using Corollary \ref{cor:0bridge}) readily yield the proof of Proposition \ref{prop:supercritical} for this alternative (and weaker) definition of $\bar h$. The additional arguments of Section \ref{subsec:supercriticalproof} are needed to deal with the (stronger) case $\alpha=\beta$.
\end{remark}

\subsection{Proof of Proposition~\ref{prop:supercritical}}\label{subsec:supercriticalproof}

In view of \eqref{eq:barh1} and \eqref{eq:barh3}, in order to conclude the proof of Proposition~\ref{prop:supercritical}, it suffices to show that there exists $c=c(\varepsilon) >0$ such that for all $h\leq \tilde{h}-6\varepsilon$,
	\begin{equation}
	\label{eq:propuniqtohbar1}
	\P[{\rm Exist}(N,h)^c] \leq \exp(-cN^{\rho})
	\quad\text{and}\quad
	\P[{\rm Unique}(N, h)^c] \leq \exp(-cN^{\rho}),
	\end{equation}
 	where $\rho$ is given by Proposition~\ref{prop:ubiq}. Fix  $h\leq \tilde{h}-6\varepsilon$ and let $A$ be the event on the left-hand side of \eqref{eq:uniqnew6}. As a consequence of the definition of good boxes, see above \eqref{eq:sigma_F}, for configurations in the event $A$, there is a cluster $\mathscr{C}$ of $\{\varphi\geq h+\varepsilon\}$ intersecting every $L_0$-box of $\mathscr{S}=\mathscr{S}_N$, which in turn intersects every connected subset of $B_N$ with diameter at least $N^{1/2}$. In particular one has $A\subset {\rm Exist}(N, h+\varepsilon)\subset {\rm Exist}(N, h)$, and the first inequality of \eqref{eq:propuniqtohbar1} follows directly from \eqref{eq:uniqnew6}. We now focus on the second one.
	
	By definition of good boxes, we have that
	$|\varphi_x-\varphi^0_x|\leq M$ for every $x\in \mathscr{S}$ (the parameter $L_0$ was fixed below \eqref{eq:F+H=good}). 
	We now argue that conditionally on $\mathcal{F}$ (defined in \eqref{eq:sigma_F}),
	the level set $\{\varphi\geq h\}$ is simply an independent site percolation with certain inhomogeneous parameters $\mathbf{p}=(\mathbf{p}_x)_{x\in\Z^d}$. Indeed, for every $x\in\Z^d$ conditionally on $\mathcal{F}$ we know the precise value of $\varphi_x-\varphi^0_x$ and that $\varphi^0_x$ lies in a prescribed interval (depending only on the value of $\varphi_x-\varphi^0_x$ and on whether $x$ is in $\{\varphi\geq h+m\varepsilon\},~m=1,2$, or not). Since $\varphi^0$ is an i.i.d.~field, the claim follows. Furthermore, on $A$ we know that $\mathscr{C}$ is in $\{\varphi\geq h+\varepsilon\}$ and $|\varphi_x-\varphi^0_x|\leq M$ for every $x\in\mathscr{S}$, so we easily infer that
	\begin{align}
\mathbf{p}_x=1 \text{ for all } x\in\mathscr{C} \text{ and } \label{eq:parameters}	&\mathbf{p}_x\geq \Cl[c]{FEiid}>0 \text{ for all } x\in\mathscr{S}, 
	\end{align}
	where 
	$$\Cr{FEiid}=\Cr{FEiid}(\varepsilon)\stackrel{\textnormal{def.}}{=} \inf_{|t|\leq M} \inf_{h \leq \tilde{h}} \P[\varphi^0_0\geq h-t ~|~ \varphi^0_0 < h+\varepsilon-t]>0.$$	
	By these observations and Proposition~\ref{prop:ubiq}, we obtain
	\begin{multline}
	\label{eq:algo11}
			\P[{\rm Unique}(N,h)^c] \leq \P[A^c]+\E\big[1_A \P[{\rm Unique}(N,h)^c | \mathcal{F}]\big] \\ \leq \exp(-\Cr{c:renorm}N^{\rho})+ \E\big[1_A {P}_{\mathbf{p}} [{\rm Unique} (N)^c]\big],
	\end{multline}
	where ${P}_{\mathbf{p}}$ represents the independent site percolation with parameters $\mathbf{p}$, and
	$
	{\rm Unique}(N)$ is the event that any two clusters in $B_{N}$ having diameter at least $N/10$ are connected to each other in $B_{2N}$. Thus, to complete the proof of \eqref{eq:propuniqtohbar1}, it is enough to show that on the event $A$,
	\begin{equation}
	\label{eq:algofinal}
	{P}_{\mathbf{p}} [{\rm Unique} (N)^c]\leq \exp(-cN^{\rho})
	\end{equation} 
	uniformly over all families of parameters $\mathbf{p}$ satisfying the properties  \eqref{eq:parameters} and all pair of sets $\mathscr{C}$ and $\mathscr{S}$ as above. To this end, first notice that
	\begin{equation} \label{eq:finalbdexplo1}
	{P}_{\mathbf{p}} [{\rm Unique} (N)^c]\leq \sum_{x\in B_N} {P}_{\mathbf{p}} \big[ \text{diam}(C(x))\geq N/10, C(x)\cap\mathscr{C}=\emptyset \big],
	\end{equation}
	where $C(x)$ denotes the cluster of $x$ in $B_N$ (under ${P}_{\mathbf{p}}$). 
	
	We now bound the summands on the right of \eqref{eq:finalbdexplo1} individually. In order to do that, we explore the cluster of $C(x)$ vertex by vertex starting from $x$ following a deterministic ordering of the vertices of $\Z^d$, i.e.~checking at each step the state of the unexplored vertex with smallest label neighboring the currently explored piece of $C(x)$. We do so until the first time $\tau_1$ we discover some vertex $y_1\in C(x)$ which lies in the outer boundary (for the $\ell^{\infty}$-norm) of some $L_0$-box $B_1\in\mathscr{S}$ (recall that $\mathscr{S}$ is determined since we have conditioned on $\mathcal{F}$; possibly $\tau_1 = \infty$). In case $\tau_1< \infty$, we explore the state of every vertex in $B_1$. By definition, $\mathscr{C}$ intersects $B_1$. We stop the exploration if at some point we discover that some vertex of $\mathscr{C} \cap B_1$ lies in $C(x)$. Otherwise we continue exploring $C(x)$ until the first time $\tau_2 (> \tau_1)$ we discover some vertex $y_2\in C(x)$ in the outer boundary of some $L_0$-box $B_2\in \mathscr{S}\setminus B_1$ which was not visited by the exploration yet. As before, we then explore the state of every vertex in $B_2$, stopping the exploration if $C(x)$ intersects $\mathscr{C}$ in that box and continuing otherwise.
	
We proceed like this until we either find that $C(x)\cap\mathscr{C}\neq\emptyset$ or we discover the whole cluster $C(x)$.
Attached to this exploration is an increasing sequence of stopping times $\tau_k$, $k \geq 1$ (by convention if $\tau_k = \infty$ for some $k \geq 1$ we set $\tau_l =\infty$ for all $l \geq k$). Let $C_k \, (\subset C(x))$ denote the set of open vertices connected to $x$ among those revealed until time $\tau_k$ union with the open vertices in $B_k$ connected to $x$. By construction, $C_k$ will intersect $\mathscr{C}$ if, for instance, all the vertices of $B_k$ belong to $C(x)$. Consequently, using \eqref{eq:parameters}, one obtains that on the event $\{ \tau_k < \infty\}$,
\begin{equation}
\label{eq:alg012}
{P}_{\mathbf{p}}[C_k\cap\mathscr{C}\neq \emptyset \, |\, \mathcal{F}_{\tau_k}] \geq \Cl[c]{c:algo_LB}(\varepsilon, L_0) \,(>0),
\end{equation}
where $\mathcal{F}_{\tau_k}$ refers to the $\sigma$-algebra of the $\tau_k$-past in the exploration process.

Let $\kappa = \sup \{ k \geq 1 : \tau_k < \infty\}$. Since $\mathscr{S}$ intersects every connected set of diameter at least $N^{1/2}$ (recall that we are on the event $A$ given by \eqref{eq:uniqnew6}, see \eqref{eq:algo11}), by considering successive annuli of width $N^{1/2}$ around $x$, one sees that on the event $\{\text{diam}(C(x))\geq N/10, C(x)\cap\mathscr{C}=\emptyset\}$ the exploration runs until fully discovering $C(x)$, and therefore $\kappa$ admits the deterministic lower bound $ \kappa \geq N^{1/2}/20\geq \lceil N^{\rho} \rceil =: k_0$, for $N \geq C$. Since $\bigcap_{1 \leq i < k} \{C_{i} \cap\mathscr{C}=\emptyset\}$ is $\mathcal{F}_{\tau_k}$-measurable, it follows by successive applications of the strong Markov property and \eqref{eq:alg012} that
\begin{multline*}
{P}_{\mathbf{p}} \big[ \text{diam}(C(x))\geq N/10,~ C(x)\cap\mathscr{C}=\emptyset \big]1_A
	\\
	\leq {P}_{\mathbf{p}} \big[C_k \cap\mathscr{C}=\emptyset, k \leq k_0, \kappa \geq k_0 \big]1_A
	\leq (1-\Cr{c:algo_LB})^{k_0},\end{multline*}
	thus giving \eqref{eq:algofinal}, as desired.
	
\begin{remark}
\label{R:exponent} As follows from \eqref{eq:propuniqtohbar1}, the exponent $\rho$ governing the rate of decay for the bound in \eqref{eq:barh1} and \eqref{eq:trunc2pt_decay}, which originates from Proposition \ref{prop:ubiq}, is in fact uniform in $h< \tilde h =\bar h$. 
\end{remark}

\section{Interpolation scheme}
\label{sec:comparison}
This section is devoted to the proof of Proposition~\ref{prop:comparison}, which will be split into several steps, as explained in the next paragraph. To lighten the notation we often use $E[X;  F]$ to denote the expectation $E[X 1_{F}]$ when $X$ is a random variable and  $F$ is an event. Throughout the whole section, we assume that $\tilde h < h_{**}$ and that $\varepsilon>0$ is chosen such that $6\varepsilon<h_{**}-\tilde h$. Constants $c,C$ may depend implicitly on $\varepsilon$ (and $d$). We recall the notation $L_n=\ell_0^nL_0$ from \eqref{eq:bridge1}, with $\ell_0$ as fixed at the beginning of Section \ref{subsec:bridging}. The parameter $L_0$ will be chosen following \eqref{eq:piv_coarse_piv2}.

We decompose this section into three subsections. In Section \ref{sec5.1}, we explain the proof of Proposition~\ref{prop:comparison} for $\delta=0$, i.e. the existence of $L=L(\varepsilon)$ large enough ($L$ will be of the form $L_n$ for some $n$) so that uniformly in $h \in (\tilde{h} + 3\varepsilon, 
	h_{**} - 3\varepsilon)$ and $r \geq 1$, $R \geq 2r$,
	\begin{align}
	 &\P[\lr{}{\varphi\geq h-\varepsilon}{B_{r}}{\partial B_{R}}]\geq \P[\lr{}{\varphi^{L}\ge h}{B_{r}}{\partial B_{R}}] - C \exp(-\e^{c(\log r)^{1/3}}),\,\mbox{and} \label{eq:comparison1delta=0} \\
	 &\P[\lr{}{\varphi\geq h+\varepsilon}{B_{r}}{\partial B_{R}}] \leq \P[\lr{}{\varphi^{L}\ge h}{B_{r}}{\partial B_{R}}] + C \exp(-\e^{c(\log r)^{1/3}}) \label{eq:comparison2delta=0},
	\end{align}
provided that one is given a certain decoupling result, see Lemma \ref{lem:piv_decoupling} below. The proof of this lemma is then presented separately in Section \ref{sec5.2}. This proof is the core of the section. It relies on a multi-scale analysis which is the most technical and innovative part of the paper. Finally, in Section \ref{sec5.3} we explain how to add the noise parameter $\delta>0$ into the game to obtain \eqref{eq:comparison1} and \eqref{eq:comparison2}.

\subsection{Setting of the proof}
\label{sec5.1}
The main difficulty in proving Proposition~\ref{prop:comparison} is the long-range dependence of $\varphi$. To overcome this problem,  we will go from $\omega = \{ \varphi \geq h\}$ to $\omega^L =  \{ \varphi^L \geq h\}$, cf.~\eqref{eq:phi^k}, step by step by interpolating between fields of \emph{comparable} ranges and allowing $h$ to vary slightly (we refer to this as the sprinkling). More precisely, extend our notation to non-integer $t$ by setting $L_t\stackrel{\textnormal{def.}}{=}L_{\ceil t}$ and
 define 
\begin{equation}\label{eq:def_chi}
\chi^t \stackrel{\textnormal{def.}}{=} \varphi^{L_{\lfloor t\rfloor}} + (t-\lfloor t\rfloor) \psi^t\quad\text{where}\quad\psi^t\stackrel{\textnormal{def.}}{=}\varphi^{L_t}-\varphi^{L_{\lfloor t\rfloor}}\,.
\end{equation}
By definition, $\chi^n=\varphi^{L_n}$ and $\chi^t$ interpolates between $\varphi^{L_n}$ and $\varphi^{L_{n+1}}$ when $n \leq t \leq n+1$. 
Introduce the function
\begin{align}
	\label{eq:theta_th}
	\theta(t,h,r,R)\stackrel{\textnormal{def.}}{=}\P[\lr{}{\chi^t\ge h}{B_r}{\partial B_R}]\,.
	\end{align}
	We now sketch the argument. Neglecting the additive error terms on the right of \eqref{eq:comparison1delta=0} and \eqref{eq:comparison2delta=0}, we roughly aim at proving that the functions 
$f_{\pm}(t)\stackrel{\textnormal{def.}}{=}\theta\big(t,h \pm Ce^{-t},r,R\big)$
are increasing (for $+$) and decreasing (for $-$), for $t$ large enough, a fact which is implied by
\begin{align}
\label{eq:comparison_truncated_partial}
|\partial_t \theta| \leq -Ce^{-t} \partial_h \theta\,.
\end{align}
Let us now look at the probabilistic 
statement that \eqref{eq:comparison_truncated_partial} 
corresponds to. The partial derivatives of $\theta$ exist for all $h$ and all non-integer $t$ and take the following form.  
\begin{lemma}\label{L:Russo}
For all $L_0 \geq 1$, $h \in \R$ and non-integer $t > 0$,
\begin{equation}
\label{eq:Russo}
\begin{split}
\partial_h \theta &= -\sum_{x \in \Z^d}\P[\piv_x | \chi^t_x = h]p_t(h),\\\quad\partial_t \theta &= -\sum_{x \in \Z^d}\E[\psi^{t}_x;\piv_x | \chi^t_x = h]p_t(h),
\end{split}
\end{equation}
where $\piv_x$ is the event that $B_r$ and $\partial B_R$ are connected in $\{\chi^t\ge h\}\cup \{x\}$ but not in 
$\{\chi^t\ge h\} \setminus \{x\}$, and $p_t(\cdot)$ denotes the density of $\chi^t_x$.
\end{lemma}

\begin{proof}
 First observe that $\chi^t$ is non-degenerate. Indeed, on account of \eqref{eq:phi^k}, \eqref{eq:white_noise_representation} and \eqref{eq:def_chi}, one has with $C(x,y)=\E[\chi^{t}_x \chi^{t}_y]$, $x,y \in \Z^d$ that, for all finitely supported $f:\Z^d\to \R$, 
\begin{equation}\label{eq:nondeg}
\langle f,Cf\rangle  \geq  \langle f,{Q}^0 f\rangle = 2^{-1} \Vert f \Vert_{\ell^2({\Z}^d)},
\end{equation}
 where ${Q}^{\ell} f ({z}) =\sum_{{z}'} {q}_{\ell}({z},{z}')f({z}')$ and we have used that $C-Q^0$ is positive semi-definite since each of ${Q}^{\ell}$, $\ell \geq 0$ is. 
 
Now define $\mathcal{S}_{\varepsilon} = \{ x \in B_R : h \leq \chi^t_x < h+\varepsilon \}$, for $\varepsilon >0$.
By non-degeneracy, as implied by \eqref{eq:nondeg}, a straightforward explicit calculation yields that $\P[h \leq \chi^t_x < h+\varepsilon , h \leq \chi^t_y < h+\varepsilon ] \leq C(h)\varepsilon^2$, hence
\begin{equation}
\label{eq:higher-order-Russo}
\varepsilon^{-1}\P[|\mathcal{S}_{\varepsilon}| \geq 2 ] \to 0 \text{ as } \varepsilon \downarrow 0.
\end{equation}
Next, abbreviating $A^h=\{ \lr{}{\chi^t\ge h}{B_r}{\partial B_R}\} $, which is decreasing in $h$ note that $(A^h \setminus A^{h+\varepsilon}) \cap \{\mathcal{S}_{\varepsilon}=x\}= \piv_x \cap \{\mathcal{S}_{\varepsilon}=x\}$. It follows that
\begin{multline} \label{eq:russo-eff}
\lim_{\varepsilon} \varepsilon^{-1}(  \P[A^h]- \P[A^{h+\varepsilon}]) \stackrel{\eqref{eq:higher-order-Russo}}{=}  \lim_{\varepsilon} \sum_x \varepsilon^{-1} \P[A^h \setminus A^{h+\varepsilon},  \, \mathcal{S}_{\varepsilon}=x ]\\
=  \lim_{\varepsilon} \sum_x \varepsilon^{-1} \P[\piv_x, \,  \mathcal{S}_{\varepsilon}=x ]  \stackrel{\eqref{eq:higher-order-Russo}}{=}  \lim_{\varepsilon} \sum_x \varepsilon^{-1} \P[\piv_x,  \, h \leq \chi^t_x < h+\varepsilon  ].
\end{multline}
Finally, observing that $\piv_x$ is measurable with respect to the $\sigma$-algebra generated by $\chi^t_y$, $y \neq x$, the first equality in \eqref{eq:Russo} follows readily from the last expression in \eqref{eq:russo-eff} upon conditioning on $\chi_x^t$ and letting $\varepsilon \to 0$, using that $\varepsilon^{-1} \P[ h \leq \chi^t_x < h+\varepsilon] \to p_t(h)$ and that the given regular conditional probability is continuous in $h$, along with a straightforward dominated convergence argument. The second formula is obtained similarly, by considering the set $\widetilde{\mathcal{S}}_{\varepsilon} = \{ x \in B_R : \chi^{t+\varepsilon}_x  \geq h>  \chi^t_x  \text{ or } \chi^{t}_x  \geq h>  \chi^{t+\varepsilon}_x \}$ instead. 
\end{proof}
A vertex $x$ for which $\piv_x$ occurs will be called {\em pivotal} in the sequel. The dependence of $\piv_x$ on the parameters $r$, $R$, $t$ and $h$ is omitted in order to lighten the notation and will always be obvious from the context. Note that the sums in \eqref{eq:Russo} are effectively over a finite set and that $(\partial_t \theta) (t,h) $ can be extended to a continuous function on any strip $\mathbb{S}_n \stackrel{\textnormal{def.}}{=}\{(t,h): h \in \R, n\leq t \leq n+1\}$, for $n \in \mathbb{N}$.

Returning to \eqref{eq:comparison_truncated_partial}, suppose for a moment that we were working with Bernoulli percolation. In this case, the pivotality at a vertex $x$ would be independent of the value of the field at $x$, so that 
\begin{align}
	\label{eq:piv_compare_ideal}
	\E[|\psi_x^t|; \piv_x \,| \chi^t_x = h]=\E[|\psi_x^t|\,| \chi^t_x = h]\,\P[\piv_x \,| \chi^t_x = h].
\end{align}
As a consequence, 
the proof would follow from the fact that $\E[|\psi_x^t|\,|\chi^t_x=h]$ is quite small and  that the quantity can be taken smaller than $\varepsilon e^{-t}$ by choosing $L$ large enough. 

In our case, the range of $\chi^t$ is $L_t$ so we must make several adjustments to the plan stated above. First of all, we might want to replace $\piv_x$ with a weaker ``coarse pivotality'' event that is supported outside $B_{L_{\ceil t}}(x)$, thus 
allowing us to achieve a {\em decoupling} with $|\psi_x^t|$ as in the 
last display. Then the task becomes, roughly speaking, to reconstruct a pivotal vertex from a coarse one. This is the content of Lemma \ref{lem:piv_decoupling} below. Its proof, which spans Section \ref{sec5.2}, will involve showing that conditionally on the coarse pivotality, the probability that there are pivotal vertices is not too small. Lemma \ref{lem:piv_decoupling} will then allow us to deduce a differential inequality similar to \eqref{eq:comparison_truncated_partial}, see \eqref{difinal} below. 

As we shall see in detail in Section \ref{sec5.2}, the estimate derived in Lemma~\ref{lem:piv_decoupling} hinges on \emph{a priori} lower bounds on the disconnection and connection probabilities 
for $\{\chi^t\ge h\}$ similar to those available for $\{\varphi \geq h\}$ when $h \in (\tilde h, h_{**})$ (hence the restriction on the value of $h$ in Proposition~\ref{prop:comparison}). Set 
\begin{multline}
\label{def:cm0}
\Cl[c]{piv_cm0}\stackrel{\textnormal{def.}}{=}\tfrac12\inf \big\{n^{d}\P[\nlr{}{\varphi \geq h}{B_{u(n)}}{B_n}],\, \P[\lr{}{\varphi \geq h}{B_{\lceil n/2 \rceil}}{\partial B_{n}}]: \\n\geq 1,\, h \in (\tilde h + \varepsilon, h_{**} - \varepsilon)\big\}, 
\end{multline}
which is strictly positive on account of \eqref{eq:h_**} and \eqref{eq:tildeh}. We introduce the convenient notation
\begin{equation}
\label{def:q_N}
q_N(t,h)\stackrel{\textnormal{def.}}{=}\inf \big\{n^{d}\P[\nlr{}{\chi^t \geq h}{B_{u(n)}}{B_n}],\, \P[\lr{}{\chi^t \geq h}{B_{\lceil n/2\rceil}}{\partial B_{n}}]: 1\le n\le N\big\}.
\end{equation} 
Note that $q_N$ is decreasing in $N$. We now state the main technical result. 

\begin{lemma}[Decoupling]
\label{lem:piv_decoupling}
For any $\varepsilon>0$, there exist positive constants $\Cl{piv_cm100}, \Cl{piv_cm1}, \Cl[c]{piv_cm2}$ and $L_0$ (depending on $\varepsilon$ and $d$ only) such that, for every $h \in (\tilde h + 2\varepsilon, h_{**} - 
2\varepsilon)$, $\Cr{piv_cm100} \le r\le R/2$ and $t , R \geq 1$ such that $q_R(t,h)\ge \Cr{piv_cm0}$, 
\begin{multline*}
\sum_{x \in \Z^d}\E[f(\psi^t_x);\piv_x\, | \chi^t_x = h]
\\ \leq \E[f(\psi^t_0) | \chi^t_0 = h] \Big(\e^{\Cr{piv_cm1}t^{18}}\sum_{x \in \Z^d}\P[\piv_x | \chi^t_x = h] +\exp[\Cr{piv_cm1}t^{3}- 
r^{\Cr{piv_cm2}}\e^{-\Cr{piv_cm1}t^{3}}]  \Big),\,
\end{multline*}
for any non-negative function $f$ such that $ \E[f(\psi^t_0) | \chi^t_0 = h] < \infty$.
\end{lemma}
\begin{remark}
\label{remark:decoupling}
The proof of Lemma~\ref{lem:piv_decoupling} entails a construction like in Section~\ref{subsec:bridging}; hence the ``additive'' error term $\exp[\Cr{piv_cm1}t^{3}- 
r^{\Cr{piv_cm2}}\e^{-\Cr{piv_cm1}t^{3}}]$ (cf. 
Lemma~\ref{lem:bridging} or Lemma~\ref{lem:sprinkling}), which has contributions from all the vertices in the annulus 
$B_R \setminus B_{r-1}$. The ``correction'' term $\e^{\Cr{piv_cm1}t^{18}}$, on the other hand, appears only to offset for the vertices close to $B_r$.
\end{remark}

We postpone the proof of this lemma until the next section and first show how to obtain \eqref{eq:comparison1delta=0} and \eqref{eq:comparison2delta=0}. 

\begin{proof}[Proofs of \eqref{eq:comparison1delta=0} and \eqref{eq:comparison2delta=0}]
Let $L_0$ be given by Lemma~\ref{lem:piv_decoupling}. Throughout the proof, we tacitly assume that $r$ and $R$ satisfy $\Cr{piv_cm100} \le r\le R/2$. The remaining cases, i.e.~$r < \Cr{piv_cm100}$, can  be accommodated by adapting the constant $C$ in \eqref{eq:comparison1delta=0} and \eqref{eq:comparison2delta=0}. Recalling the definition of $\chi^t$ from \eqref{eq:def_chi} and noting that 
\begin{equation}
\label{def:alphat'}
\var(\varphi^{L_{\lfloor t \rfloor}}_0)\geq 1/2\text{\ \ and\ \ }\var(\psi^t_0)\leq L_{t}^{-\frac{d-2}{2}}\leq \exp[-(\tfrac{1}{2}\log \ell_0)\,t]
\end{equation} (where $\ell_0\geq1000>\e^6$), a standard Gaussian bound gives
\begin{align}
\label{def:alphat}
\alpha(t)\stackrel{\textnormal{def.}}{=}\sup\{\E[|\psi_0^t|;|\psi_0^t| \geq e^{-t}\,| \chi^t_0 = h] \,:\; h\in(\tilde h,h_{**})\}\le 
C\exp(-L_t^{1/6}).
\end{align}
We can therefore fix $\Cr{t_0}(d,\varepsilon)$ large enough such that for all $t \geq \Cl{t_0}$, all $h\in(\tilde h+2\varepsilon,h_{**}-2\varepsilon)$ and every $r\ge1$,
\begin{align}
\label{eq:fix_t_1}\e^{\Cr{piv_cm1}t^{18}}\alpha(t)&\le \e^{-t} \leq \varepsilon/2,\\
\label{eq:fix_t_2}\exp[\Cr{piv_cm1}t^{3}- 
r^{\Cr{piv_cm2}}\e^{-\Cr{piv_cm1}t^{3}}]\alpha(t)p_t(h)&\le \exp[-\e^{\Cl[c]{piv_cm4}(\log r)^{1/3}}]\,\e^{-t}.
\end{align}

\medskip

Now recalling the formulas from \eqref{eq:Russo}, we can write, for any integer $n$ and $(t,h)\in \mathbb{S}_n$, 
\begin{multline}
\label{eq:derivative_bnd}
|\partial_t \theta(t,h,r,R)|
\\\leq \sum_{x \in \Z^d}\big(\,\E[|\psi_x^t|1_{|\psi_x^t| \leq e^{-t}};\piv_x \,| \chi^t_x = h]+ \E[|\psi_x^t|1_{|\psi_x^t| \geq e^{-t}};\piv_x \,| \chi^t_x = h]\,\big)p_t(h)  \\
\leq e^{-t}\,(-\partial_h \theta) + \sum_{x \in \Z^d}\E[|\psi_x^t|1_{|\psi_x^t| \geq e^{-t}};\piv_x \,| \chi^t_x = h]p_t(h).
\end{multline}
In order to bound the second term in terms of $-\partial_h \theta$ (up to an additive error), we will apply Lemma~\ref{lem:piv_decoupling} to the function $f(x) = |x| 
1_{|x|\ge  e^{-t}}$. But we are only allowed to do so as long as $q_R(t, h) 
\geq \Cr{piv_cm0}$. To this end let us define 
\begin{multline*}
t_* = t_*(R) \stackrel{\textnormal{def.}}{=} \sup\{t\ge \Cr{t_0}: q_{R}(t,h) < \Cr{piv_cm0}, \,\, \\ \mbox{for some} \,\, h \in (\tilde h + 2\varepsilon + 2\e^{-t},\, h_{**} -2\varepsilon - 
2\e^{-t})\}\end{multline*} (with the convention $\sup\emptyset= \Cr{t_0}$). Note that $t_*$ is finite since, by the weak convergence 
of $\chi^t$ to $\varphi$ on $B_R$ as $t \to \infty$, there exists $t$ such that 
$q_{R}(t,h)\ge \Cr{piv_cm0}$ for all $h \in (\tilde h + 
2\varepsilon , h_{**} -2\varepsilon)$ (see \eqref{def:cm0}). Moreover, by the definition of 
$t_*$, $q_R(t, h) \geq \Cr{piv_cm0}$ for all $t > t_*$ and $h \in (\tilde h + 2\varepsilon + 2\e^{-t},\, h_{**} 
-2\varepsilon - 2\e^{-t})$. Therefore, we can apply Lemma~\ref{lem:piv_decoupling} in this region of the $(t, h)$-plane to the function $|x|1_{|x| \geq \e^{-t}}$ to obtain (recall the definition of $\alpha(t)$ from \eqref{def:alphat})
\begin{multline*}
\sum_{x \in \Z^d}\E[|\psi_x^t|1_{|\psi_x^t| \geq e^{-t}};\piv_x \,| \chi^t_x = h]p_t(h) \\\leq \e^{{\Cr{piv_cm1}t^{18}}}\alpha(t)(-\partial_h \theta)+\exp[{\Cr{piv_cm1}}t^3- 
r^{\Cr{piv_cm2}}\e^{-{\Cr{piv_cm1}t^3}}]\alpha(t)p_t(h).\end{multline*}
Combined with \eqref{eq:derivative_bnd} as well as the bounds in \eqref{eq:fix_t_1} and \eqref{eq:fix_t_2}, this leads to the 
following differential inequalities (one for each $n\in \mathbb{N}$), which are valid for $\{(t,h): t > t_*, h \in (\tilde h + 2\varepsilon + 2\e^{-t},\, h_{**} 
-2\varepsilon - 2\e^{-t})\}\cap \mathbb{S}_n$:
\begin{align}
	\label{difinal}
|\partial_t \theta(t,h,r,R)|&\le -2e^{-t}\partial_h \theta(t,h,r,R) + \exp[-\e^{\Cr{piv_cm4}(\log r)^{1/3}}]\,\e^{-t}.
\end{align}

Integrating this family of inequalities along $\gamma_{\pm}:s  \mapsto h \pm 2(\e^{-s}-\e^{-t})$ between $(t, h)$ and $(\infty, h \mp 2 \e^{-t})$, see Fig.~\ref{F:interpolation}, where $t > t_*$ and $h \in 
(\tilde h + 2 \varepsilon + 2\e^{-t}, h_{**} - 2 \varepsilon - 2\e^{-t})$ (chopping to this effect $\gamma_{\pm}$ into its pieces intercepted by each of the strips $\mathbb{S}_n$) yields that 
	\begin{equation}
		\label{eq:comparison}
		\begin{split}
	 &\P[\lr{}{\varphi\geq h- 2\e^{-{t}}}{B_{r}}{\partial B_{R}}]\ge \P[\lr{}{\chi^{t}\ge h}{B_{r}}{\partial B_{R}}] - \exp[-\e^{\Cr{piv_cm4}(\log r)^{1/3}}]\,\e^{-t}, \\
	 &\P[\lr{}{\varphi\geq h+ 2\e^{-t}}{B_{r}}{\partial B_{R}}] \leq \P[\lr{}{\chi^{ t}\ge h}{B_{r}}{\partial B_{R}}] + \exp[-\e^{\Cr{piv_cm4}(\log r)^{1/3}}]\,\e^{-t},
	 \end{split}
	\end{equation}
for all $t > t_*(R)$.
	\begin{figure}[h!]
  \centering 
  \includegraphics[width=8cm, height=5cm]{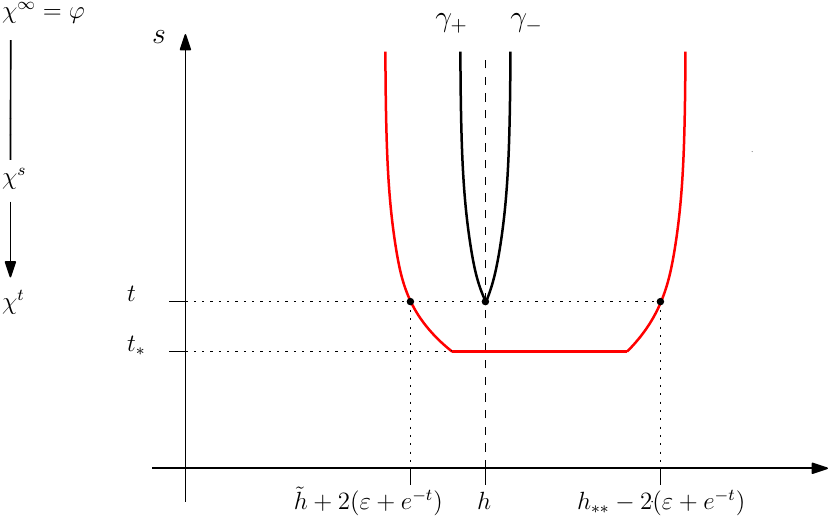}
  \caption{The two interpolation curves used in \eqref{eq:comparison}. The red curve demarcates the boundary of the region in which the family of differential inequalities \eqref{difinal} hold.}
  \label{F:interpolation}
\end{figure}

Now let $t_{**} \geq \Cr{t_0}$ be the minimum number satisfying
\begin{equation}
\label{eq:u-constraint2}
\sup\{a \geq 1: a^d\exp[-\e^{\Cr{piv_cm4}(\log u(a))^{1/3}}]\}\e^{-t_{**}} \leq \frac{\Cr{piv_cm0}}{2}\,,
\end{equation}
which, in particular, is independent of $R$. We will now argue that $t_* = t_*(R)\leq t_{**}$ for all $R$. If this holds, then since $2\e^{-t} \leq \varepsilon$ by \eqref{eq:fix_t_1}, \eqref{eq:comparison1delta=0} and \eqref{eq:comparison2delta=0} follow from \eqref{eq:comparison} by choosing $L=L_{t_{**}}$ and confining $h$ to the interval $(\tilde h + 3\varepsilon , h_{**} 
-3\varepsilon)$.

Assume on the contrary, that $t_* = t_*(R) > t_{**}$ for some $R$. Let $q_N(h)$ denote the quantity defined below \eqref{def:cm0} with $\varphi$ in place of $\chi^t$ and note that $q_N(h) \geq 2\Cr{piv_cm0}$ for all $h \in (\tilde h + \varepsilon, h_{**} - \varepsilon) \eqqcolon I$. Then \eqref{eq:comparison} and \eqref{eq:u-constraint2} together imply that uniformly over $t > t_*(>t_{**})$ and $h \in (\tilde h + 2 \varepsilon + 2\e^{-t}, h_{**} - 2 \varepsilon - 2\e^{-t})$,
\begin{align*}
	q_R(t, h)\geq \inf_{h' \in I}q_R(h') - R^d\,\exp[-\e^{\Cr{piv_cm4}(\log u(R))^{1/3}}]\e^{-t_{**}} \ge \frac{3\Cr{piv_cm0}}{2}.
	\end{align*}
On the other hand, from the definition of $t_*$ it follows that for all $t < t_*$,
\begin{align*}
\inf\{q_R(t, h):  h \in (\tilde h + 2\varepsilon + 2\e^{-t},\, h_{**} -2\varepsilon - 
2\e^{-t})\} \leq \Cr{piv_cm0}.
\end{align*}
However, the previous two displays violate the (joint) continuity of $\theta(\cdot, \cdot)$, cf.~the discussion following \eqref{eq:Russo}.
\end{proof}

\begin{remark}\label{R:u_2} The uniform bound on the error term in \eqref{eq:u-constraint2} yields a condition on the function $u(\cdot)$ entering the definition of $\tilde h$ (cf.~\eqref{eq:tildeh}) not to grow too slowly. Another such condition will arise from the competing prefactors $\e^{\Cr{C_piv_compare}(\log L_n)^2}$ and $\e^{-\Cr{c_piv_compare}u^*(L_n)^{\rho}}$ in the estimate \eqref{eq:piv_compare} below.
\end{remark}

\subsection{Proof of Lemma~\ref{lem:piv_decoupling}}
\label{sec5.2}
Throughout this subsection, we fix all the parameters $r,R,t,h$ and assume tacitly that $h \in (\tilde h + 2 \varepsilon, h_{**} - 2 \varepsilon)$, $t \geq1$,
\begin{equation}
\label{conditionr}
\Cr{piv_cm100}\stackrel{\textnormal{def.}}{=} 10^2\kappa L_0 \leq r\leq R/2
\end{equation}
 and $(t, R)$ satisfy $q_R(t,h)\ge \Cr{piv_cm0}$ 
as in Lemma~\ref{lem:piv_decoupling}, where $L_0$ will be given by Lemma~\ref{lem:piv_compare} below. 
Although they depend on $t$, we will write 
 $\chi$ and $\psi$ instead of $\chi^t$ and $\psi^t$. 
We set $T$ to be the smallest integer such that $u(L_T)\geq 20\kappa L_t$ and $\overline{T}$ the smallest integer such that $u(L_{\overline{T}})\geq 20\kappa L_T$. Note that it follows directly from the definition of $u(\cdot)$ that $T\leq C(L_0) t^3$ and $\overline{T}\leq C(L_0) t^9$.
We then define 
\begin{equation}
\label{eq:u^*cond}
u^*(L_m)\stackrel{\textnormal{def.}}{=}
\begin{cases} 
0 & \text{if $m\leq \overline{T}$},\\
u(L_m) & \text{if $m>\overline{T}$}.
\end{cases}\,
\end{equation}
The following lemma is a key step in proving Lemma~\ref{lem:piv_decoupling}. For $N\ge1$, let $\mathrm{CoarsePiv}_x(N)$ denote the event that $B_r$ and $\partial B_R$ are connected in $\{\chi\ge h\} \cup B_{N}(x)$ but 
not in $\{\chi\ge h\}$.
\begin{lemma}
	\label{lem:piv_compare}
For every $\varepsilon>0$ and $L_0 \geq \Cl{minL_0}(d,\varepsilon)$ there exist positive constants $\Cl{C_piv_compare}=\Cr{C_piv_compare}(L_0,\varepsilon,d)$, $\Cl[c]{c_piv_compare}=\Cr{c_piv_compare}(L_0,\varepsilon,d)$, $\rho=\rho(d)$ such that for all $x \in  B_R \setminus B_{r-1}$, 
\begin{multline}
\label{eq:piv_compare}
\P[\mathrm{CoarsePiv}_x(L_T)] \\ \leq \e^{-\Cr{c_piv_compare}(|x| /L_T)^{\rho}}+\sum_{n\geq T}\e^{\Cr{C_piv_compare}(\log L_n)^2-\Cr{c_piv_compare}u^*(L_n)^{\rho}}\Big(\sum_{y\in B_{10\kappa L_n}(x)}\P[\piv_y | \chi_y = h]\Big).
\end{multline}
\end{lemma}

We first assume Lemma \ref{lem:piv_compare} to hold and give the proof of Lemma~~\ref{lem:piv_decoupling}.

\begin{proof}[Proof of Lemma~\ref{lem:piv_decoupling}]

Let $L_0=\Cr{minL_0}(d,\varepsilon)$ be given by Lemma~\ref{lem:piv_compare}. All constants $C,c$ below may depend on $L_0$ (and therefore on $\varepsilon$ and $d$). Assume first that $R \geq 8L_T$. Let us introduce the event
$E_x$ that $B_{2L_t}(x)$ and $\partial B_{L_T-L_t}(x)$ are not connected in $\{\chi\ge h\}$, and 
$F_x$ the event that $B_r$ and $B_{R}$ are connected in $\{\chi\ge h\} 
\cup B_{L_T}(x)$ but not in $\{\chi\ge h\} \setminus \{x\}$, see Fig.~\ref{F:decoupling}. 

\begin{figure}[h!]
  \centering 
  \includegraphics[scale=0.5]{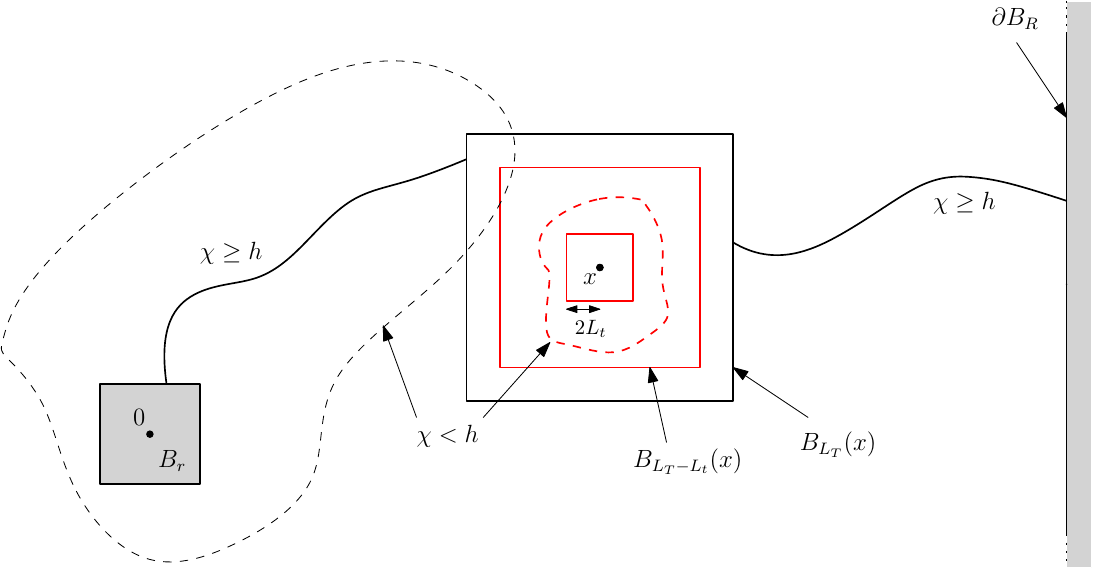}
  \caption{Decoupling in the proof of Lemma~\ref{lem:piv_decoupling}. By forcing the event $E_x$ (in red), which is independent of $ \mb Z(\L_x)$ and not too costly since $h> \tilde h$, $F_x$ and $f(\psi_x)$ decouple.}
  \label{F:decoupling}
\end{figure}

Observe that, conditionally on $\mb Z(\L_x)$ where $\L_x \stackrel{\textnormal{def.}}{=} \{x\}\cup B_{L_T}(x)^c$ (recall the definition in \eqref{Z_Lambda}), the 
event $F_x$ is decreasing in all the variables belonging to $\mb Z(\Z^d) \setminus \mb Z(\L_x)$, and so is $E_x$. From the FKG inequality for independent random variables, we deduce that
\begin{equation}
\label{eq:pivtilde_disconnect_fkg}
\begin{split}
\E[f(\psi_x);E_x\cap F_x\, |\, \mb Z(\L_x)] &= f(\psi_x)\P[E_x\cap F_x\, |\, \mb Z(\L_x)] \\
&\geq f(\psi_x)\P[E_x\, |\, \mb Z(\L_x)]\,\P[F_x\, |\, \mb Z
(\L_x)]\\
& \geq \Cr{piv_cm0}L_T^{-d}\E[f(\psi_x);F_x\, |\, \mb Z(\L_x)]\,,
\end{split}
\end{equation}
where in the final step we used the lower bound $q_R(t,h)\ge \Cr{piv_cm0}$ from the hypothesis of Lemma~\ref{lem:piv_decoupling} (note that $u(L_T-L_t) > 2L_t$ by definition of $T$) together with the fact that the event $E_x$ is independent of $\mb Z(\L_x)$ by \eqref{eq:phi_L_indep}. Since $\chi _x$ is measurable with respect to $\mb Z(\L_x)$, integrating with respect to $\mb Z(\L_x)$ gives  
\begin{align}
\label{eq:pivtilde_disconnect_fkg1}
&\E[f(\psi_x);F_x|\chi_x=h] \le \Cr{piv_cm0}^{-1}L_T^{d}\,\E[f(\psi_x);E_x\cap F_x
|\chi_x=h]\,\end{align}
(to obtain \eqref{eq:pivtilde_disconnect_fkg1}, one first integrates \eqref{eq:pivtilde_disconnect_fkg} against a set of the form $\{h\leq  \chi_x < h + \delta\}$, normalizes 
suitably and takes the limit $\delta \to 0$).
%
%
%
Since  $R - r \geq 4 L_T$ (recall that $R \geq 8L_T$ by assumption) and consequently 
$B_{L_T}(x)$ cannot intersect both $B_r$ and $\partial B_R$, the event $E_x\cap F_x$
is 
independent of $(\chi_x, \psi_x)$ as the range of $\chi$ is 
$L_t$. From this observation, we deduce that
\begin{align}\label{eq:diff_ineq3}
\begin{split}
\E[f(\psi_x);E_x\cap F_x
|\chi_x=h] &= \E[f(\psi_x)|\chi_x=h]\P[E_x\cap F_x]\\ 
&\le \E[f(\psi_x)|\chi_x=h]\P[\mathrm{CoarsePiv}_x(L_T)]\,,
\end{split}
\end{align}
where in the second step we used the fact that $E_x\cap F_x \subset\mathrm{CoarsePiv}_x(L_T)$ when 
$R - r \geq 4 L_T$. Now, Lemma~\ref{lem:piv_compare} gives that
\begin{align}
\label{eq:piv_coarse_piv2}
\sum_{x \in B_R \setminus B_{r-1}}\P[\mathrm{CoarsePiv}_x(L_T)] \leq L_T^{C} \e^{-c(r /  L_T)^{\rho}} + \e^{C(\log L_{\overline{T}})^{2}}\sum_{x\in \Z^d}\P[\piv_x | \chi_x = h],
\end{align}
where we used that (see \eqref{eq:u^*cond} for $u^*(\cdot)$) 
	\begin{align*}
	& \quad \ \sum_{n \geq T}|B_{10\kappa L_n}|\e^{{\Cr{C_piv_compare}}(\log L_n)^2-\Cr{c_piv_compare}u^*(L_n)^{\rho}} \le \e^{C(\log L_{\overline{T}})^2},\\
	& \sum_{x \in B_R\setminus B_{r-1}}\e^{-\Cr{c_piv_compare}(|x| / L_T)^{\rho}}  \leq L_T^{C} \e^{-c(r / L_T)^{\rho}},
\end{align*}
which follow by considering separately the cases $T \leq n \leq \overline{T}$ and $n \geq  \overline{T}$ in the first line and the cases $|x|\leq (L_T\vee r)$ and $|x|> (L_T\vee r)$ in the second.
Lemma~\ref{lem:piv_decoupling} now follows in the case $R \geq 8L_T$ from \eqref{eq:pivtilde_disconnect_fkg1}, \eqref{eq:diff_ineq3} and  \eqref{eq:piv_coarse_piv2} since $F_x \supset \text{Piv}_x$ and $T\leq C(L_0) t^3$, $\overline{T}\leq C(L_0) t^9$.

On the other hand if $R < 8L_T$, we simply bound
\begin{align}
\label{eq:diff_ineq1_case1}
\sum_{x \in \Z^d}\E[f(\psi_x); \piv_x\, | \chi_x = h] &\leq \sum_{x \in B_R}\E[f(\psi_x) | \chi_x = h]\le c'L_T^d\,\E[f(\psi_0)|\chi_0=h]. \end{align}
The proof is thus concluded by noting that, since $r\le4 L_T$ and  $T\leq C(L_0) t^3$, one can find $\Cr{piv_cm1}$ large enough (depending on $d,\varepsilon, L_0$ only) such that $c'L_T^d\leq \exp[\Cr{piv_cm1}t^{3}- 
r^{\Cr{piv_cm2}}\e^{-\Cr{piv_cm1}t^{3}}]$.
	\end{proof}

\bigskip

\noindent We now turn to the proof of Lemma~\ref{lem:piv_compare}. 
Roughly speaking, we would like to show that conditionally on the event $\mathrm{CoarsePiv}_x(L_T)$, some $\piv_y$ occurs in the box $B_{L_T}(x)$ with not too small probability. A natural strategy consists in trying to create paths between the clusters of $B_r$ and $\partial B_R$, which must necessarily intersect $B_{L_T}(x)$. However, the fact that the range of dependence of $\chi$ is $L_t$ presents a potential barrier for constructing these 
paths, for instance by forcing the field to be quite large. In order to poke through this barrier, we will use a good bridge -- in the sense mentioned below -- connecting the clusters of $B_r$ and $\partial B_R$ in 
$\{\chi\ge h\} \cap B_R$ so we can apply a result akin to Lemma~\ref{lem:sprinkling} (see 
Remark~\ref{remark:sprinkling} and \eqref{eq:yz} below) to construct open paths. 

To begin with, let	
\begin{equation}
\label{def:Fcompare}
\begin{split}
&F_{0, y} \stackrel{\textnormal{def.}}{=} \left\{\varphi^{L_0}_z  - \varphi^0_z \geq -M +\varepsilon,\, \varphi_z^0 \geq -M, ~~\forall z\in B_{L_0}(y)\right\}, \text{ for $y \in \mathbb{L}_0$},\\
&H_{n, y} \stackrel{\textnormal{def.}}{=} \{\varphi^{L_n}_z  - \varphi^{L_{n-1}}_z \geq -\frac{6\varepsilon}{(\pi n)^2},~~\forall z\in  B_{2L_n}(y)\}, \text{ for $n \geq 1$ and $y \in \mathbb{L}_n$,}
\end{split}
\end{equation}
with $M=M(L_0)$ chosen large enough (eg.~$M=\log L_0$) so that the bound in \eqref{C3} holds when $L_0 \geq \Cr{minL_0}(d,\varepsilon) $. We call a bridge from Definition \ref{def:bridge} \emph{good} if it satisfies Definition~\ref{def:goodbridge}, except that we only require \eqref{G2} to hold for all $j$ satisfying $1\vee m \leq j \leq n$ (which is a weaker condition). With a slight abuse of notation, we define $\mathcal G_{n, x}\stackrel{\textnormal{def.}}{=} \tau_x\mathcal{G}_n$ for $x \in \Z^d$, where $\mathcal G_n$ denotes the event from \eqref{eq:bridge.GOOD} corresponding to this weaker notion of good bridge, and with the choice $\Lambda_n = B_{10 \kappa L_n}\setminus B_{\kappa 
L_n}$ in \eqref{eq:annuli}. 
 For later reference, we also define $\mathcal{G}(S_1, S_2)$ as in \eqref{eq:defG} for any pair of admissible subsets $S_1, S_2$ of $\Lambda_n$. 

In view of \eqref{def:Fcompare} and 
Theorem~\ref{T:bridge1}, for all $L_0 \geq \Cr{minL_0}(d,\varepsilon)$ (which we will henceforth tacitly assume), there exist constants $\Cl[c]{c_piv_goodness}=\Cr{c_piv_goodness}(L_0), \rho=\rho(d) >0$ such that for every $n\geq0$ and $x\in\mathbb Z^d$, 
\begin{equation}\label{eq:nicelikely2}
\P [\G_{n, x}]\geq 1 -e^{-\Cr{c_piv_goodness}L_n^\rho}\,.
\end{equation}
We henceforth fix $L_0$ as above. All the constants $C,c$ below may depend on $L_0$, $\varepsilon$ and $d$.

 By definition, $\G_{n, x}$ guarantees the existence of a good bridge between the clusters of $B_r$ and $\partial B_R$ in $\{\chi\ge h\} \cap B_R$ {\em 
provided} they are both admissible in $\Lambda_n$, i.e.~they both intersect $\partial B_{10\kappa L_n}(x)$ as well as $B_{8\kappa L_n}(x)$ -- cf.~above \eqref{eq:bridge.GOOD}. But the latter 
condition is satisfied on the event $\mathrm{CoarsePiv}_x(8\kappa L_{n})$ 
when $B_r\not\subset B_{10\kappa L_n}(x)$. Putting these two observations together and using \eqref{def:Fcompare}, one ends up with the following lemma, whose proof is postponed for a few lines.
\begin{lemma}[Creating pivotals from coarse pivotals]
	\label{lem:good_bridge}
For all $x \in B_R\setminus B_{r-1}$ and $n \geq T$ such that $B_r\not\subset B_{10\kappa L_n}(x)$, we have
\begin{equation}
\label{eq:piv_good_bridge}
\P[\mathrm{CoarsePiv}_x(8\kappa L_{n}), { \G_{n, x}}] \leq \e^{{C}(\log L_n)^{2}} \sum_{y\in B_{10 \kappa L_n}(x)} \P[{\piv}_y | \chi_y = h]\,.
\end{equation}
\end{lemma}
In order to prove Lemma~\ref{lem:piv_compare}, we then find the \emph{first} scale $L_n$ at which the event $\G_{n, x}$ 
occurs so that we can apply the previous lemma, which is the content of Lemma \ref{lem:piv_find_goodscale} below.

For $x\in B_R\setminus B_{r-1}$, let $S_x$ denote the largest integer such that i) $B_r\not\subset B_{10\kappa L_{S_x}}(x)$, and ii) $B_{10\kappa L_{S_x}}(x)$ has empty intersection 
with at least one of $B_r$ and $\partial B_R$. Recall that we used a condition similar to ii) to derive 
\eqref{eq:diff_ineq3} (this was ensured by the assumption $R\geq 8L_T$ in the argument leading to \eqref{eq:diff_ineq3}). The quantity $S_x$ is well-defined, i.e.~$S_x \geq 0$ by condition \eqref{conditionr}. Moreover, $L_{S_x} \geq c |x|$, as can be readily deduced from the following: if $d(x,B_r) > \tfrac{|x|}{4}$, the ball $B(x,\frac{|x|}{5})$ does not intersect $B_r$, whereas for $d(x,B_r) \leq \frac{|x|}{4}$, 
the ball $B(x, \frac{|x|}{4})$ does not intersect $\partial B_R$ as $R \geq 2r$. 

%
%
\begin{lemma}[Finding the first good scale]
	\label{lem:piv_find_goodscale}
For all $x \in B_R\setminus B_r$ such that $S_x>T$, the following holds:
	\begin{align*}
		\P[\mathrm{CoarsePiv}_x(L_T)] \leq \sum_{n=T}^{S_x} \e^{-cu^{*}(L_{n})^\rho}\,\P[\mathrm{CoarsePiv}_x(8 \kappa L_{n}),{\G_{n, x}}] + \P[\G_{S_x, x}^c].
	\end{align*}
\end{lemma}

Lemma~\ref{lem:piv_compare} now follows readily by combining \eqref{eq:nicelikely2} with Lemmas \ref{lem:good_bridge} and \ref{lem:piv_find_goodscale} in case $S_x > T $ (this requires $L_0$ to be large enough, cf. above \eqref{eq:nicelikely2}), and simply bounding $\P[\mathrm{CoarsePiv}_x(L_T)]$ by 1 otherwise. The latter is accounted for by the first term on the right-hand side of  \eqref{eq:piv_compare} due to the factor $1/L_T$ appearing in the exponent and the fact that $L_{S_x} \geq c |x|$. 

\medskip

We now turn to the proofs of Lemmas~\ref{lem:good_bridge} and \ref{lem:piv_find_goodscale}.

\begin{proof}[Proof of Lemma~\ref{lem:good_bridge}]
The proof is divided into two steps. In the first step, we prove an {\em unconditional} version of \eqref{eq:piv_good_bridge}, namely
\begin{align}
\label{eq:piv_good_bridge2}
\P[\mathrm{CoarsePiv}_x(8\kappa L_{n}), { \G_{n, x}}] \leq \e^{C(\log L_n)^{2}} \sum_{y\in \Lambda_n(x)} \P[{\piv}_y, |\chi_y - \varphi^0_y| \leq M'],\,
\end{align}
where $M'\stackrel{\textnormal{def.}}{=}h+\varepsilon+M$ and $\Lambda_n(x)\stackrel{\textnormal{def.}}{=}\Lambda_n+x$. 
In the second step, we transform the unconditional probability into a conditional one: 
\begin{align}
	\label{eq:piv_uncon_con}
	\P[{\piv}_y, |\chi_y - \varphi^0_y| \leq M'] \leq C\,\P[{\piv}_y | \chi_y = h]\,. 
	\end{align}
It is clear that \eqref{eq:piv_good_bridge} follows from these two bounds as $\Lambda_n(x)\subset B_{10\kappa L_n}(x)$.

\medskip

Let us first prove \eqref{eq:piv_good_bridge2}. To this end, consider any pair of disjoint subsets $C_1$ and $C_2$ of $B_R$ such that $\mathcal C(C_1, C_2)  \stackrel{\textnormal{def.}}{=} \{ {\mathcal{C}}_{B_r} = C_1,\, {\mathcal C}_{\partial B_R} = C_2\} \subset \mathrm{CoarsePiv}_x(8\kappa L_{n})$, where $\mathcal C_A$ denotes the cluster of $A$ in $B_R\cap 
\{\chi \geq h\}$ (observe that $\mathrm{CoarsePiv}_x$ is measurable relative to the pair of random sets ${\mathcal C}_{B_r}$ and 
${\mathcal C}_{\partial B_R}$). Taking the union over all possible choices of pairs $(C_1, C_2)$ (call this collection of pairs $\mathscr C$) yields the decomposition
\begin{align}
\label{eq:containment}
\bigcup_{(C_1, C_2) \in \mathscr{C}}\mathcal C(C_1, C_2) = \mathrm{CoarsePiv}_x(8\kappa L_{n}).\,
\end{align}
By \eqref{eq:containment}, the sets $C_1$ and $C_2$ are admissible in $\Lambda_n(x)$ for any pair $(C_1, C_2) \in \mathscr{C}$ -- cf.~below \eqref{eq:nicelikely2} -- and so are $\overline{C}_1$ and $\overline{C}_2$, where $\overline{C}=(C \cup \partial_{\text{out}}C)\cap \Lambda_n(x)$, for $C\subset \Z^d$. 
We will use a good bridge to create a (closed) pivotal point $y$ in $(\partial_{\rm out}C_1 \cup \partial_{\rm out}C_2)\cap\Lambda_n(x)$, cf.~Fig.~\ref{F:reconstruction}. 

Now, similarly to the bound \eqref{eq:finite_energy_wo_sprinkle} derived in Section~\ref{sec:decompose_GFF} (recall that $\mathcal{N}(x)\stackrel{\textnormal{def.}}{=}\{y \in \Z^d: |y-x|_1 \leq 1\}$), we can prove that
\begin{equation}\label{eq:yz}
	\P\Big[\bigcup_{\substack{y,z}} \{\lr{\Lambda_n(x)\setminus (\overline C_1\cup \overline C_2)}{\chi\ge h}{\mathcal N(y)}{\mathcal N(z)}\}\cap \mathcal H_y\cap
	\mathcal H_z ~\Big\vert~ \begin{array}{c} \mathcal C(C_1, C_2),\\ \mathcal{G}(\overline{C}_1, \overline{C}_2)\end{array} \Big] \geq  e^{-C(\log  L_n)^2},
\end{equation}
where the union ranges over $y \in \partial_{\rm out}C_1$, $z \in \partial_{\rm out}C_2$ with $y,z\in\Lambda_n(x)$ and $\mathcal H_v =\{ \chi_v-\varphi_v^0\geq -M\text{ and } \varphi_v^0\geq -M\}$. We now explain the small adjustments to the proof of Lemma~\ref{lem:sprinkling} (or of \eqref{eq:finite_energy_wo_sprinkle}) needed in order to accommodate the different setup implicit in \eqref{eq:yz}. First, \eqref{eq:niceness} is replaced by the following: for each box $B = B_{L_m}(y) \in \mathcal{B}$, where $\mathcal{B}$ is any good bridge in $\Lambda_n(x)$, 
\begin{align}\label{eq:niceness2}
	\begin{split}
		&\chi_z-\varphi_z^0\geq -M\text{ and } \varphi_z^0\geq -M ~~~\forall z\in B_{L_0}(y),  ~~\text{when $m=0$},\\
		&\chi_z-\varphi_z^{L_m}\geq -\varepsilon ~~~\forall z\in \tilde{B}= B_{2L_m}(y), ~~\text{when $1\leq m\leq t$},
	\end{split}
\end{align}
as follows from \eqref{def:Fcompare} and our (weaker) version of \eqref{G2} (see below \eqref{def:Fcompare}). For each $B=B_{L_m}(y)\in \mathcal{B}$, one then redefines the event $A_B$ (see \eqref{eq:defA_B}) in the proof of Lemma~\ref{lem:sprinkling} as follows:
\begin{equation*} 
A_B\stackrel{\textnormal{def.}}{=}
\begin{cases} 
\{\lr{B_{L_0}(y)}{\varphi^{0}\geq h+ M + \varepsilon}{x_B}{y_B}\} & \text{if $m=0$},\\
\{\lr{B_{2L_m(y)}}{\varphi^{L_m}\geq h + \varepsilon}{x_B}{y_B}\} & \text{if $1\le m\le t$}, \\
\{\lr{B_{2L_m}(y)}{\chi\ge h}{x_B}{y_B}\} & \text{if $m>t$},
\end{cases}\,
\end{equation*}
and observes that, due to \eqref{eq:niceness2}, $x_B$ and $y_B$ are connected in $\{\chi\ge h\}$ whenever $B\in \mathcal{B}$ and $A_B$ occurs (the points $x_B$ and $y_B$ are chosen like in the paragraph above \eqref{eq:defA_B}). In view of the constraint $h\in(\tilde{h}+2\varepsilon, h_{**}-2\varepsilon)$ and the lower bound $q_R(t,h)\geq\Cr{piv_cm0}$ from the hypothesis of 
Lemma \ref{lem:piv_decoupling}, which are in force (see the beginning of this subsection), Lemmas \ref{lem:def_infty} and \ref{lem:twopointsbound} (see also Remark~\ref{rem:twopointsbound}) together imply that $\P[A_B]\geq  L_m^{-C'}$ for all $m \geq 1$ satisfying 
$L_m \leq R$. The rest of the proof of 
Lemma~\ref{lem:sprinkling} then follows as before, yielding~\eqref{eq:yz}. 

Rewriting \eqref{eq:yz} as an inequality involving the corresponding unconditional probabilities, using that $\G_{n, x} \subset \mathcal{G}(\overline{C}_1, \overline{C}_2)$ (see \eqref{eq:bridge.GOOD}), and subsequently 
summing over all possible choices of pairs $(C_1, C_2)\in \mathscr{C}$, we obtain
\begin{multline}\label{eq:bridging_comp1}
\P\big[\bigcup_{y, z \in \Lambda_n(x)} E(y,z)\cap \mathcal H_y\cap \mathcal H_z \big] \\ \geq\, \e^{-C(\log  L_n)^2} \sum_{(C_1, C_2) \in \mathscr C} \P[\G_{n, x},\, \mathcal C(C_1,  C_2)]\\ \stackrel{\eqref{eq:containment}}{=} \e^{-C(\log  L_n)^2}\P[\G_{n, x},\, \mathrm{CoarsePiv}_x(8\kappa L_{n})],
\end{multline}
where
\begin{equation}\label{eq:bridging_comp2}
E(y,z)\stackrel{\textnormal{def.}}{=}\big\{\lr{}{\chi\geq h}{B_r}{}\lr{}{\chi\ge h}{\mathcal N(y)}{\mathcal N(z)}\lr{}{\chi\ge h}{}{\partial B_R},\, \nlr{}{\chi\ge h}{B_r}{\partial B_R},\, \chi_y<h\big\}.
\end{equation}
Splitting into whether $z\in\partial_{\rm out}\mathcal{C}_{B_r}$ or $z\notin\partial_{\rm out}\mathcal{C}_{B_r}$, one can easily verify that $E(y,z)\subset E_1(z)\cup E_2(y,z)$, where $E_1(z)\stackrel{\textnormal{def.}}{=}\piv_z\cap\{\chi_z<h\}$  and 
$$E_2(y,z)\stackrel{\textnormal{def.}}{=}\big\{ \text{$y$ is pivotal in } \{\chi\geq h\}\cup\{z\}\big\}\cap \{\chi_y<h\}.$$ Also notice that 
$$\{\chi_z<h\}\cap\mathcal H_z \subset \{\chi_z < h, \, \chi_z - \varphi_z^0 \geq - M, \, \varphi_z^0 \geq -M \} \subset \{|\chi_z - \varphi_z^0| \leq M'\}.$$ Altogether we have 
$$E_1(z)\cap\mathcal{H}_z\subset \piv_z\cap\{|\chi_z-\varphi^0_z|\leq M'\}.$$
Therefore a simple union bound gives
\begin{align}\label{eq:bridging_comp4}
\begin{split}
&\P\Big[\bigcup_{y, z \in \Lambda_n(x)} E(y,z)\cap \mathcal H_y\cap \mathcal H_z \big\}\Big] \\
&\leq |\Lambda_n|\sum_{z\in\Lambda_(x)}\P[{\piv}_z, |\chi_z - \varphi^0_z| \leq M']
+ \sum_{y,z\in\Lambda_n(x)} \P[E_2(y,z)\cap\mathcal{H}_y\cap\mathcal{H}_z].
\end{split}
\end{align}
Now note that $E_2(y,z)\cap\mathcal{H}_y\cap\mathcal{H}_z$ is independent of $\varphi^0_z$, therefore
\begin{align}\label{eq:bridging_comp5}
\begin{split}
\P[E_2(y,z)\cap\mathcal{H}_y\cap\mathcal{H}_z]=\, &C(L_0)\P[E_2(y,z)\cap\mathcal{H}_y\cap\mathcal{H}_z\cap\{\varphi^0_z\geq M+h\}]\\
\leq\, &C(L_0)\P[{\piv}_y, |\chi_y - \varphi^0_y| \leq M'],
\end{split}
\end{align}
where $C(L_0)\stackrel{\textnormal{def.}}{=}\P[\varphi^0_0\geq M+h]^{-1}$. In the last inequality we used that $E_2(y,z)\cap\mathcal{H}_y\cap\mathcal{H}_z\cap\{\varphi^0_z\geq M+h\}\subset {\piv}_y\cap \{|\chi_y - \varphi^0_y| \leq M'\}$, which can be easily verified.
The desired inequality \eqref{eq:piv_good_bridge2} then follows directly from \eqref{eq:bridging_comp1}, \eqref{eq:bridging_comp4} and \eqref{eq:bridging_comp5} together.

\medskip
We now turn to the proof of \eqref{eq:piv_uncon_con}. Below, for a stationary Gaussian process $\Phi$ indexed by $\Z^d$, we write $p_{\Phi}$ for the density of $\Phi_0$. The key observation is that the pair $(\piv_{y}, \chi_y - \varphi^0_y)$ is independent of 
$\varphi^0_y$ ($\piv_{y}$ is measurable with respect to $\chi_z$, $z\neq y$, which is independent of $\varphi^0_y$). Consequently,
\begin{align*}
\P[\piv_y |\, \chi_y - \varphi^0_y = h_1, \varphi^0_y =  h_2] = \P[\piv_y |\, \chi_y - \varphi^0_y = h_1]
\end{align*}
for $(h_1, h_2) \in \R^2$, which leads to
\begin{align}
\label{eq:piv_uncon_con_uncon}
\P[\piv_y, |\chi_y - \varphi^0_y| \leq M'] = \int_{-M'}^{M'} \P[\piv_y |\, \chi_y - \varphi^0_y = h_1]p_{\chi - \varphi^0}(h_1) dh_1\,,
\end{align}
and also
\begin{multline}\label{eq:piv_uncon_con_con}
p_{\chi}(h)\P[{\piv}_y, |\chi_y - \varphi^0_y| \leq M' \;|\; \chi_y = h]\\
= \int_{-M'}^{M'} \P[\piv_y |\, \chi_y - \varphi^0_y = h_1] p_{\chi - \varphi^0}(h_1)p_{\varphi^0}(h - h_1) dh_1.
\end{multline}
Since $\varphi^0$ is a centered Gaussian variable, we have
$$\inf_{|h_1| \leq M' }p_{\varphi^0}(h - h_1) \geq c(L_0)>0$$
and \eqref{eq:piv_uncon_con} now follows from the displays \eqref{eq:piv_uncon_con_uncon} and \eqref{eq:piv_uncon_con_con}. This completes the proof.
\end{proof}

\begin{figure}[h!]
  \centering 
  \includegraphics[scale=0.6]{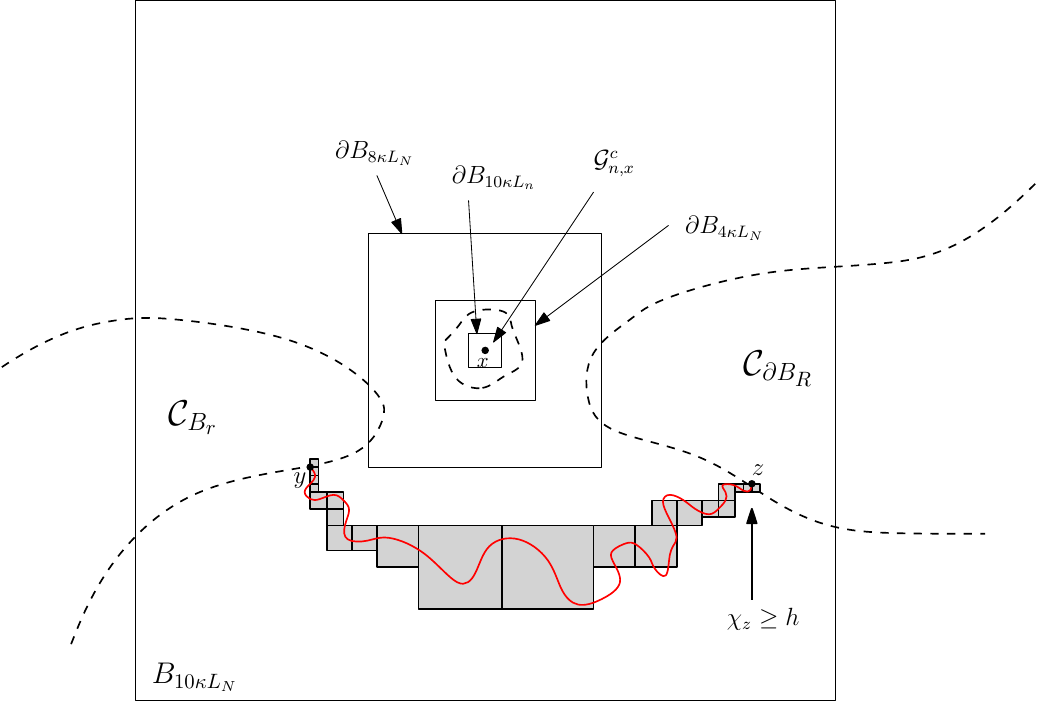}
  \caption{Finding a good scale and reconstructing. On the event $\mathrm{CoarsePiv}_x(8\kappa L_N)\cap\G_{N, x}$, one constructs a path (red) in $\{ \chi \geq h\}$ connecting the boundaries (dotted) of the clusters of $B_r$ and $B_R$. The point $z$ is flipped to open and $y$ becomes a  pivotal point (Lemma \ref{lem:good_bridge}). The occurrence of $\G_{n,x}^c$, decoupled by a dual surface, balances the reconstruction cost from the bridge (Lemma \ref{lem:piv_find_goodscale}).} 
  \label{F:reconstruction}
\end{figure}

\begin{proof}[Proof of Lemma~\ref{lem:piv_find_goodscale}]
In the proof below, we will consistently use the letters $N$ and $n$ as follows: $n$ 
is the largest integer such that $u(L_N)\ge 20\kappa L_n$. Together with the definition of $\overline{T}$ (see above  \eqref{eq:u^*cond}), this implies $n > T$ whenever $N \geq \overline{T}$.
Now, decomposing on the smallest good scale from $T$ to $S_x$ gives
\begin{align*}
	1 &= \sum_{N=T+1}^{ S_x} \big(1_{\mathcal G_{N, x}} \prod_{ k=T+1}^{  N-1}1_{\mathcal G_{k, x}^c}\big)  + \,1_{\G_{S_x, x}^c} \\ &\leq \sum_{N=T+1}^{\overline{T}-1}1_{\G_{N,x}} + \sum_{N= \overline{T}}^{S_x} 1_{\G_{N, x} \cap \G_{n, x}^c} + \,1_{\G_{S_x, x}^c},
	\end{align*}
which in turn implies
\begin{equation}
	\label{eq:coarse_to_point_decomposition}
\begin{split}
	\P[\mathrm{CoarsePiv}_x(L_T)]\leq &\sum_{N = T}^{\overline{T}-1} \P[\mathrm{CoarsePiv}_x(8\kappa L_N),\G_{N, x}] \\
	+ &\sum_{N = \overline{T}}^{S_x} \P[\mathrm{CoarsePiv}_x(8\kappa L_N),\G_{N, x},\G_{n, x}^c]+   \P[\G_{S_x, x}^c],
	 \end{split}
\end{equation}
where we also used the monotonicity of $\mathrm{CoarsePiv}_x(L)$ in $L$. In view of \eqref{eq:u^*cond}, the first sum in the right-hand side of \eqref{eq:coarse_to_point_decomposition} is accounted for in the statement of Lemma~\ref{lem:piv_find_goodscale}. As for the second sum, we now decouple the events $\G_{n,x}^c$ using a similar technique as in the proof of Lemma~\ref{lem:piv_decoupling}, cf. also Fig.~\ref{F:decoupling} and \ref{F:reconstruction}. The event $\G_{N, x}\cap \G_{n, x}^c$ is measurable relative to $\mb Z(\Lambda_x)$ where $\Lambda_x \stackrel{\textnormal{def.}}{=} B_{10\kappa L_{n}}(x) \cup B_{6\kappa L_N}(x)^c$; see the paragraph below \eqref{def:Fcompare} and \eqref{eq:annuli}--\eqref{eq:bridge.GOOD}. In particular,
\begin{align*}
&\P[\mathrm{CoarsePiv}_x(8\kappa L_N),\G_{N, x}, \G_{n, x}^c, \nlr{}{\chi\ge h}{B_{20\kappa L_{n}}(x)}{\partial B_{4\kappa L_N}(x)}\,|\, \mb Z(\Lambda_x)] \\
&= 1_{\G_{N, x} \cap \G_{n, x}^c} \P[\mathrm{CoarsePiv}_x(8\kappa L_N), \nlr{}{\chi\ge h}{B_{20\kappa L_{n}}(x)}{\partial B_{4\kappa L_N}(x)}\,|\, \mb Z(\Lambda_x)].
\end{align*}
Since our choice of $(n,N)$ guarantees that $u(L_N) \geq 20 \kappa L_n$ and $10 \kappa L_n$ exceeds the range of $\chi$, the standing assumption $q_R(t, h) \geq \Cr{piv_cm0}$ implies a lower bound of the form $\Cr{piv_cm0}L_N^{-d}$ for the above 
disconnection probability under $\P[\, \cdot \, | \,\mb Z(\Lambda_x)]$ for all $N \leq S_x$. Now, applying the same argument as for \eqref{eq:pivtilde_disconnect_fkg} with $E_x$ and $F_x$ replaced by the above disconnection and coarse pivotality events, respectively, we deduce that
\begin{multline*}
\P[\mathrm{CoarsePiv}_x(8\kappa L_N)| \mb Z(\Lambda_x)]\\\le \Cr{piv_cm0}^{-1}L_N^d\P[\mathrm{CoarsePiv}_x(8\kappa L_N),\, \nlr{}{\chi \geq h}{B_{20\kappa L_n}(x)}{\partial B_{4\kappa L_N}(x)}| \mb Z(\Lambda_x)].
\end{multline*}
Plugging this inequality into the previous display and integrating with respect to $\mb Z(\Lambda_x)$ gives
\begin{multline*}
\P[\mathrm{CoarsePiv}_x(8\kappa L_N),\G_{N, x}, \G_{n, x}^c] 
\\\leq  \Cr{piv_cm0}^{-1} L_N^{d }\,\P[\mathrm{CoarsePiv}_x(8\kappa L_N),\nlr{}{\chi\ge h}{B_{20\kappa L_{n}}(x)}{\partial B_{4\kappa L_N}(x)},\G_{N, x}, \G_{n, x}^c].
\end{multline*}
However, since $N\le S_x$ (recall its definition from the discussion preceding the statement of Lemma \ref{lem:piv_find_goodscale}), $B_{10\kappa L_N}(x)$ has empty intersection with at least one of $B_r$ or $\partial B_R$. We deduce from this fact that the event
$$\displaystyle{\{\mathrm{CoarsePiv}_x(8\kappa L_N),\,\nlr{}{\chi\ge h}{B_{20\kappa L_{n}}(x)}{\partial B_{4\kappa L_N}(x)},\,\G_{N, x}\}}$$
is measurable with respect to $\mb Z( B_{20\kappa L_{n}}(x)^c)$ and thus independent of 
$\G_{n, x}$ by  
\eqref{eq:phi_L_indep} (our modification of property \eqref{G2}, cf.~below \eqref{def:Fcompare}, is geared towards this decoupling). 
Using this observation to factorize the right-hand side of the previous displayed inequality and subsequently bounding $\P[\G_{n, x}^c]$ by \eqref{eq:nicelikely2}, we obtain for $\overline{T}\leq N \leq S_x$ that
\begin{multline*}
\P[\mathrm{CoarsePiv}_x(8\kappa L_N) \cap \G_{N, x} \cap \G_{n, x}^c]   \\ \leq \Cr{piv_cm0}^{-1}L_N^{d}\e^{-\Cr{c_piv_goodness}L_{n}^\rho}\P[\mathrm{CoarsePiv}_x(8\kappa L_N),\G_{N, x}]\\
 \leq \e^{- cL_{n}^\rho}\P[\mathrm{CoarsePiv}_x(8\kappa L_N),\G_{N, x}],
\end{multline*}
where in the final step we used the fact that $L_N \leq \exp[C(\log L_{n})^3]$. Substituting this bound into \eqref{eq:coarse_to_point_decomposition} and using that $20\kappa\ell_0 L_{n}>u(L_N)$ completes the proof. \end{proof}

\subsection{Adding the noise}
\label{sec5.3}

We extend the bounds in \eqref{eq:comparison1delta=0} and \eqref{eq:comparison2delta=0} from $\delta = 0$ to some positive $\delta$ depending {\em only} on $\varepsilon$, with $L=L(\varepsilon)$ now fixed such that the conclusions of Proposition \ref{prop:comparison} hold for $\delta=0$, see \eqref{eq:comparison1delta=0} and \eqref{eq:comparison2delta=0}. It is enough to compare $\{\varphi^L\geq h\}$ to $T_{\delta}\{\varphi^L\geq h\pm\varepsilon\}$. To this end, consider some $\delta > 0$ and define for every $t \in [0, 1]$, the percolation process $ \omega^{t,h}\stackrel{\textnormal{def.}}{=}T_{t\delta}\,\{\varphi^L \geq h \}$ (recall $T_{\delta}$ from below \eqref{eq:intro_phi_L}) as well as 
$
	\theta(t, h) \stackrel{\textnormal{def.}}{=} \P[\lr{}{\omega^{t,h}}{B_r}{\partial B_R}],
$
for $r\le R/2$. The analogue of \eqref{eq:Russo} in this case reads
\begin{align*}
&\partial_h \theta = -(1 - \delta)\sum_{x \in \Z^d}\P[\piv_x | \varphi^L_x = h]p(h)\ \text{ and } \\ 
&\partial_t \theta = \frac{\delta}{2}\sum_{x \in \Z^d}\big(\P[\piv_x, \omega^{t,h}_x=0] - \P[\piv_x, \omega^{t,h}_x=1]\big),\,
\end{align*}
where $p(\cdot) $ is the density of $\varphi^L_0$. Notice that 
\begin{align}
\label{eq:deriv_noise}
|\partial_t \theta|	 \leq \frac{\delta}{2}\sum_{x \in \Z^d}\P[\piv_x]\,.
	\end{align}	
Define $\tilde{c_7}$ and $\tilde{q}_N(t,h)$ similarly to \eqref{def:cm0} and \eqref{def:q_N}, but replacing $\varphi$ by $\varphi^L$ and $\{\chi^t\geq h\}$ by $\omega^{t,h}$, respectively. One then follows the proof of Lemma~\ref{lem:piv_decoupling} -- which is actually slightly simpler -- to obtain that, under the condition that $\tilde{q}_R(t,h)\geq\tilde{c_7}$, the following inequality holds
\begin{align*}
\sum_{x \in \Z^d}\P[\piv_x]
\leq C(L)\Big(\sum_{x \in \Z^d}\P[\piv_x | \varphi^L_x = h] +\exp[-c(L)r^{\Cr{piv_cm2}}] \Big).
\end{align*}
Then, the proof follows similar lines of reasoning as the proof of  \eqref{eq:comparison1delta=0} and \eqref{eq:comparison2delta=0} from Lemma~\ref{lem:piv_decoupling} at the end of Section \ref{sec5.1}, choosing the prefactor $\delta$ appearing in \eqref{eq:deriv_noise} suitably small (recall that $L=L(\varepsilon)$ is fixed) to obtain an analogue of the differential inequality \eqref{difinal}. We omit further details.

\section{Proof of Proposition~\ref{prop:sharptruncated}}
\label{sec:sharpness}
Throughout this section, for a fixed $L\geq0$ and $\delta \in (0, 1)$, we set 
\begin{equation}
\omega= \{ \omega_h : h \in \R \}, \text{ where }
\omega_h\stackrel{\textnormal{def.}}{=} T_\delta \{\varphi^L\ge h\}, \ h \in \R
\end{equation}
(recall $T_{\delta}$ from the paragraph preceding the statement of Proposition~\ref{prop:sharptruncated}). To be specific, we assume that $\omega$ is sampled 
in the following manner. There exists a collection of i.i.d.~uniform random variables ${\rm U} =\{{\rm U}_x: x \in \Z^d\}$ independent of the process~$\mb Z$ (recall the definition from Section~\ref{subsec:decomposition}) under $\P$ such that, for $h \in \R$,
\begin{equation}
\label{def:tempered}
\omega_h(x) = \begin{cases}
0 &  \mbox{if } {\rm U}_x \leq \delta/2\,,\\
1_{\varphi_x ^L \geq h} & \mbox{if } {\rm U}_x \in (\delta/2, 1 - \delta/2)\,,\\
1 & \mbox{if } {\rm U}_x \geq 1 - \delta/2.\,
\end{cases}
\end{equation}

 We proceed to verify that $\omega$ satisfies the following properties:
 \begin{enumerate}[(a)]
\item {\em Lattice symmetry. }  For all $h \in \R$, the law of $\omega_h$ is invariant with respect to translations of $\Z^d$, reflections with respect to hyperplanes and rotations by $\pi/2$.

\item {\em Positive association. }  For all $h \in \R$, the law of $\omega_h$ is positively associated, i.e.~any pair of increasing events satisfies the FKG-inequality.

\item {\em Finite-energy. } There exists $c_{{\rm FE}} \in (0,1)$ such that 
for any $h \in \R$, $$\P[\omega_h(x) = 0 \,|\, \omega_h(y) \mbox{ for all }y \neq x] \in (c_{{\rm FE}}, 1 - c_{{\rm FE}})\,.$$

\item {\em Bounded-range i.i.d.~encoding. } Let $\{{\rm V}(x) : x \in \Z^d\}$ denote a family of i.i.d.~random variables (e.g.~uniform in $[0,1]$). Then, for every $h \in \R$, there exists a (measurable) function $g=g_h: \R^{\Z^d} \to \{0, 1\}^{\Z^d}$ such that the law of $g(({\rm V}(x))_{x \in \Z^d})$ is the same as that of $\omega_h$ and, for any $x \in \Z^d$, $g_x((v_y)_{y \in \Z^d})$ depends only on $\{v_y: 
y \in B_{L}(x)\}$. Thus, in particular, $\omega_h$ is an $L$-dependent process.

\end{enumerate}

Property~(a) is inherited from corresponding symmetries of the laws of $\rm U$ and $\varphi^L$. In view of \eqref{def:tempered}, \eqref{eq:white_noise_representation} and \eqref{eq:phi^k}, $\omega_h$ is an increasing function of the independent collection $({\rm U}, \mathcal{Z})$, and Property~(b) follows by the FKG-inequality for independent random variables. Still by  \eqref{def:tempered}, \eqref{eq:white_noise_representation} and \eqref{eq:phi^k}, Properties~(c) and~(d) hold: for the former, take $c_{{\rm FE}}=\delta/2$; for the latter one can use ${\rm V}(x)$ to generate the independent random variables ${\rm U}_x$ and $\mathcal{Z}_{\ell} (\tilde z)$, for all $0\leq \ell \leq L$ and $\tilde{z}\in \{ x, x+\frac12 e_i, 1\leq i \leq d\}$ (this can always be arranged, e.g.~by considering the dyadic expansion of ${\rm V}(x)$ and looking at the digits corresponding to residue classes modulo $M$ for suitably large~$M$, which give rise to $M$ independent uniform random variables).
We will use another property of $\omega_h$, whose proof is more involved. We therefore state it as a separate lemma.

\begin{lem} The field $\omega$ satisfies the following: 
{\rm \begin{enumerate}
\item[(e)] {\em Sprinkling property. }For every pair $h < h' \in \R$, 
there exists $\varepsilon =\varepsilon(h,h') > 0$ such that 
$\omega_h$ {\em stochastically 
dominates} $\omega_{h'} \vee \eta_\varepsilon$, where $\eta_{\varepsilon}$ is a Bernoulli percolation with density $\varepsilon$ independent of 
$\omega_{h'}$. Henceforth we will denote this domination by $\omega_h \succ \omega_{h'} \vee \eta_{\varepsilon}$.
\end{enumerate}}
\end{lem}

\begin{proof}
We assume by suitably extending the underlying probability 
space that there exists $\eta = \{ \eta_{\varepsilon} :\varepsilon \in (0,1)\}$
with $\eta$ independent of $\omega$ and $ \eta_{\varepsilon} $ distributed as i.i.d.~Bernoulli variables with density $\varepsilon$. We will progressively replace the threshold $h$ with $h'$ in a finite 
number of steps. To this end we claim that for any $\kappa > 0$ 
there exists $\varepsilon > 0$ such that for any $\mathsf h  \in [h, h']$ and $\mathbb{L} \in \mathcal L$, where $\mathcal L$ comprises all the (finitely many) translates of the sub-lattice $10L\,\Z^d$, we have
\begin{align}
	\label{eq:weak_domination}
	\omega_{\mathsf h} \succ \omega_{\mathsf h + \kappa} \vee \eta_{\varepsilon}^{\mathbb{L}},
\end{align}
where $\eta_{\varepsilon}^{\mathbb{L}}(x)= \eta_{\varepsilon}(x)$ if $x \in \mathbb{L}$ and $\eta_{\varepsilon}^{\mathbb{L}}(x)=0$ otherwise. Let us first explain how to derive 
property~(e) from \eqref{eq:weak_domination}: choosing $\kappa \stackrel{\textnormal{def.}}{=} \tfrac{h' - h}{|\mathcal L|}$  gives $\omega_{h} \succ \omega_{h + \kappa} \vee 
\eta_{\varepsilon}^{\mathbb L}$, for suitable $\varepsilon 
=\varepsilon (h,h',L,d)>0$ and any choice $\mathbb{L}$, so that  iterating this over the lattices in $\mathcal L$ gives
\begin{equation*}
	\omega_{h} \succ \omega_{h + |\mathcal L| \kappa} \vee  \bigvee_{\mathbb L \in \mathcal L} \eta_{\varepsilon}^{\mathbb L} = \omega_{h'} \vee \eta_{\varepsilon},\,
\end{equation*}
as desired. 

By approximation, it suffices to verify \eqref{eq:weak_domination} for all fields restricted to a finite set $\L \subset \Z^d$. Let $\omega_h(S)\stackrel{\textnormal{def.}}{=} \{\omega_h(x): x \in S \}$
for $S\subset \Z^d$ and define similarly $ \eta_\varepsilon(S)$,  $\eta_\varepsilon^{\mathbb{L}}(S)$. Notice that $\omega_{\mathsf h}=\omega_{\mathsf 
h+\kappa} \vee \omega_{\mathsf h }$ and $\omega_{\mathsf h + 
\kappa} \vee \eta_\varepsilon^{\mathbb{L}}$ are both increasing functions 
of $(\omega_{\mathsf h + \kappa}, \omega_{\mathsf h})$ and 
$(\omega_{\mathsf h + \kappa}, \eta_\varepsilon^{\mathbb{L}})$ respectively. 
Therefore it suffices to show that, conditionally on any realization of $\omega_{\mathsf h + \kappa}$, the field $\omega_{\mathsf h}(\L)$ 
stochastically dominates $\eta_{\varepsilon}^{\mathbb{L}}(\Lambda)$. To this end we fix an ordering $\{x_1, x_2, \dots\}$ of the 
vertices in $\L$ such that all the vertices in $\mathbb{L} \cap \L$ appear 
before all the vertices in $\mathbb{L}^c \cap \L$. In view of \cite[Lemma~1.1]{LigSchSta97}, it then suffices to show that for any $k\geq 1$, with $\Lambda_k= \{ x_1,\dots, x_{k-1}\}$ ($\Lambda_0 =\emptyset$),
\begin{align}
	\label{eq:weak_domination_cond}
	\P \text{-a.s., } \ \ \P[\, \omega_{\mathsf h}(x_{k }) = 1 \,|\, A,\, A'\,] \geq \varepsilon \cdot \mathbf{1}_{x_{k} \in \mathbb{L}}, \,\,\,\, 
\end{align}
where $A\stackrel{\textnormal{def.}}{=}\{\omega_{\mathsf h}(x)=\sigma(x) , x \in \Lambda_k\}$, $A'\stackrel{\textnormal{def.}}{=}\{  \omega_{\mathsf h + \kappa}(x) = \sigma'(x),x \in \Lambda\}$, for arbitrary $\sigma, \sigma' \in \{ 0,1\}^{\Z^d}$ with $\sigma \geq \sigma'$ (as $\omega_{\mathsf h} \geq \omega_{\mathsf h + 
\kappa}$). We now show  \eqref{eq:weak_domination_cond} and assume that $x_{k} \in \mathbb{L}$ (the other cases are trivial). 
To this end let us first define, for any $x \in \Z^d$, the pair of events
\begin{equation}
\label{eq:Ulastsection}
\mathcal{T}(x) \stackrel{\textnormal{def.}}{=}  \{{\rm U}_x \notin (\delta/2, 1 - \delta/2)\}\ \,\,\,\mbox{and}\,\,\,\  \mathcal U(x) \stackrel{\textnormal{def.}}{=} \bigcap_{y \in B_L(x) \setminus \{x\}} \mathcal{T}(y).
\end{equation}
Notice that, in view of {\eqref{def:tempered}}, $\mathcal{T}(x)$ corresponds to the event that the noise at $x$ is \textit{triggered}. We then write, with $\mathcal{U}=\mathcal{U}(x_k)$,
\begin{equation}
\label{eq:cond_prob_decomp}
\begin{split}
\P[&\omega_{\mathsf h}(x_{k }) = 1 \,|\, A, A'\,] \geq \P[\,\omega_{\mathsf h}(x_{k }) = 1 \,|\, \mathcal{U},\, A,\, A'\,] \cdot \P[\, \mathcal{U} \,|\,A,\, A'\,]. 
\end{split}
\end{equation}
We will bound the two probabilities on the right-hand side separately from below. It follows from the definition of $\omega_h$ in \eqref{def:tempered} that $\omega_{\mathsf h + \kappa}(y) =\zeta(y) $ on the event $\mathcal{T}(y)$, where $  \zeta(y)= 
1_{{\rm U}_y \geq 1 - \delta/2}$. Consequently, decomposing $A'$, we obtain, with $A'_1\stackrel{\textnormal{def.}}{=}\{ \omega_{\mathsf h + \kappa}=\sigma' \text{ on }  \L \cap B_L(x_{k })^c\}$ and $A'_2\stackrel{\textnormal{def.}}{=}\{\zeta=\sigma' \text{ on }  B_L(x_{k }) \setminus \{x_{k}\}\}$, that
\begin{multline*}
p \stackrel{\textnormal{def.}}{=}  \P[\,\omega_{\mathsf h}(x_{k }) = 1 \,|\, \mathcal{U},\, A,\, A'\,] \\= \P[\,\omega_{\mathsf h}(x_{k }) = 1 \,|\, \mathcal{U},\, A,\, \omega_{\mathsf h + \kappa}(x_{k })=\sigma'(x_k),\, A_1', A_2'\,]\,.
\end{multline*}
However, in view of definition \eqref{def:tempered}, \eqref{eq:Ulastsection} and the fact that $\mathbb{L}$ is $10L$-separated, it follows that both events $\{\omega_{\mathsf h}(x_{k }) = 1, \omega_{\mathsf h + \kappa}(x_{k })=\sigma'(x_k)\}$ and $\{ \omega_{\mathsf h + \kappa}(x_{k })=\sigma'(x_k) \}$ are independent of $\mathcal{U} \cap A \cap A_1'\cap A_2'$, leading to 
\begin{equation*}
p = \P[\omega_{\mathsf h}(x_{k }) = 1 \,|\,\omega_{\mathsf h + \kappa}(x_{k }) =\sigma'(x_{k})]\,.
\end{equation*}
Now, $p = 1$ when $\sigma'(x_{k }) = 1$. On the other hand, when $\sigma'(x_{k }) = 0$, we have, using \eqref{def:tempered}, \eqref{eq:Ulastsection},
\begin{align*}
p &\geq \P[\omega_{\mathsf h}(x_{k }) = 1 \,|\,\mathcal{T}(x_k)^c, \omega_{\mathsf h + \kappa}(x_{k }) = 0]\,\P[\mathcal{T}(x_k)^c \,|\, \omega_{\mathsf h + \kappa}(x_{k }) = 0]\\
&= \P[\varphi_{x_{k}}^L \geq \mathsf h \,|\varphi_{x_{k}}^L < \mathsf h + \kappa\,]\,\P[\mathcal{T}(x_k)^c] \, \P[ \varphi_{x_{k}}^L < \mathsf h + \kappa]/ \P[\omega_{\mathsf h + \kappa}(x_{k }) = 0]  \ge \varepsilon', 
\end{align*}
where $\varepsilon'>0$ depends only on $\delta, \mathsf{h}, \kappa, L$ and $d$.  As to bounding the second term on the right of \eqref{eq:cond_prob_decomp}, we write, with $A''(\rho)\stackrel{\textnormal{def.}}{=}\{1_{\varphi_{y}^L \geq \mathsf h+\kappa}=\rho(y), y \in B_L(x_k) \setminus \{x_k\} \}$,
\begin{multline*}
 \P[\, \mathcal{U} \,|\,A,\, A'\,]  \geq \inf_{\rho}  \P[\, \mathcal{U} \,|\,A,\, A',\, A''(\rho)]\\
 = \inf_{\rho} \prod_{y \in B_L(x_{k}) \setminus \{x_{k}\}}\P[\, \mathcal{T}(y) \,|\,1_{\varphi_y^L \geq \mathsf h + \kappa} =\rho(y), \omega_{\mathsf h + \kappa}(y)=\sigma'(y)] \ge (\delta / 2)^{|B_L|}
\end{multline*}
where the equality in the second line follows by \eqref{def:tempered} and \eqref{eq:Ulastsection} upon conditioning on $\omega_{\mathsf h}(x)$, $x\in \Lambda_k$, $\omega_{\mathsf h + \kappa}(y)$, $y \in (\Lambda \setminus B_L(x_k))\cup\{x_k\}$ and $1_{\varphi_y^L \geq \mathsf h + \kappa}$, $y \in  B_L(x_k) \setminus \{x_k\}$, and the final lower bound follows by distinguishing whether $\rho(y)\neq \sigma'(y)$ (in which case the given conditional probability equals $1$) or not, and using \eqref{def:tempered}. Overall, the right-hand side of \eqref{eq:cond_prob_decomp} is thus bounded from below by $\varepsilon \stackrel{\textnormal{def.}}{=} \varepsilon' (\delta / 2)^{|B_L|}$, which implies \eqref{eq:weak_domination_cond} and completes the verification of~(e).
\end{proof}


\medskip

The rest of this section is devoted to deriving Proposition~\ref{prop:sharptruncated} from Properties~(a)--(e). We prove in two parts that $\tilde 
h(\delta,L) \ge h_*(\delta,L)$ and $h_*(\delta,L) = h_{**}(\delta,L) \ge \tilde h(\delta,L)$ which together imply $\tilde h(\delta, 
L) = h_*(\delta,L) = h_{**}(\delta,L)$, as asserted. 

\medskip

We first argue that $\tilde h(\delta,L) \ge h_*(\delta,L)$. As a consequence of Properties~(a)--(c) and (e), one can adapt the argument in \cite{GriMar90} in the context of Bernoulli percolation --- see also \cite[Chap.~7.2]{Gri99a}  --- to deduce that $\omega_h$ percolates in ``slabs'', i.e. for every $h < h_*(\delta,L)$, there exists 
$M \in \N$ such that 
\begin{align}
\label{eq:slab}
\P\big[\lr{\Z^2 \times \{0,\dots,M\}^{d-2}}{\omega_h}{0}{\infty}\big] > 0 .
\end{align}
Now, fix $h<h_*(\delta,L)$ and $M$ such that the above holds. Set $\mathbb S\stackrel{\textnormal{def.}}{=}\Z^2 \times \{0,\dots,M\}^{d-2}$ and $x_i\stackrel{\textnormal{def.}}{=}(0,\dots,0,(M+L)i)$, and observe that since the range of $\omega_h$ is $L$, it follows that for every $R \geq C(h)$, 
\begin{align*}
\P[\nlr{}{\omega_h}{B_{u(R)}}{B_R}] \stackrel{}\leq \prod_{i<u(R)/(M+L)}\P[\nlr{x_i+\mathbb S}{\omega_h}{x_i}{\infty}] \stackrel{\eqref{eq:slab}}{\leq}\P[\nlr{\mathbb S}{\omega_h}{0}{\infty}]^{u(R)/(M+L)} \,.
\end{align*}
But this implies $h \le \tilde h(\delta,L)$, as desired.

\medskip

The proof of $h_*(\delta,L) = h_{**}(\delta,L) \ge \tilde 
h(\delta,L)$ on the other hand follows directly from the exponential decay of $\omega_h$ in the subcritical regime. More precisely, we claim that for every $h > h_*(\delta,L)$, there exists $c = c(h)> 0$ such that  
\begin{align}
\label{eq:exp_decay}
\P[\lr{}{\omega_h}{0}{\partial B_R}] \leq \exp(-c R), \quad \text{for every $R\ge0$},
\end{align}
which implies that $h_*(\delta,L) = h_{**}(\delta,L)$, cf. \eqref{eq:h_**} (recall that $ h_{**}(\delta,L) \ge \tilde 
h(\delta,L)$, as explained below \eqref{eq:tildeh} for $\delta=0$ and $L=\infty$). We therefore focus on the proof of \eqref{eq:exp_decay}. For each $h\in\mathbb R$, consider the family of processes $\gamma_{\varepsilon} \stackrel{\textnormal{def.}}{=} \omega_h \vee \eta_{\varepsilon}$ indexed by $\varepsilon \in 
[0, 1]$ where $\eta_{\varepsilon}$ is a Bernoulli percolation with density $\varepsilon$ independent of $\omega_h$. Let $\varepsilon_c=\varepsilon_c(h) \in [0, 
1]$ denote the critical parameter of the family of percolation processes 
$\{\gamma_\varepsilon : \varepsilon \in [0, 1]\}$ and suppose for a moment that for every $0 \leq \varepsilon < \varepsilon_c(h)$, there exists $c=c(h,\varepsilon)>0$ such that for every $R\ge1$,
\begin{equation}
\label{eq:subcritical_ep}
\theta_R(\varepsilon) \stackrel{\textnormal{def.}}{=} \P[\lr{}{\gamma_{\varepsilon}}{0}{\partial B_R}] \leq \exp(-cR).
\end{equation}
Then \eqref{eq:exp_decay} follows immediately by taking $\varepsilon=0$ in case $\varepsilon_c (h)>0$ for all $h>h_*(\delta,L)$. But the latter holds by the following reasoning. From Property (e), we see that whenever $h>h_*(\delta,L)$, choosing $h'= (h_*(\delta,L)+h)/2$, one has $\omega_{h'} \succ \omega_{h} \vee 
\eta_{\varepsilon'}$ for some $\varepsilon'>0$, whence $ \varepsilon' \leq \varepsilon_c(h)$ as $\omega_{h'}$ is subcritical. The remainder of this subsection is devoted to proving \eqref{eq:subcritical_ep}. 

For this purpose, we use the strategy developed in the series of papers \cite{DumRaoTas17b, 
DumRaoTas17a, DumRaoTas17c} to prove subcritical sharpness using decision trees. This strategy consists of two main parts. In the first part, one bounds the variance $\theta_R(1 - 
\theta_R)$ using the OSSS inequality from \cite{OSSS} by a weighted sum of influences where the weights are given by revealment probabilities of a randomized algorithm 
(see \eqref{eq:reveal} below). Then in a second part one relates these influence terms to the derivative of $\theta_R(\varepsilon)$ with respect to $\varepsilon$ to deduce a system of differential inequalities of the following form:
\begin{equation}
\label{eq:osssdiff_ineq}
\theta_R' \geq  \beta\,\frac{R}{\Sigma_R}\,\theta_R(1 - \theta_R),
\end{equation}
where $\Sigma_R \stackrel{\textnormal{def.}}{=} \sum_{r = 0}^{R-1}\theta_r$ and $\beta=\beta(\varepsilon):(0,1)\rightarrow(0,\infty)$ is a continuous function. Using purely analytical arguments, see \cite[Lemma~3.1]{DumRaoTas17b}, one then obtains \eqref{eq:subcritical_ep} for $\varepsilon<\varepsilon_c$.

In our particular context, we apply the OSSS inequality for product measures (see \cite{DumRaoTas17b} for a more general version) and for that we use the encoding of $\omega_h$ in terms of $\{{\rm 
V}(x): x \in \Z^d\}$ provided by Property~(d). In view of this, $1_{\mathcal E_R}$ with $\mathcal E_R = \{\lr{}{\gamma_{\varepsilon}}{0}{\partial B_R}\}$ can now be written as a function of independent 
random variables $({\rm V}(x): x \in B_{R + L})$ and $(\eta_\varepsilon(x): x \in B_{R})$. Using a randomized algorithm very similar to the one used in  \cite[Section~3.1]{DumRaoTas17c} (see the discussion after the proof of Lemma~\ref{lem:influence_compare} below for a more detailed exposition), we get
\begin{equation}
\label{eq:OSSSvarbnd_avgpart}
\sum_{x \in B_{R+ L}} \Inf_{{\rm V}(x)} + \sum_{x \in B_{R}} \Inf_{\eta_{\varepsilon}(x)} \geq {c}L^{-(d-1)}\frac{R}{\Sigma_R}\,\theta_R(1 - \theta_R) \,,
\end{equation}
where
$
	\Inf_{{\rm V}(x)} \stackrel{\textnormal{def.}}{=} \P[1_{\mathcal E_R}({\rm V}, \eta_{\varepsilon} ) \neq 1_{\mathcal E_R}({\tilde {\rm V}}, \eta_{\varepsilon} )]\,$
with $\tilde V$ being the same as $V$ for every vertex except at $x$ where it is {\em resampled 
independently}, and $\Inf_{{\eta}_\varepsilon(x)}$ is defined similarly. In order to derive \eqref{eq:osssdiff_ineq} from \eqref{eq:OSSSvarbnd_avgpart}, we use the following lemma.
\begin{lemma}
	\label{lem:influence_compare}
	There exists a continuous function $\alpha:(0, 1)\rightarrow(0,\infty)$ such that for all $R \geq 1$,
	\begin{align}
		\label{eq:influence_compare}
\theta'_R(\varepsilon)	\geq \alpha(\varepsilon) \,\Big(\sum_{x \in B_{R + L}} \Inf_{{\rm V}(x)} + \sum_{x \in B_{R}} \Inf_{\eta_{\varepsilon}(x)}\Big).
		\end{align}
	\end{lemma}
\begin{proof}
Since the derivative is with respect to the parameter of the 
Bernoulli component of the process, it follows from standard computations that
\begin{align*}
\theta_R' = \P [\omega_h(0) = 0] \sum_{x \in B_R} \P[\piv_x]\,.
\end{align*}
where $\piv_x$ is the same event as in Section~\ref{sec:comparison} (see below \eqref{eq:Russo}) with $\chi^t$ replaced by $\gamma_{\varepsilon}$. Now, on the one hand, $	
	\Inf_{\eta_{\varepsilon}(x)}  \leq  \P[\piv_x]$. 
On the other hand, since ${\rm V}(x)$ affects the states of vertices only in $B_{L}(x)$ by Property~(d), one immediately gets
$	
	\Inf_{V(x)}  \leq  \P[\piv(B_{L}(x))]$,
where $\piv(B_{L}(x))$ -- called the {\em pivotality of the box} $B_{L}(x)$ -- is the event that $0$ is connected to $\partial B_R$ in $\gamma_\varepsilon\cup B_{L}(x)$ where as it is not in $\gamma_\varepsilon\setminus B_{L}(x)$. In fact, due to finite-energy property of $\omega_h$, we can write
\begin{align*}
	\P[\piv(B_{L}(x))] \leq (c_{{\rm FE}}(1 -\varepsilon))^{-{c'}L^d}\P[\piv(B_{L}(x)), \gamma_{\varepsilon}(y) = 0 \mbox{ for all } y \in B_{L}(x)]\,.
\end{align*}
If we now start on the event on the right-hand side and open the vertices
inside $B_{L}(x)$ one by one in $ \gamma _{\varepsilon}$ until $0$ gets
connected to $\partial B_{R}$, which must happen by the definition of
$\mathrm{Piv}(B_{L}(x))$, then the last vertex to be opened is pivotal for the
resulting configuration. This implies that
\begin{align*}
	\P[\piv(B_{L}(x)), \eta_{\varepsilon}(y) = 0 \mbox{ for all } y \in B_{L}(x)] \leq  \varepsilon^{-{c'}L^d}\sum_{y \in B_{L}(x)}\P[\piv_y]\,.
\end{align*}
Combining all these displays yields \eqref{eq:influence_compare}.
\end{proof}

\bigskip

To conclude, we now give a full derivation of \eqref{eq:OSSSvarbnd_avgpart} for sake of completeness. To this end let us define a \emph{randomized} algorithm $\mathsf T$ which takes $({\rm V}(x): x 
\in B_{R + L})$ and $(\eta_\varepsilon(x): x \in B_{R})$ as inputs, and determines the event $\mathcal E_R$ by revealing ${\rm V}(x), \eta_\varepsilon(x)$ one by one as follows:
\begin{defn}[Algorithm $\mathsf T$]
	Fix a deterministic ordering of the vertices in $B_R$ and choose $j \in \{1, \dots, R\}$ uniformly and independently of $({\rm V}, \eta_{\varepsilon})$. Now set $x_0$ to be the smallest $x \in \partial B_j$ in 
	the ordering and reveal $\eta_{\varepsilon}(x_0)$ as well as 
	${\rm V}(y)$ for all $y \in B_{L}(x_0)$. Set $E_0 \stackrel{\textnormal{def.}}{=}\{ x_0 \}$ and $C_0\stackrel{\textnormal{def.}}{=}\{ x_0\}$ if $\gamma_{\varepsilon}(x_0)=1$ and $C_0\stackrel{\textnormal{def.}}{=}\emptyset$ otherwise. At each step $t \geq 1$, assume that $E_{t-1}$ and $C_{t-1}$ have been defined. 
	Then,
	\begin{itemize}
		\item If the intersection of $B_R$ with the set $(\partial_{\text{out}}C_{t-1} \cup \partial B_j)\setminus E_{t-1}$ is non-empty, let $x$ be the smallest vertex in this intersection and set $x_t \stackrel{\textnormal{def.}}{=} x$, $E_t\stackrel{\textnormal{def.}}{=} E_{t-1} \cup \{x_t \}$. Reveal $\eta_{\varepsilon}(x_t)$ as well as 
	${\rm V}(y)$ for all $y \in B_{L}(x_t)$ and set $C_t\stackrel{\textnormal{def.}}{=} C_{t-1}\cup \{x_t\}$ if $\gamma_{\varepsilon}(x_t)=1$ and $C_t\stackrel{\textnormal{def.}}{=} C_{t-1}$ otherwise.
		\item If the intersection is empty, halt the algorithm.
	\end{itemize}
\end{defn}
In words, the algorithm $\mathsf T$ explores the clusters in $\gamma_{\varepsilon}$ of the vertices in $\partial B_j$.
By applying the OSSS inequality \cite{OSSS} to $1_{\mathcal E_R}$ and the algorithm $\mathsf T$, we now get 
\begin{align}
	\label{eq:OSSSvar_bnd}
	\var[1_{\mathcal E_R}] = \theta_R (1 - \theta_R) \leq \sum_{x \in B_{R + L}}\theta_{{\rm V}(x)}(\mathsf T)\,\Inf_{{\rm V}(x)}  + \sum_{x \in B_{R}}\theta_{{\eta_\varepsilon(x)}}(\mathsf T)\,\Inf_{\eta_\varepsilon(x)},\,
	\end{align}
where the function $\theta_{\cdot}(\mathsf{T})$, called the revealment of the respective variable, is defined as
\begin{align}
	\label{eq:reveal}
	\theta_{{\rm V}(x)}(\mathsf{T}) &\stackrel{\textnormal{def.}}{=} \pi[\mathsf{T} \text{ reveals the value of }{\rm V}(x)],
\end{align}
and $\theta_{\eta_\varepsilon(x)}(\mathsf{T})$ is defined in a similar way. Here $\pi$ denotes the probability governing the extension of $({\rm V}, \eta_{\varepsilon})$ which accommodates the random choice of layer $j$, which is independent of $({\rm V}, \eta_{\varepsilon})$.

\smallskip

Let us now 
bound the revealments for ${\rm V}(x)$ and $\eta_{\varepsilon}(x)$. Notice that since ${\rm V}(x)$ affects the state of vertices only in $B_{L}(x)$ by Property~(d), ${\rm 
V}(x)$ is revealed only if there is an explored vertex $y 
\in B_{L}(x) \cap B_R$. The vertex $y$, on the other hand, is explored 
only if $y$ is connected to $\partial B_j$ in $\gamma_{\varepsilon}$. We deduce that
\begin{equation*}
\label{eq:v_revealment_bnd}
\theta_{{\rm V}(x)}(\mathsf T) \leq \frac{1}{R}\sum_{j=1}^R\P[\lr{}{\gamma_{\varepsilon}}{(B_{L}(x) \cap B_R)}{\partial B_j}] \leq \frac{1}{R}\sum_{j=1}^R\,\sum_{y \in 
 B_{L}(x) \cap B_R}  \P[\lr{}{\gamma_{\varepsilon}}{y}{\partial B_j}]\,,
\end{equation*}
where in the last step we used a naive union bound. For $\eta_{\varepsilon}(x)$, it is even simpler:
\begin{equation*}
\label{eq:revealment_bnd}
\theta_{\eta_{\varepsilon}(x)}(\mathsf T) \leq \frac{1}{R}\sum_{j=1}^R\P[\lr{}{\gamma_{\varepsilon}}{x}{\partial B_j}]\,.
\end{equation*}
Now \eqref{eq:OSSSvarbnd_avgpart} follows by plugging the previous two displays into 
\eqref{eq:OSSSvar_bnd} combined with the translation invariance of $\gamma_{\varepsilon}$ implied by Property~(a) and the triangle inequality.

\subsection*{Acknowledgments.}

We thank Aran Raoufi and Augusto Teixeira for interesting discussions at various stages of this project. We are grateful to Alain-Sol Sznitman for stimulating discussions on the subject of this work and for his comments on an earlier draft of this manuscript. We thank two anonymous referees for their comments on a prior version of this article. We thank Augusto Teixeira for pointing out that our multiscale bridges should really be called \textit{croissants}.

\medskip
This research was supported by an IDEX grant from Paris-Saclay, the European Research Council (ERC) project CriBLaM, a grant from the Swiss National Science Foundation (FNS), and the National Centres for Competence in Research (NCCR) SwissMAP. SG's research was also partially supported by the SERB grant SRG/2021/000032, a grant from the Department of Atomic Energy, Government of India under project 12--R\&D--TFR--5.01--0500 and a grant from the Infosys Foundation as a member of the Infosys-Chandrasekharan virtual center for Random Geometry. PFR thanks the Research Institute for Mathematical Sciences (RIMS) in Kyoto for its hospitality.

\end{document}